\documentclass[11pt]{article}
\usepackage{algpseudocode}  
\usepackage{amsmath}  
\usepackage[colorlinks=true,citecolor=blue]{hyperref}
\usepackage{amsfonts}
\usepackage{amssymb,amsmath}
\usepackage{graphicx}
\usepackage{cite}
\usepackage{color}
\usepackage{amsthm}
\usepackage{wrapfig}
\usepackage[figuresright]{rotating}
 \usepackage{subfig}
 
\usepackage{float}
\numberwithin{equation}{section}

\newtheorem{theorem}{Theorem}[section]
\newtheorem{definition}{Definition}[section]
\newtheorem{lemma}{Lemma}[section]
\newtheorem{proposition}{Proposition}[section]

\newtheorem{Assumption}{Assumption}[section]
\newtheorem{algorithm}{Algorithm}[section]
\allowdisplaybreaks

\begin{document}
\vspace{1.3cm}
\title
{A stochastic two-step inertial Bregman proximal alternating linearized minimization algorithm for nonconvex and nonsmooth problems{\thanks{Supported by Scientific Research Project of Tianjin Municipal Education Commission (2022ZD007).}}}
\author{ {\sc Chenzheng Guo{\thanks{Email: g13526199036@163.com}},
Jing Zhao{\thanks{Corresponding author. Email: zhaojing200103@163.com}},
Qiao-Li Dong{\thanks{Email: dongql@lsec.cc.ac.cn}}
}\\
\small College of Science, Civil Aviation University of China, Tianjin 300300, China\\
}
\date{}
\date{}
\maketitle{}
{\bf Abstract.}
In this paper, for solving a broad class of large-scale nonconvex and nonsmooth optimization problems, we propose a stochastic two-step inertial Bregman proximal alternating linearized minimization (STiBPALM) algorithm  with variance-reduced stochastic gradient estimators. And we show that SAGA and SARAH are variance-reduced gradient estimators. Under expectation conditions with the Kurdyka--{\L}ojasiewicz property and some suitable conditions on the parameters, we obtain that the sequence generated by the proposed algorithm converges to a critical point. And the general convergence rate is also provided. Numerical experiments on sparse nonnegative matrix factorization and blind image-deblurring are presented to demonstrate the performance of the proposed algorithm.


\vskip 0.4 true cm
\noindent AMS Mathematics Subject Classification: 47J06, 49J52, 65K10, 90C26, 90C30.
\section{Introduction}\label{sect1}
\ \par In this paper, we are interested in solving the following composite optimization problem:
\begin{equation}
\label{MP}
\min_{x\in \mathbb R^{l} ,y\in\mathbb R^{m}}   \Phi(x,y)=f(x)+H(x,y)+g(y),
\end{equation}
where $f:\mathbb{R}^l\rightarrow {(-\infty,+\infty]}$, $g:\mathbb{R}^m\rightarrow {(-\infty,+\infty]}$ are  proper lower semicontinuous. $H(x,y)=\frac{1}{n} \sum_{i=1}^{n} H_i(x,y)$ has a finite-sum structure, $H_i:\mathbb{R}^l\times \mathbb{R}^m \rightarrow \mathbb{R}$ is continuously differentiable and $\nabla H_i$ is Lipschitz continuous on bounded subsets. Note that here and throughout the paper, no convexity is imposed on $\Phi$. In practical application, numerous problems can be formulated into the form of \eqref{MP}, such as signal and image processing \cite{MDC,XH}, nonnegative matrix facorization \cite{PT,LS,MH}, blind image-deblurring \cite{TS,MH}, sparse principal component analysis \cite{ALM,HR}, compressed sensing \cite{ABSS,DC} and so on. Here we list two applications of \eqref{MP}, which will also be used in the numerical experiments.
\par (1) Sparse nonnegative matrix factorization (S-NMF). The S-NMF has important applications in image processing (face recognition) and bioinformatics (clustering of gene expressions), see \cite{LS} for details. Given a matrix $A\in \mathbb{R} ^{ l\times m}$ and an integer $r>0$, we want to seek a factorization $A \approx XY$, where $X \in \mathbb{R} ^{ l\times r}$, $Y \in \mathbb{R} ^{r\times m}$ are nonnegative with $r \leq \min\left \{l,m\right \}$ and $X$ is sparse. One way to solve this problem is by finding a solution for the non-negative least squares model given by
\begin{equation}
\label{(1.2)}
\aligned
\underset{X,Y}{\min}  \left \{ \frac{\eta}{2}\left \| A-XY \right \| _{F}^{2} : \ X,Y\ge 0,\ \left \| X_i \right \| _0\le s,\ i=1,2,\dots ,r\right \},
\endaligned
\end{equation}
where $\eta>0$, $X_i$ denotes the $i$th column of $X$, $\left \| X_i \right \| _0$ denotes the number of nonzero elements of the $i$th column of $X$. In this formulation, the sparsity on $X$ is strictly enforced using the nonconvex $l_ 0$ constraint. Let $H(X,Y)=\frac{\eta}{2}\left \| A-XY \right \| _{F}^{2}=\sum_{i=1}^l\frac{\eta}{2}\left \| A_i-X_iY \right \| _{F}^{2}$, $f(X)=\iota_{X\ge0}(X)+\iota_{\left \| X_1\right \| _0\ge s}(X)+\cdots +\iota_{\left \| X_r \right \| _0\ge s}(X)$, $g(Y)=\iota_{Y\ge0}(Y)$, where $A_i$ denotes the $i$th low of $A$, $\iota_C$ is the indicator function on $C$. Then this model \eqref{(1.2)} can be converted to \eqref{MP}. 
\par (2)  Blind image deconvolution (BID). Let $A$ be the observed blurred image, and let $X$ be the unknown sharp image of the same size. Furthermore, let $Y$ denote a small unknown blur kernel, a typical variational formulation of the blind deconvolution problem is given by:
\begin{equation}
\label{(1.3)}
\underset{X,Y}{\min}  \left \{\frac{1}{2} \left \| A-X\odot Y \right \| _{F}^{2}+\eta\sum_{r=1}^{2d} R([D(X)]_r) : \ 0\le X\le 1,\ 0\le Y\le 1,\ \left \| Y \right \| _1\le 1\right \},
\end{equation}
where $\eta>0$, $\odot$ is the two-dimensional convolution operator, $X$ is the image to recover, and $Y$ is the blur kernel to estimate. Here $R(\cdot )$ is an image regularization term, that imposes sparsity on the image gradient and hence favors sharp images.  $D(\cdot )$ is the differential operator, computing the horizontal and vertical gradients for each pixel. This model \eqref{(1.3)} can be converted to \eqref{MP}, where $H(X,Y)=\frac{1}{2}\left \| A-X\odot Y \right \| _{F}^{2}+\eta\sum_{r=1}^{2d} R([D(X)]_r)$, $f(X)=\iota_{0\le X\le 1}(X)$, $g(Y)=\iota_{\left \| Y \right \| _1\le 1}(Y)+\iota_{0\le Y\le 1}(Y)$. See \cite{TS} for details.
\par For solving problem \eqref{MP}, a frequently applied algorithm is the following proximal alternating linearized minimization algorithm (PALM) by Bolte et al. \cite{JSM} based on results in \cite{ABR,AI}:
\begin{equation}
\label{PALM}
\begin{cases}
x_{k+1}\in \arg\min_{ x\in \mathbb{R}^l}\{f(x)+\langle x,\nabla_xH(x_k,y_k)\rangle+\frac{1}{2\lambda_k}\|x-x_k\|^2_2\},\\
y_{k+1}\in \arg\min_{y\in \mathbb{R}^m}\{g(y)+\langle y,\nabla_yH(x_{k+1},y_k)\rangle+\frac{1}{2\mu_k}\|y-y_k\|^2_2\},\\
\end{cases}
\end{equation}
where $\{\lambda_k\}_{k\in\mathbb{N}}$ and $\{\mu_k\}_{k\in\mathbb{N}}$ are positive sequences. To further improve the performance of PALM, {Pock and Sabach \cite{TS}} introduced an inertial step to PALM, and proposed the following inertial proximal alternating linearized minimization (iPALM) algorithm: 
\begin{equation}
\label{iPALM}
\begin{cases}
u_{1k}=x_k+\alpha _{1k}(x_k-x_{k-1}), v_{1k}=x_k+\beta _{1k}(x_k-x_{k-1}),\\
x_{k+1}\in \arg\min_{ x\in \mathbb{R}^l}\{f(x)+\langle x,\nabla_xH(v_{1k},y_k)\rangle+\frac{1}{2\lambda_k}\|x-u_{1k}\|^2_2\},\\
u_{2k}=y_k+\alpha _{2k}(y_k-y_{k-1}), v_{2k}=y_k+\beta _{2k}(y_k-y_{k-1}),\\
y_{k+1}\in \arg\min_{y\in \mathbb{R}^m}\{g(y)+\langle y,\nabla_yH(x_{k+1},v_{2k})\rangle+\frac{1}{2\mu_k}\|y-u_{2k}\|^2_2\},\\
\end{cases}
\end{equation}
where $\alpha _{1k},\alpha _{2k},\beta _{1k},\beta _{2k}\in \left [ 0,1 \right ] $. Then {Gao} et al. \cite{GCH} presented a Gauss--Seidel type inertial proximal alternating linearized minimization (GiPALM) algorithm, in which the inertial step is performed whenever the $x$ or $y$-subproblem is updated. In order to use the existing information as much as possible to further improve the numerical performance, {Wang} et al. \cite{XX} proposed a new inertial version of proximal alternating linearized minimization (NiPALM) algorithm, which inherits both advantages of iPALM and GiPALM. 
\par Bregman distance regularization is an effective way to improve the numerical results of the algorithm. In \cite{ZQ}, the authors constructed the following two-step inertial Bregman alternating minimization (TiBAM) algorithm using the information of the previous three iterates:
\begin{equation}
\label{TiBAM}
\begin{cases}
x_{k+1}\in \arg\min_{ x\in \mathbb{R}^l}\{\Phi(x,y_k)+D_{\phi_1}(x,x_k)+\alpha_{1k} \langle x,x_{k-1}-x_k\rangle+\alpha_{2k} \langle x,x_{k-2}-x_{k-1}\rangle\},\\
y_{k+1}\in \arg\min_{y\in \mathbb{R}^m}\{\Phi(x_{k+1},y)+D_{\phi_2}(y,y_k)+\beta_{1k} \langle y,y_{k-1}-y_k\rangle+\beta_{2k} \langle y,y_{k-2}-y_{k-1}\},
\end{cases}
\end{equation}
where $D_{\phi_i}(i=1,2)$ denotes the Bregman distance with respect to  $\phi_i(i=1,2)$. By linearizing $H(x,y)$ in TiBAM algorithm, the authors \cite{GZ} proposed the following two-step inertial Bregman proximal alternating linearized minimization (TiBPALM) algorithm:
\begin{equation}
\label{TiBPALM}
\begin{cases}
\aligned
x_{k+1}\in \arg\min_{ x\in \mathbb{R}^l}\{&f(x)+\langle x,\nabla_xH(x_k,y_k)\rangle+D_{\phi_1}(x,x_k)+\alpha_{1k} \langle x,x_{k-1}-x_k\rangle\\
&+\alpha_{2k} \langle x,x_{k-2}-x_{k-1}\rangle\},\\
y_{k+1}\in \arg\min_{y\in \mathbb{R}^m}\{&g(y)+\langle y,\nabla_yH(x_{k+1},y_k)\rangle+D_{\phi_2}(y,y_k)+\beta_{1k} \langle y,y_{k-1}-y_k\rangle\\
&+\beta_{2k} \langle y,y_{k-2}-y_{k-1}\rangle\}.
\endaligned
\end{cases}
\end{equation}
If we take $\phi_1(x)=\frac{1}{2\lambda}\|x\|^2_2$ and $\phi_2(y)=\frac{1}{2\mu}\|y\|^2_2$ for all $x\in \mathbb{R}^l$ and $y\in \mathbb{R}^m$, then \eqref{TiBPALM} becomes two-step inertial proximal alternating linearized minimization (TiPALM) algorithm. Then based on alternating minimization algorithm, {Chao} et al. \cite{Z} proposed inertial alternating minimization with Bregman distance (BIAM) algorithm. Other related work can be found in \cite{MP,ML} and their references.
\par It should be noted that all these works are obtained for deterministic methods, i.e. no randomness involved. But when the dimension of data is very large, the computing cost of the full gradient of the function $H(x,y)$ is often prohibitively expensive. In order to overcome this difficulty, stochastic gradient approximations were applied, see, e.g. \cite{B} and the references therein. A block stochastic gradient iteration combining simple stochastic gradient descent (SGD) estimator with PALM  was first proposed by Xu and Yin \cite{XW}. 
To weaken the assumptions on the objective function in \cite{XW} and improve the estimates on the convergence rate of a stochastic PALM algorithm, Driggs et al. \cite{DT} used more sophisticated so-called variance-reduced gradient estimators instead of the simple stochastic gradient descent estimators, and proposed the following stochastic proximal alternating linearized minimization (SPRING) algorithm:
\begin{equation}
\label{SPRING}
\begin{cases}
x_{k+1}\in \arg\min_{ x\in \mathbb{R}^l}\{f(x)+\langle x,\widetilde{\nabla}_x(x_k,y_k)\rangle+\frac{1}{2\lambda_k}\|x-x_k\|^2_2\},\\
y_{k+1}\in \arg\min_{y\in \mathbb{R}^m}\{g(y)+\langle y,\widetilde{\nabla}_y(x_{k+1},y_k)\rangle+\frac{1}{2\mu_k}\|y-y_k\|^2_2\}.\\
\end{cases}
\end{equation}
The key of SPRING algorithm is replacing the full gradient computations $\nabla_x H(x_k,y_k)$ and $\nabla_yH(x_{k+1},y_k)$ with stochastic estimations  $\widetilde{\nabla}_x(x_k,y_k)$ and $\widetilde{\nabla}_y(x_{k+1},y_k)$, respectively. Then Hertrich et al. \cite{JG} introduced the following inertial variant of a stochastic PALM algorithm with variance-reduced gradient estimator, called SiPALM: 
\begin{equation}
\label{SiPALM}
\begin{cases}
u_{1k}=x_k+\alpha _{1k}(x_k-x_{k-1}), v_{1k}=x_k+\beta _{1k}(x_k-x_{k-1}),\\
x_{k+1}\in \arg\min_{ x\in \mathbb{R}^l}\{f(x)+\langle x,\widetilde{\nabla}_x(v_{1k},y_k)\rangle+\frac{1}{2\lambda_k}\|x-u_{1k}\|^2_2\},\\
u_{2k}=y_k+\alpha _{2k}(y_k-y_{k-1}), v_{2k}=y_k+\beta _{2k}(y_k-y_{k-1}),\\
y_{k+1}\in \arg\min_{y\in \mathbb{R}^m}\{g(y)+\langle y,\widetilde{\nabla}_y(x_{k+1},v_{2k})\rangle+\frac{1}{2\mu_k}\|y-u_{2k}\|^2_2\},\\
\end{cases}
\end{equation}
where $\alpha _{1k},\alpha _{2k},\beta _{1k},\beta _{2k}\in \left [ 0,1 \right ] $. Also, some variance-reduced gradient estimators are proposed to solve the nonconvex
optimization problem. The classical stochastic gradient direction is modified in various ways so as to drive the variance of the gradient estimator towards zero.
Such as SAG \cite{SR}, SVRG \cite{KL,JZ}, SAGA \cite{AFS} and SARAH \cite{LM,NL}. 
\par In this paper, we combine inertial technique, Bregman distance and stochastic gradient estimators to develop a stochastic two-step inertial Bregman proximal alternating linearized minimization (STiBPALM) algorithm to solve the nonconvex optimization problem \eqref{MP}. Our contributions are listed as follows.

(1) We propose the STiBPALM algorithm with variance-reduced stochastic gradient estimators to solve the nonconvex optimization problem \eqref{MP}. And we show that SAGA and SARAH are variance-reduced gradient estimators (Definition \ref{Def211}) in the appendix.

(2) We provide theoretical analysis to show that the proposed algorithm with the variance-reduced stochastic gradient estimator has global convergence under expectation conditions. 
Under the expectation version of Kurdyka--{\L}ojasiewicz (K{\L}) property, the sequence generated by the proposed algorithm converges to a critical point and the general convergence rate is also obtained.

(3) We use several well studied stochastic gradient estimators (e.g. SGD, SAGA and SARAH) to test the performance of STiBPALM for sparse nonnegative matrix factorization and blind image-deblurring problems. And comparing with some existing algorithms (e.g. PALM, iPALM, SPRING and SiPALM) in the literature, we report some preliminary numerical results to demonstrate the effectiveness of the proposed algorithm.
\par This paper is organized as follows. In Section \ref{sect2}, we recall some concepts and important lemmas which will be used in the proof of main results. Section \ref{sect0} introduces our STiBPALM algorithm in detail. 
We discuss the convergence behavior of STiBPALM in Section \ref{sect4}. In Section \ref{sect5}, we perform some numerical experiments and compare the results with other algorithms. We give the specific theoretical analysis to show that SAGA and SARAH have variance-reduced stochastic gradient estimators in the appendix.
\section{Preliminaries}\label{sect2}
\ \par In this section, we summarize some useful definitions and lemmas.
\begin{definition}
\rm
\label{Def25}(Kurdyka--{\L}ojasiewicz property \cite{ABR})
Let $F: \mathbb{R}^d \rightarrow(-\infty,+\infty]$ be a proper and lower semicontinuous function.

(i) The function  $F: \mathbb{R}^d \rightarrow(-\infty,+\infty]$  is said to have the Kurdyka--{\L}ojasiewicz (K{\L}) property at $x^\ast\in$dom$F$ if there exist $\eta\in (0,+\infty]$, a neighborhood $U$ of $x^\ast$ and a continuous concave function $\varphi:[0,\eta)\rightarrow \mathbb{R}_{+}$ such that $\varphi(0)=0$, $\varphi$ is $C^1$ on $(0,\eta)$, for all $s\in(0,\eta)$ it is $\varphi'(s)>0$ and for all $x$ in $U\cap[F(x^\ast)<F<F(x^\ast)+\eta]$ the Kurdyka--{\L}ojasiewicz  inequality holds, $$\varphi'(F(x)-F(x^\ast)){\rm dist}(0,\partial F(x))\geq 1.$$

(ii) Proper lower semicontinuous functions which satisfy the Kurdyka--{\L}ojasiewicz inequality at each point of {its domain} are called Kurdyka--{\L}ojasiewicz (K{\L}) functions.
\end{definition}
Roughly speaking, K{\L} functions become sharp up to reparameterization via $\varphi$, a desingularizing function for $F$. Typical K{\L} functions include the class of semialgebraic functions \cite{JAA,JAO}. For instance, the $l _0$ pseudonorm and the rank function are K{\L}. Semialgebraic functions admit desingularizing functions of the form $\varphi (r)=ar^{1-\vartheta } $ for $a > 0$, and $\vartheta \in [0, 1)$ is known as the K{\L} exponent of the function \cite{JAA,JSM}. For these functions, the K{\L} inequality reads
\begin{equation}
\label{(2.011)}
(F(x)-F(x^\ast))^\vartheta \le C\left \| \xi  \right \| ,\ \ \forall \xi \in \partial F(x)
\end{equation}
for some $C>0$.
\begin{definition}
\rm
{ A function $F$ is said convex if dom$F$ is a convex set and if, for all $x$, $y\in$dom$F$, $\alpha\in[0,1]$,
$$F(\alpha x+(1-\alpha)y)\leq \alpha F(x)+(1-\alpha)F(y).$$
$F$ is said $\theta$-strongly convex with $\theta> 0$ if $F-\frac{\theta}{2}\|\cdot\|^2$ is convex, i.e.,
$$F(\alpha x+(1-\alpha)y)\leq \alpha F(x)+(1-\alpha)F(y)-\frac{1}{2}\theta\alpha(1-\alpha)\|x-y\|^2$$
for all $x$, $y\in$dom$F$ and  $\alpha\in[0,1]$.}
\end{definition}
Suppose that the function $F$ is differentiable. Then $F$ is convex if and only if dom$F$ is a convex set and
$$F(x)\ge F(y)+\langle \nabla F(y),x-y\rangle$$
holds for all $x$, $y\in$dom$F$. Moreover, $F$ is $\theta$-strongly convex with $\theta> 0$ if and only if
$$F(x)\ge F(y)+\langle \nabla F(y),x-y\rangle+\frac{\theta}{2}\|x-y\|^2$$
for all $x$, $y\in$dom$F$.
\begin{definition}
\rm
\label{Defbreg}
Let $\phi:\mathbb{R}^d \rightarrow(-\infty,+\infty]$ be a convex and G\^{a}teaux differentiable function.
The function $D_\phi :$ dom$\phi\,\,\times$ intdom$\phi \rightarrow [0,+\infty)$, defined by
$$D_\phi(x,y)=\phi(x)-\phi(y)-\langle \nabla\phi(y),x-y\rangle,$$
is called the Bregman distance with respect to $\phi$.
\end{definition}
{From the above definition,  it follows  that
\begin{equation}
\label{(2.4)}
D_\phi(x,y)\geq\frac{\theta}{2}\|x-y\|^2,
\end{equation}
if $\phi$ is $\theta$-strongly convex. 
\begin{lemma}{\rm (Descent lemma\cite{BT})}
\label{lem20}
Let $F: \mathbb{R}^{d}\rightarrow \mathbb{R}$ be a continuously differentiable function with gradient $\nabla F$ assumed $L$-Lipschitz continuous. Then
\begin{equation}
\label{(2.0)}
\left | F(y)-F(x)-\left \langle y-x,\nabla F(x) \right \rangle  \right | \le \frac{L }{2}\left \| x-y \right \| ^{2},\ \forall x,y\in \mathbb R^{d}. 
\end{equation}
\end{lemma}
\begin{lemma}
\label{lem21}
Let $F:\mathbb R^{d}\to  \mathbb R$ be a function with $L$-Lipschitz continuous gradient, $G:\mathbb R^{d}\to  \mathbb R$ a proper lower semicontinuous function, and $z\in \arg\min_{v\in \mathbb{R}^d}\{G(v)+\langle d,v-x\rangle+D_{\phi}(v,x)+\gamma\langle v,u\rangle+\mu \langle v,w\rangle\}$, where $D_{\phi}$ denotes the Bregman distance with respect to $\phi$, and $x$, $d$, $u$, $w\in \mathbb R^{d}$. Then, for all $y\in \mathbb R^{d}$,

\begin{equation}
\label{(2.5)}
\aligned
F(z)+G(z)\le &F(y)+G(y)+\left \langle \nabla F(x)-d,z-y \right \rangle +\frac{L}{2} \left \| x-y \right \|^2+D_{\phi}(y,x)\\
&+\frac{L}{2} \left \| z-x \right \|^2-D_{\phi}(z,x)+\gamma\langle y-z,u\rangle+\mu \langle y-z,w\rangle.
\endaligned
\end{equation}
\begin{proof}
By Lemma \ref{lem20}, we have the inequalities
\begin{equation*}
\aligned
&F(x)-F(y)\le \left \langle \nabla F(x), x-y\right \rangle+\frac{L}{2} \left \| x-y \right \|^2,\\
&F(z)-F(x)\le \left \langle \nabla F(x), z-x\right \rangle+\frac{L}{2} \left \| z-x \right \|^2,
\endaligned
\end{equation*}
which implies that\\
\begin{equation}
\label{(2.6)}
F(z)\le F(y)+\left \langle \nabla F(x), z-y\right \rangle+\frac{L}{2} \left \| x-y \right \|^2+\frac{L}{2} \left \| z-x \right \|^2.
\end{equation}
Furthermore, by the definition of $z$, taking $v=y$, we obtain
\begin{equation*}
\aligned
&G(z)+\langle d,z-x\rangle+D_{\phi}(z,x)+\gamma\langle z,u\rangle+\mu \langle z,w\rangle \\
\le &G(y)+\langle d,y-x\rangle+D_{\phi}(y,x)+\gamma\langle y,u\rangle+\mu \langle y,w\rangle,
\endaligned
\end{equation*}
which implies that\\
\begin{equation}
\label{(2.7)}
\aligned
G(z)\le G(y)+\langle d,y-z\rangle+D_{\phi}(y,x)-D_{\phi}(z,x)+\gamma\langle y-z,u\rangle+\mu \langle y-z,w\rangle.
\endaligned
\end{equation}
Adding \eqref{(2.6)} and \eqref{(2.7)} completes the proof.
\end{proof}
\end{lemma}
\begin{lemma}{\rm (sufficient decrease property)}
\label{lem22}
Let $F$, $G$, and $z$ be defined as in Lemma \ref{lem21}, where $x$, $d$, $u$, $w\in \mathbb R^{d}$. Assum that $\phi$ is $\theta$-strongly convex. Then the following inequality holds, for any $\lambda>0$,
\begin{equation}
\label{(2.8)}
\aligned
F(z)+G(z)\le &F(x)+G(x)+\frac{1}{2L\lambda }\left \| d-\nabla F(x) \right \| ^2 +\frac{L(\lambda+1) -\theta }{2} \left \| x-z \right \|^2\\
&+\gamma\langle x-z,u\rangle+\mu \langle x-z,w\rangle.
\endaligned
\end{equation}
\begin{proof}
From Lemma \ref{lem21} with $y=x$, we have
\begin{equation*}
\aligned
F(z)+G(z)\le &F(x)+G(x)+\left \langle \nabla F(x)-d,z-x \right \rangle+\frac{L}{2} \left \| x-z \right \|^2\\
&-D_{\phi}(z,x)+\gamma\langle x-z,u\rangle+\mu \langle x-z,w\rangle.
\endaligned
\end{equation*}
Using Young's inequality $\left \langle \nabla F(x)-d,z-x \right \rangle\le \frac{1}{2L\lambda }\left \| d-\nabla F(x) \right \| ^2 +\frac{L\lambda}{2} \left \| x-z \right \|^2$ and \eqref{(2.4)} we can obtain
\begin{equation*}
\aligned
F(z)+G(z)\le &F(x)+G(x)+\frac{1}{2L\lambda }\left \| d-\nabla F(x) \right \| ^2 +\frac{L\lambda}{2} \left \| x-z \right \|^2+\frac{L}{2} \left \| x-z \right \|^2\\
&-\frac{\theta}{2} \left \| z-x \right \|^2+\gamma\langle x-z,u\rangle+\mu \langle x-z,w\rangle,
\endaligned
\end{equation*}
which can be abbreviated as the desired result.
\end{proof}
\end{lemma}

\section{Stochastic two-step inertial Bregman proximal alternating linearized minimization algorithm}\label{sect0}
Throughout this paper, we impose the following assumptions.
\begin{Assumption}
\label{Assumption21}
\rm
(i) The function $\Phi$ is bounded from below, i.e., $\Phi(x,y)\ge \underline{\Phi}.$

(ii) For any fixed $y$, 
the partial gradient $\nabla_{x} H_i(\cdot,y)$ is globally Lipschitz with module $L_y$ for all $i\in \left \{ 1,\dots ,n \right \}$, that is,
$$\left \|\nabla_{x} H_i\left ( x_{1}  ,y  \right ) - \nabla_{x} H_i\left ( x_{2},y  \right ) \right \|\le L_y\left \| x_{1}-x_{2}   \right \|, \ \forall x_{1} ,x_{2} \in \mathbb R^{l}.  $$
Likewise, for any fixed $x$, the partial gradient $\nabla_{y} H_i(x,\cdot)$ is globally Lipschitz with module $L_x$,
$$\left \|\nabla_{y} H_i\left ( x,y_{1} \right ) - \nabla_{y} H_i\left ( x,y_{2} \right ) \right \|\le L_x\left \| y_{1}-y_{2}   \right \| , \ \forall y_{1} ,y_{2} \in \mathbb R^{m}.  $$

(iii) $\nabla H$ is Lipschitz continuous on bounded subsets of $\mathbb R^{l}\times \mathbb R^{m}$. In other words, for each bounded subset $B_1\times B_2$ of $\mathbb R^{l}\times \mathbb R^{m}$, there exists $M_{B_1\times B_2} > 0$ such that
$$\left \| \nabla_{x} H\left ( x_{1}  ,y_1  \right ) - \nabla_{x} H\left ( x_{2},y_2 \right )  \right \|\le M_{B_1\times B_2}\left \| \left ( x_{1}-x_{2},y_{1}-y_{2} \right )    \right \|$$
for all $( x_{1}  ,y_1), ( x_{2}  ,y_2)\in B_1\times B_2$.

(iv) $\phi_i(i=1,2)$ is $\theta_i$-strongly convex differentiable function. And the gradient $\nabla \phi_i$ is $\eta_i$-Lipschitz continuous, i.e.,
\begin{equation*}
\aligned
&\left \| \nabla{\phi_1}(x_1) -\nabla{\phi_1}(x_2)\right \|\le \eta_1 \|x_{1}-x_{2}\|,\ \forall x_{1} ,x_{2}\in \mathbb R^{l},\\
&\left \| \nabla{\phi_2}(y_1) -\nabla{\phi_2}(y_2)\right \|\le \eta_2 \|y_{1} -y_{2}\|,\ \  \forall y_{1} ,y_{2}\in \mathbb R^{m}. 
\endaligned
\end{equation*}
\end{Assumption}
\ \par We now introduce a stochastic version of the two-step inertial Bregman proximal alternating linearized minimization algorithm. The key of our algorithm is replacing the full gradient computations $\nabla_x H(u_k,y_k)$ and $\nabla_y(x_{k+1},v_k)$ with stochastic estimations $\widetilde{\nabla}_x(u_k,y_k)$ and $\widetilde{\nabla}_y(x_{k+1},v_k)$, respectively. We describe the resulted algorithm as follows.
\begin{algorithm}\label{alg1}
\rm
Choose  $(x_0,y_0)\in$dom$\Phi$ and set  $(x_{-i},y_{-i})=(x_0,y_0)$, $i=1, 2$. Take the sequences $\{\gamma_{1k}\}$, $\{\mu_{1k}\}\subseteq[0,\gamma_1]$, $\{\gamma_{2k}\}$, $\{\mu_{2k}\}\subseteq[0,\gamma_2]$, $\{\alpha_{1k}\}$, $\{\beta_{1k}\}\subseteq[0,\alpha_1]$ and $\{\alpha_{2k}\}$, $\{\beta_{2k}\}\subseteq[0,\alpha_2]$, where $\gamma_1\geq0$, $\gamma_2\geq0$, $\alpha_1\geq0$ and $\alpha_2\geq0$. For $k\geq 0$, let
\begin{equation}
\label{TiBSPALM}
\begin{cases}
\aligned
u_k=x_k+\gamma_{1k}(x_k&-x_{k-1})+\gamma_{2k}(x_{k-1}-x_{k-2}),\\
x_{k+1}\in \arg\min_{ x\in \mathbb{R}^l}\{&f(x)+\langle x,\widetilde{\nabla}_x(u_k,y_k)\rangle+D_{\phi_1}(x,x_k)+\alpha_{1k} \langle x,x_{k-1}-x_k\rangle\\
&+\alpha_{2k} \langle x,x_{k-2}-x_{k-1}\rangle\},\\
v_k=y_k+\mu_{1k}(y_k&-y_{k-1})+\mu_{2k}(y_{k-1}-y_{k-2}),\\
y_{k+1}\in \arg\min_{y\in \mathbb{R}^m}\{&g(y)+\langle y,\widetilde{\nabla}_y(x_{k+1},v_k)\rangle+D_{\phi_2}(y,y_k)+\beta_{1k} \langle y,y_{k-1}-y_k\rangle\\
&+\beta_{2k} \langle y,y_{k-2}-y_{k-1}\rangle\},
\endaligned
\end{cases}
\end{equation}
where $D_{\phi_1}$ and $D_{\phi_2}$ denote the Bregman distance with respect to  $\phi_1$ and $\phi_2$, respectively.
\end{algorithm}
\par Stochastic gradients $\widetilde{\nabla}_x(u_k,y_k)$ and $\widetilde{\nabla}_y(x_{k+1},v_k)$ use the gradients of only a few indices $\nabla _xH_i(u_k,y_k)$ and $\nabla _yH_i(x_{k+1},v_k)$ for $i \in B_k \subset \left \{ 1,2,\dots , n \right \}$. The minibatch $B_k$ is chosen uniformly at random from all subsets of $\left \{ 1,2,\dots , n \right \}$ with cardinality $b$. The simplest one is the stochastic gradient descent (SGD) estimator \cite{RS}. While the SGD estimator is not variance-reduced, many popular gradient estimators as the SAGA \cite{AFS} and SARAH \cite{LM,NL} estimators have this property. In this paper, we mainly consider SAGA (Appendix \ref{71}) and SARAH (Appendix \ref{72}) gradient estimators.
\begin{definition}
\rm
\label{DefA1}(SGD \cite{RS})
The SGD gradient estimator $\widetilde{\nabla}_x^{SGD}(x_k,y_k)$ is defined as follows,
\begin{equation*}
\aligned
\widetilde{\nabla}_x^{SGD}(x_k,y_k)=\frac{1}{b}\sum_{i\in B_k} \nabla _xH_i(x_k,y_k),
\endaligned
\end{equation*}
where $B_k$ are mini-batches containing $b$ indices. 
\end{definition}
The SGD gradient estimator uses the gradient of a randomly sampled batch to represent the full gradient.
\begin{definition}
\rm
\label{DefA1}(SAGA \cite{AFS})
The SAGA gradient estimator $\widetilde{\nabla}_x^{SAGA}(x_k,y_k)$ is defined as follows,
\begin{equation*}
\aligned
\widetilde{\nabla}_x^{SAGA}(x_k,y_k)=\frac{1}{b}\sum_{i\in B_k}\left (  \nabla _xH_i(x_k,y_k)- \nabla _xH_i(\varphi _{k}^{i},y_{k}) \right ) + \frac{1}{n}\sum_{j=1}^n\nabla _xH_j(\varphi _{k}^{j},y_{k}),
\endaligned
\end{equation*}
where $B_k$ are mini-batches containing $b$ indices. The variables $\varphi _{k}^{i}$ follow the update rules $\varphi _{k+1}^{i}=x_k$ if $i\in B_k$ and $\varphi _{k+1}^{i}=\varphi _{k}^{i}$ otherwise.
\end{definition}

\begin{definition}
\rm
\label{DefA2}(SARAH \cite{LM,NL})
The SARAH gradient estimator reads for $k = 0$ as $$\widetilde{\nabla}_x^{SARAH}(x_0,y_0)=\nabla_xH(x_0,y_0).$$ 
For $k = 1, 2,\dots$, we define random variables $p_k\in\left \{ 0,1 \right \}$ with $P(p_k=0)=\frac{1}{p}$ and $P(p_k=1)=1-\frac{1}{p}$, where $p \in(1,\infty   )$ is a fixed chosen parameter. Let $B_k$ be a random subset uniformly drawn from $\left \{ 1,\dots , n \right \}$ of fixed batch size $b$. Then for $k= 1, 2,\dots$, the SARAH gradient estimator reads as
\begin{equation*}
\aligned
&\widetilde{\nabla}_x^{SARAH}(x_{k},y_{k})\\
=&
\begin{cases}
\nabla_xH(x_k,y_k),&\text{ if } p_k=0, \\
\frac{1}{b}\sum_{i\in B_k}\left (  \nabla _xH_i(x_k,y_k)- \nabla _xH_i(x_{k-1},y_{k-1}) \right ) +\widetilde{\nabla}_x^{SARAH}(x_{k-1},y_{k-1}),& \text{ if } p_k=1.
\end{cases}
\endaligned
\end{equation*}
\end{definition}
In our analysis, we assume that stochastic gradient estimator used in Algorithm \ref{alg1} is variance-reduced, which is a quite general assumption in stochastic gradient algorithms \cite{DT,JG}. The following definition is analogous to Definition 2.1 in \cite{DT}.
\begin{definition}
\rm
\label{Def211}(variance-reduced gradient estimator)
Let $\left \{ z_k  \right \} _{k\in \mathbb{N} }=\left \{ (x_k,y_k)\right \} _{k\in \mathbb{N} }$ be the sequence generated by Algorithm \ref{alg1} with some gradient estimator $\widetilde{\nabla}$. This gradient estimator is called variance-reduced with constants $V_1,V_2,V_\Upsilon \ge 0$, and $\rho \in (0,1]$ if it satisfies the following conditions:

(i) (MSE bound) There exists a sequence of random variables $\left \{ \Upsilon _k \right \} _{k\in \mathbb{N} }$ of the form $\Upsilon _k=\sum_{i=1}^{s} (v_{k}^{i} )^2$ for some nonnegative random variables $v_{k}^{i}\in \mathbb{R} $ such that
\begin{equation}
\label{(2.1)}
\aligned
&\mathbb{E}_k\left [  \left \| \widetilde{\nabla}_x(u_k,y_k)-\nabla _xH(u_k,y_k) \right \| ^2+\left \| \widetilde{\nabla}_y(x_{k+1},v_k)-\nabla _yH(x_{k+1},v_k) \right \| ^2\right ]\\
\le &\Upsilon _k+V_1\left (\mathbb{E}_k\left \| z_{k+1}-z_{k} \right \| ^{2}+\left \| z_{k}-z_{k-1} \right \| ^{2} +\left \| z_{k-1}-z_{k-2} \right \| ^{2}+\left \| z_{k-2}-z_{k-3} \right \| ^{2}\right ),
\endaligned
\end{equation}
and, with $\Gamma _k=\sum_{i=1}^{s} v_{k}^{i} $
\begin{equation}
\label{(2.2)}
\aligned
&\mathbb{E}_k\left [  \left \| \widetilde{\nabla}_x(u_k,y_k)-\nabla _xH(u_k,y_k) \right \| +\left \| \widetilde{\nabla}_y(x_{k+1},v_k)-\nabla _yH(x_{k+1},v_k) \right \| \right ]\\
\le& \Gamma_k+V_2\left (\mathbb{E}_k\left \| z_{k+1}-z_{k} \right \|+\left \| z_{k}-z_{k-1} \right \| +\left \| z_{k-1}-z_{k-2} \right \|+\left \| z_{k-2}-z_{k-3} \right \|\right ).
\endaligned
\end{equation}

(ii) (Geometric decay) The sequence $\left \{ \Upsilon _k \right \} _{k\in \mathbb{N} }$ decays geometrically:
\begin{equation}
\label{(2.3)}
\aligned
\mathbb{E}_k\Upsilon _{k+1}\le &(1-\rho )\Upsilon _k+V_\Upsilon \left (\mathbb{E}_k\left \| z_{k+1}-z_{k} \right \| ^{2}+\left \| z_{k}-z_{k-1} \right \| ^{2}+\left \| z_{k-1}-z_{k-2} \right \| ^{2}\right.\\
&\left.+\left \| z_{k-2}-z_{k-3} \right \| ^{2}\right ).
\endaligned
\end{equation}

(iii) (Convergence of estimator) If $\left \{ z_k  \right \} _{k\in \mathbb{N} }$ satisfies $\lim_{k \to \infty } \mathbb{E}\left \| z_{k}-z_{k-1} \right \| ^{2}=0$, then $ \mathbb{E}\Upsilon _k\to 0$ and $\mathbb{E}\Gamma _k\to 0$.
\end{definition}
In the following, if $\left \{ z_k  \right \} _{k\in \mathbb{N} }=\left \{ (x_k,y_k)\right \} _{k\in \mathbb{N} }$ be the bounded sequence generated by Algorithm \ref{alg1}, we assume $\nabla H$ is $M$-Lipschitz continuous on $\left \{ (x_k,y_k)\right \} _{k\in \mathbb{N} }$.
\begin{Assumption}
\label{Assumption41}
\rm
For the sequences $\left \{ x_k \right \}_{k\in \mathbb{N} } $ and  $\left \{ y_k \right \}_{k\in \mathbb{N} } $ generated by Algorithm \ref{alg1}, there exist $L> 0$ such that
$$\sup\left \{ L _{y_k}:k\in \mathbb N  \right \} \le L\ \ {\rm and}\ \sup\left \{  L _{x_k}:k\in \mathbb N  \right \} \le L, $$
where $L _{y_k}$ and $L _{x_k}$ are the Lipschitz constants for $\nabla_{x} H_i(\cdot,y_k)$ and $\nabla_{y} H_i(x_k,\cdot)$, respectively.
\end{Assumption}

\begin{proposition}
\label{pro1}
If $\left \{ z_k  \right \} _{k\in \mathbb{N} }=\left \{ (x_k,y_k)\right \} _{k\in \mathbb{N} }$ be the bounded sequence generated by Algorithm \ref{alg1}. Then the SAGA gradient estimator is variance-reduced with parameters $V_{1}=\frac{16N^2\gamma^2}{b}$, $V_{2}=\frac{4N\gamma}{\sqrt{b}}$, $V_{\Upsilon}=\frac{408nN^2(1+2\gamma_1^2+\gamma_2^2)}{b^2}$ and $\rho=\frac{b}{2n}$, where $N=\max\left \{M,L \right \} $, $\gamma=\max\left \{ \gamma_1,\gamma_2 \right \} $. The SARAH estimator is variance-reduced with parameters $V_{1}=6\left ( 1-\frac{1}{p}  \right )M^2(1+2\gamma_{1}^2+\gamma_{2}^2)$, $V_{2}=M\sqrt{6(1-\frac{1}{p})(1 +2\gamma_{1}^2+\gamma_{2}^2) }$, $V_{\Upsilon}=6\left ( 1-\frac{1}{p}  \right )M^2(1+2\gamma_{1}^2+\gamma_{2}^2)$ and $\rho= \frac{1}{p}$.
\end{proposition}
See the detailed proof of Proposition \ref{pro1} in Appendix \ref{71} and \ref{72}. And the conclusion that SVRG gradient estimator is variance-reduced can be obtained similarly. 
\par Below, we give the supermartingale convergence theorem that will be applied to obtain almost sure convergence of sequences generated by STiBPALM (Algorithm \ref{alg1}). 
\begin{lemma}{\rm (supermartingale convergence)}
\label{lem23}
Let $\left \{ X_k \right \} _{k\in \mathbb{N} } $ and $\left \{ Y_k \right \}  _{k\in \mathbb{N} } $ be sequences of bounded nonnegative random variables such that $X_k$, $Y_k$ depend only on the first $k$ iterations of Algorithm \ref{alg1}. If
\begin{equation}
\label{(2.9)}
\mathbb{E} _kX_{k+1}+Y_k\le X_k
\end{equation}
for all $k$, then $\sum_{k=0}^{\infty } Y_k<+\infty $ a.s. and $\left \{ X_k \right \}$ converges a.s.
\end{lemma}
\section{Convergence analysis under the KL property}\label{sect4}
\ \par In this section, under Assumption \ref{Assumption21} and \ref{Assumption41}, we prove convergence of the sequence and extend the convergence rates of SPRING to Algorithm \ref{alg1}, for semialgebraic function $\Phi$. Given $k\in \mathbb{N}$, define the quantity
\begin{equation}
\label{(3.1)}
\aligned
\Psi _k= &\Phi (z_{k})+ \frac{1}{L\lambda\rho  } \Upsilon _k+\left (\frac{V_1+V_\Upsilon /\rho }{L\lambda }+\frac{\alpha_1+\alpha_2}{2}+\frac{2L(\gamma_{1}^2+\gamma_{2}^2)}{\lambda }+3Z  \right )\|z_{k}-z_{k-1}\|^2 \\
&+\left ( \frac{V_1+V_\Upsilon /\rho }{L\lambda }+\frac{\alpha_2}{2}+\frac{2L\gamma_{2}^2}{\lambda }+2Z \right )\|z_{k-1}-z_{k-2}\|^2+\left ( \frac{V_1+V_\Upsilon /\rho }{L\lambda } +Z \right ) \|z_{k-2}-z_{k-3}\|^2,\\
\endaligned  
\end{equation}
where $\lambda=\sqrt{\frac{10(V_1+V_\Upsilon /\rho)+4L^2(\gamma_{1}^2+\gamma_{2}^2)}{L^2}}$, $Z=\frac{V_1+V_\Upsilon /\rho }{\sqrt{10(V_1+V_\Upsilon /\rho)+4L^2(\gamma_{1}^2+\gamma_{2}^2)} }+\epsilon>0$, $\epsilon>0$ is small enough. Our first result guarantees that $\Psi _k$ is decreasing in expectation.
\begin{lemma}{\rm ($\mathit{l} _2$ summability)}
\label{lem31}
Suppose Assumption \ref{Assumption21} and \ref{Assumption41} hold. Let $\left \{ z_k \right \}_{k\in \mathbb{N} } $ be the sequence generated by Algorithm \ref{alg1} with variance-reduced gradient estimator, and let 
\begin{equation*}
\theta\overset{\bigtriangleup }{=}\min\left \{ \theta_1,\theta_2 \right \}> L+2\alpha _1+2\alpha _2+2\sqrt{10(V_1+V_\Upsilon /\rho)+4L^2(\gamma_{1}^2+\gamma_{2}^2)}+6\epsilon,  
\end{equation*}
then the following conclusions hold.

{\rm(i)} $\Psi _k$ satisfies
\begin{equation}
\label{(3.2)}
\aligned 
\mathbb{E}_k\left [\Psi _{k+1} +\kappa \left \| z_{k+1}-z_{k} \right \|^2+\epsilon \left \| z_{k}-z_{k-1} \right \|^2+ \epsilon \left \| z_{k-1}-z_{k-2} \right \|^2+ Z \left \| z_{k-2}-z_{k-3} \right \|^2  \right ] \le \Psi _k,
\endaligned  
\end{equation}
where $\kappa=-\frac{L -\theta}{2}-\alpha_1-\alpha_2-\sqrt{10(V_1+V_\Upsilon /\rho)+4L^2(\gamma_{1}^2+\gamma_{2}^2)}-3\epsilon>0$. 

{\rm(ii)} The expectation of the squared distance between the iterates is summable:
$$\sum_{k=0}^{\infty } \mathbb{E} [\left \| x_{k+1}-x_{k} \right \|^2+\left \| y_{k+1}-y_{k} \right \|^2]=\sum_{k=0}^{\infty } \mathbb{E}\left \| z_{k+1}-z_{k} \right \|^2<\infty.$$
\begin{proof}
\par(i) Applying Lemma \ref{lem22} with $F(\cdot )=H(\cdot ,y_k)$, $G(\cdot )=f(\cdot )$, $z=x_{k+1}$, $x= x_k$, $d =\widetilde{\nabla}_x(u_k,y_k)$, $u = x_{k-1}-x_{k}$ and $w = x_{k-2}-x_{k-1}$, for any $\lambda>0$, we have
\begin{equation}
\label{(3.3)}
\aligned
&H(x_{k+1},y_k)+f(x_{k+1})\\
\le &H(x_{k},y_k)+f(x_{k})+\frac{1}{2L\lambda }\left \| \widetilde{\nabla}_x(u_k,y_k)-\nabla_x H(x_k,y_k) \right \| ^2+\frac{L(\lambda+1) -\theta_1 }{2} \left \| x_{k+1}-x_k \right \|^2 \\
&+\alpha _{1k} \langle x_{k+1}-x_k,x_{k}-x_{k-1}\rangle+\alpha _{2k} \langle x_{k+1}-x_k,x_{k-1}-x_{k-2}\rangle\\
\overset{(1)}{\le } &H(x_{k},y_k)+f(x_{k})+\frac{1}{L\lambda }\left \| \widetilde{\nabla}_x(u_k,y_k)-\nabla_x H(u_k,y_k) \right \| ^2+\frac{1}{L\lambda }\left \| \nabla_x H(u_k,y_k)-\nabla_x H(x_k,y_k) \right \| ^2\\
&+\frac{L(\lambda+1) -\theta_1 }{2} \left \| x_{k+1}-x_k \right \|^2 +\frac{\alpha_{1k}}{2} (\|x_{k+1}-x_k\|^2+\|x_k-x_{k-1}\|^2)\\
&+\frac{\alpha_{2k}}{2}(\|x_{k+1}-x_k\|^2+\|x_{k-1}-x_{k-2}\|^2)\\
\overset{(2)}{\le } &H(x_{k},y_k)+f(x_{k})+\frac{1}{L\lambda }\left \| \widetilde{\nabla}_x(u_k,y_k)-\nabla_x H(u_k,y_k) \right \| ^2+\frac{L}{\lambda }\left \| u_k-x_k \right \| ^2\\
&+\left (\frac{L(\lambda+1) -\theta_1 }{2} +\frac{\alpha_{1}+\alpha_{2}}{2} \right ) \left \| x_{k+1}-x_k \right \|^2+\frac{\alpha_{1}}{2} \|x_k-x_{k-1}\|^2+\frac{\alpha_{2}}{2}\|x_{k-1}-x_{k-2}\|^2\\
\le &H(x_{k},y_k)+f(x_{k})+\frac{1}{L\lambda }\left \| \widetilde{\nabla}_x(u_k,y_k)-\nabla_x H(u_k,y_k) \right \| ^2+\left (\frac{2L\gamma_{1k}^2}{\lambda }+\frac{\alpha_{1}}{2}  \right )\left \| x_k-x_{k-1} \right \| ^2 \\
&+\left ( \frac{2L\gamma_{2k}^2}{\lambda }+\frac{\alpha_{2}}{2} \right )\left \| x_{k-1}-x_{k-2} \right \| ^2 +\left (\frac{L(\lambda+1) -\theta_1 }{2}+\frac{\alpha_{1}+\alpha_{2}}{2} \right )\left \| x_{k+1}-x_k \right \|^2.
\endaligned  
\end{equation}
Inequality (1) is the standard inequality $\left \| a-c\right \|^2\le2\left \| a-b\right \|^2+2\left \| b-c\right \|^2$, and (2) use Assumption \ref{Assumption21} (ii) and Assumption \ref{Assumption41}. Analogously, for the updates in $y_k$, we use Lemma \ref{lem22} with $F(\cdot )=H( x_{k+1},\cdot)$, $G(\cdot )=g(\cdot )$, $z=y_{k+1}$, $x= y_k$, $d =\widetilde{\nabla}_y(x_{k+1},v_k)$, $u = y_{k-1}-y_{k}$ and $w = y_{k-2}-y_{k-1}$, we have
\begin{equation}
\label{(3.4)}
\aligned
&H(x_{k+1},y_{k+1})+g(y_{k+1})\\
\le &H(x_{k+1},y_k)+g(y_{k})+\frac{1}{L\lambda }\left \| \widetilde{\nabla}_y(x_{k+1},v_k)-\nabla_y H(x_{k+1},v_k) \right \| ^2+\left (\frac{2L\mu_{1k}^2}{\lambda }+\frac{\alpha_{1}}{2}  \right )\left \| y_k-y_{k-1} \right \| ^2 \\
&+\left ( \frac{2L\mu_{2k}^2}{\lambda }+\frac{\alpha_{2}}{2} \right )\left \| y_{k-1}-y_{k-2} \right \| ^2+\left (\frac{L(\lambda+1) -\theta_2 }{2} +\frac{\alpha_{1}+\alpha_{2}}{2} \right )\left \| y_{k+1}-y_k \right \|^2.
\endaligned  
\end{equation}
Adding \eqref{(3.3)} and \eqref{(3.4)}, we have\\
\begin{equation*}
\aligned
&\Phi (x_{k+1},y_{k+1})\\
\le &\Phi (x_{k},y_k)+\frac{1}{L\lambda }\left ( \left \| \widetilde{\nabla}_x(u_k,y_k)-\nabla_x H(u_k,y_k) \right \| ^2 +\left \| \widetilde{\nabla}_y(x_{k+1},v_k)-\nabla_y H(x_{k+1},v_k) \right \| ^2  \right )\\
&+\left ( \frac{L(\lambda+1) -\theta}{2}+\frac{\alpha_1+\alpha_2}{2}  \right )\|z_{k+1}-z_k\|^2+\left ( \frac{2L\gamma_{1}^2}{\lambda }+\frac{\alpha_1}{2}\right )\|z_{k}-z_{k-1}\|^2\\
&+\left ( \frac{2L\gamma_{2}^2}{\lambda }+\frac{\alpha_2}{2}\right )\|z_{k-1}-z_{k-2}\|^2,
\endaligned 
\end{equation*}
where $\theta=\min\left \{ \theta_1,\theta_2 \right \}$. Applying the conditional expectation operator $\mathbb{E} _k$, we can bound the MSE terms using \eqref{(2.1)}. This gives
\begin{equation}
\label{(3.5)}
\aligned
&\mathbb{E} _k\left [ \Phi (z_{k+1})+\left ( -\frac{L(\lambda+1) -\theta}{2}-\frac{\alpha_1+\alpha_2}{2}-\frac{V_1}{L\lambda }  \right ) \|z_{k+1}-z_k\|^2\right ] \\
\le& \Phi (z_{k})+ \frac{1}{L\lambda } \Upsilon _k+\left ( \frac{V_1}{L\lambda }+\frac{2L\gamma_{1}^2}{\lambda }+\frac{\alpha_1}{2}  \right )\|z_{k}-z_{k-1}\|^2+\left( \frac{V_1}{L\lambda }+  \frac{2L\gamma_{2}^2}{\lambda }+\frac{\alpha_2}{2}\right )\|z_{k-1}-z_{k-2}\|^2\\
&+\frac{V_1}{L\lambda }\|z_{k-2}-z_{k-3}\|^2.
\endaligned  
\end{equation}
Next, we use \eqref{(2.3)} to say that
\begin{equation*}
\aligned
\frac{1}{L\lambda} \Upsilon _k\le \frac{1}{L\lambda\rho  } &\left ( -\mathbb{E}_k\Upsilon _{k+1}+\Upsilon _{k}+V_\Upsilon \left (\mathbb{E}_k\left \| z_{k+1}-z_{k} \right \| ^{2}+\left \| z_{k}-z_{k-1} \right \| ^{2} \right.\right .\\
&\left. \left. +\left \| z_{k-1}-z_{k-2} \right \| ^{2} +\left \| z_{k-2}-z_{k-3} \right \| ^{2}\right ) \right ).
\endaligned  
\end{equation*}
Combining these inequalities, we have
\begin{equation*}
\aligned
&\mathbb{E} _k\left [ \Phi (z_{k+1})+ \frac{1}{L\lambda\rho } \Upsilon _{k+1} +\left ( -\frac{L(\lambda+1) -\theta}{2}-\frac{\alpha_1+\alpha_2}{2} -\frac{V_1+V_\Upsilon /\rho }{L\lambda } \right ) \|z_{k+1}-z_k\|^2\right ] \\
\le& \Phi (z_{k})+ \frac{1}{L\lambda\rho } \Upsilon_k+\left ( \frac{V_1+V_\Upsilon /\rho }{L\lambda }+\frac{2L\gamma_{1}^2}{\lambda }+\frac{\alpha_1}{2}  \right )\|z_{k}-z_{k-1}\|^2\\
&+\left( \frac{V_1+V_\Upsilon /\rho }{L\lambda }+ \frac{2L\gamma_{2}^2}{\lambda }+\frac{\alpha_2}{2}\right )\|z_{k-1}-z_{k-2}\|^2+\frac{V_1+V_\Upsilon /\rho }{L\lambda }\|z_{k-2}-z_{k-3}\|^2.
\endaligned  
\end{equation*}
This is equivalent to
\begin{equation}
\label{(3.6)}
\aligned
&\mathbb{E} _k\left [ \Phi (z_{k+1})+ \frac{1}{L\lambda\rho  } \Upsilon _{k+1}+\left (\frac{V_1+V_\Upsilon /\rho }{L\lambda }+\frac{\alpha_1+\alpha_2}{2}+\frac{2L(\gamma_{1}^2+\gamma_{2}^2)}{\lambda }+3Z  \right ) \|z_{k+1}-z_k\|^2  \right .\\
&\left. +\left ( \frac{V_1+V_\Upsilon /\rho }{L\lambda } +\frac{\alpha_2}{2}+\frac{2L\gamma_{2}^2}{\lambda }+2Z \right ) \|z_{k}-z_{k-1}\|^2+\left ( \frac{V_1+V_\Upsilon /\rho }{L\lambda } +Z \right ) \|z_{k-1}-z_{k-2}\|^2 \right .\\
&\left. +\left ( -\frac{L(\lambda+1) -\theta}{2}- \frac{2(V_1+V_\Upsilon /\rho) }{L\lambda }-\alpha_1-\alpha_2-\frac{2L(\gamma_{1}^2+\gamma_{2}^2)}{\lambda }-3Z\right )\|z_{k+1}-z_k\|^2\right ] \\
\le& \Phi (z_{k})+ \frac{1}{L\lambda\rho  } \Upsilon _k+\left (\frac{V_1+V_\Upsilon /\rho }{L\lambda }+\frac{\alpha_1+\alpha_2}{2}+\frac{2L(\gamma_{1}^2+\gamma_{2}^2)}{\lambda }+3Z  \right )\|z_{k}-z_{k-1}\|^2 \\
&+\left ( \frac{V_1+V_\Upsilon /\rho }{L\lambda }+\frac{\alpha_2}{2}+\frac{2L\gamma_{2}^2}{\lambda }+2Z \right )\|z_{k-1}-z_{k-2}\|^2+\left ( \frac{V_1+V_\Upsilon /\rho }{L\lambda } +Z \right ) \|z_{k-2}-z_{k-3}\|^2\\
&-\left (Z-\frac{V_1+V_\Upsilon /\rho }{L\lambda }\right )\|z_{k}-z_{k-1}\|^2-\left (Z-\frac{V_1+V_\Upsilon /\rho }{L\lambda }\right )\|z_{k-1}-z_{k-2}\|^2-Z\|z_{k-2}-z_{k-3}\|^2.
\endaligned  
\end{equation}
We have 
\begin{equation}
\label{(3.7)}
\aligned
&\mathbb{E} _k\left [\Psi _{k+1}+\left ( -\frac{L(\lambda+1) -\theta}{2}- \frac{2(V_1+V_\Upsilon /\rho) }{L\lambda }-\alpha_1-\alpha_2-\frac{2L(\gamma_{1}^2+\gamma_{2}^2)}{\lambda }-3Z\right )\|z_{k+1}-z_k\|^2\right ] \\
\le& \Psi _k-\left (Z-\frac{V_1+V_\Upsilon /\rho }{L\lambda }\right )\|z_{k}-z_{k-1}\|^2-\left (Z-\frac{V_1+V_\Upsilon /\rho }{L\lambda }\right )\|z_{k-1}-z_{k-2}\|^2-Z\|z_{k-2}-z_{k-3}\|^2.
\endaligned  
\end{equation}
By $\lambda= \sqrt{\frac{10(V_1+V_\Upsilon /\rho)+4L^2(\gamma_{1}^2+\gamma_{2}^2)}{L^2}}$, we have $-\frac{L(\lambda+1) -\theta}{2}- \frac{2(V_1+V_\Upsilon /\rho) }{L\lambda }-\alpha_1-\alpha_2-\frac{2L(\gamma_{1}^2+\gamma_{2}^2)}{\lambda }-3Z=-\frac{L -\theta}{2}-\alpha_1-\alpha_2-\sqrt{10(V_1+V_\Upsilon /\rho)+4L^2(\gamma_{1}^2+\gamma_{2}^2)}-3\epsilon=\kappa$. Hence \eqref{(3.7)} becomes 
\begin{equation}
\label{(3.07)}
\mathbb{E}_k\left [\Psi _{k+1} +\kappa \left \| z_{k+1}-z_{k} \right \|^2+\epsilon \left \| z_{k}-z_{k-1} \right \|^2+ \epsilon \left \| z_{k-1}-z_{k-2} \right \|^2 + Z \left \| z_{k-2}-z_{k-3} \right \|^2\right ] \le \Psi _k. 
\end{equation}
According to $\theta> L+2\alpha _1+2\alpha _2+2\sqrt{10(V_1+V_\Upsilon /\rho)+4L^2(\gamma_{1}^2+\gamma_{2}^2)}+6\epsilon$, we have $\kappa>0$. So we prove the first claim.
\par(ii) We apply the full expectation operator to \eqref{(3.07)} and sum the resulting inequality from $k=0$ to $k=T-1$,
\begin{equation*}
\aligned
&\mathbb{E}\Psi _{T}+\kappa\sum_{k=0}^{T-1}  \mathbb{E}\left \| z_{k+1}-z_{k} \right \|^2+\epsilon\sum_{k=0}^{T-1}\mathbb{E} \left \| z_{k}-z_{k-1} \right \|^2+ \epsilon \sum_{k=0}^{T-1}\mathbb{E}\left \| z_{k-1}-z_{k-2} \right \|^2\\
&+ Z \sum_{k=0}^{T-1}\mathbb{E}\left \| z_{k-2}-z_{k-3} \right \|^2\\
\le& \Psi _0, 
\endaligned  
\end{equation*}
Using the facts that $\underline{\Phi } \le \Psi_T$,
\begin{equation}
\label{(3.8)}
\aligned
&\kappa\sum_{k=0}^{T-1}  \mathbb{E}\left \| z_{k+1}-z_{k} \right \|^2+\epsilon\sum_{k=0}^{T-1}\mathbb{E} \left \| z_{k}-z_{k-1} \right \|^2+ \epsilon \sum_{k=0}^{T-1}\mathbb{E}\left \| z_{k-1}-z_{k-2} \right \|^2\\
&+ Z \sum_{k=0}^{T-1}\mathbb{E}\left \| z_{k-2}-z_{k-3} \right \|^2\\
\le& \Psi _0-\underline{\Phi }.
\endaligned  
\end{equation}
Taking the limit $T \rightarrow +\infty$, we have the sequence $\left \{ \mathbb{E}\left \| z_{k+1}-z_{k} \right \|^2 \right \}$ is summable.
\end{proof}
\end{lemma}

\par The next lemma establishes a bound on the norm of the subgradients of $\Phi(z_k)$.

\begin{lemma}{\rm (subgradient bound)}
\label{lem32}
Suppose Assumption \ref{Assumption21} and \ref{Assumption41} hold. Let $\{z_k\}_{k\in \mathbb{N}}$ be a bounded sequence, which is generated by Algorithm \ref{alg1} with variance-reduced gradient estimator. For $k\geq 0$, define
\begin{equation*}
\aligned
A_{x}^{k} =&\nabla _xH(x_{k},y_{k})-\widetilde{\nabla} _x(u_{k-1},y_{k-1})+\nabla\phi_1(x_{k-1})- \nabla\phi_1(x_{k})+\alpha_{1,k-1}(x_{k-1}-x_{k-2})\\
&+\alpha_{2,k-1}(x_{k-2}-x_{k-3}),\\
A_{y}^{k} =&\nabla _yH(x_{k},y_{k})-\widetilde{\nabla} _y(x_{k},v_{k-1})+\nabla\phi_2(y_{k-1})- \nabla\phi_2(y_{k})+\beta_{1,k-1}(y_{k-1}-y_{k-2})\\
&+\beta_{2,k-1}(y_{k-2}-y_{k-3}).
\endaligned
\end{equation*}
Then $(A_{x}^{k},A_{y}^{k} )\in \partial \Phi(x_k,y_k)$ and
\begin{equation}
\label{(3.9)}
\aligned
&\mathbb{E}_{k-1}\left \| (A_{x}^{k},A_{y}^{k} ) \right \|\\
\le& p\left (\mathbb{E}_{k-1}\left \| z_{k}-z_{k-1} \right \|+\left \| z_{k-1}-z_{k-2} \right \| +\left \| z_{k-2}-z_{k-3} \right \|+\left \| z_{k-3}-z_{k-4} \right \|\right )+\Gamma_{k-1},
\endaligned
\end{equation}
where $p=2(2N+\eta+N\gamma_1+N\gamma_2+\alpha _{1}+\alpha _{2})+V_2$, $N=\max\left \{ M,L \right \} $, $\eta =\max\left \{ \eta _1,\eta _2 \right \} $.
\begin{proof}
By the definition of $x_{k}$, we have that $0$ must lie in the subdifferential at point $x_{k}$ of the function
 $$x\longmapsto f(x)+\langle x,\widetilde{\nabla}_x(u_{k-1},y_{k-1})\rangle+D_{\phi_1}(x,x_{k-1})+\alpha_{1,k-1} \langle x,x_{k-2}-x_{k-1}\rangle+\alpha_{2,k-1} \langle x,x_{k-3}-x_{k-2}\rangle.$$
 Since $\phi$ are differential, we have
$$0\in \partial f(x_{k})+\widetilde{\nabla}_x(u_{k-1},y_{k-1})+\nabla\phi_1(x_{k})- \nabla\phi_1(x_{k-1})+\alpha_{1,k-1}(x_{k-2}-x_{k-1})+\alpha_{2,k-1}(x_{k-3}-x_{k-2}),$$
which implies that
\begin{equation}
\label{(3.10)}
\aligned
&\nabla _xH(x_{k},y_{k})-\widetilde{\nabla}_x(u_{k-1},y_{k-1})+\nabla\phi_1(x_{k-1})-\nabla\phi_1(x_{k}) \\
&+\alpha_{1,k-1}(x_{k-1}-x_{k-2})+\alpha_{2,k-1}(x_{k-2}-x_{k-3})\\
&\in  \nabla _xH(x_{k},y_{k})+\partial f(x_{k}).
\endaligned  
\end{equation}
Similarly, we have
\begin{equation}
\label{(3.11)}
\aligned
&\nabla _yH(x_{k},y_{k})-\widetilde{\nabla} _y(x_{k},v_{k-1})+\nabla\phi_2(y_{k-1})- \nabla\phi_2(y_{k})\\
&+\beta_{1,k-1}(y_{k-1}-y_{k-2})+\beta_{2,k-1}(y_{k-2}-y_{k-3})\\
&\in  \nabla _yH(x_{k},y_{k})+\partial g(y_{k}).
\endaligned  
\end{equation}
Because of the structure of $\Phi$, from \eqref{(3.10)} and \eqref{(3.11)}, we have
$(A_{x}^{k},A_{y}^{k} )\in \partial \Phi(x_k,y_k).$
 All that remains is to bound the norms of $A_{x}^{k}$ and $A_{y}^{k}$. Because $\nabla H$ is $M$-Lipschitz continuous on bounded sets, then from Assumption \ref{Assumption21} (iii) and (iv), we have\\
\begin{equation}
\label{(3.12)}
\aligned
&\left \| A_{x}^{k} \right \|\\
\le& \left \| \nabla _xH(x_{k},y_{k})-\widetilde{\nabla}_x(u_{k-1},y_{k-1}) \right \|+\left \| \nabla\phi_1(x_{k-1})-\nabla\phi_1(x_{k})\right \| \\
&+\alpha_{1,k-1}\left \| x_{k-1}-x_{k-2}\right \| +\alpha_{2,k-1}\left \| x_{k-2}-x_{k-3}\right \|  \\
\le&\left \| \nabla _xH(x_{k},y_{k})-\nabla _xH(u_{k-1},y_{k-1}) \right \|+\left \| \nabla _xH(u_{k-1},y_{k-1})-\widetilde{\nabla}_x(u_{k-1},y_{k-1}) \right \| \\
&+\eta _1\left \| x_{k-1}-x_{k}\right \|+\alpha_{1,k-1}\left \| x_{k-1}-x_{k-2}\right \| +\alpha_{2,k-1}\left \| x_{k-2}-x_{k-3}\right \|  \\
\le&\left \| \nabla _xH(u_{k-1},y_{k-1})-\widetilde{\nabla}_x(u_{k-1},y_{k-1}) \right \|+M\left \| x_{k}-u_{k-1}\right \|+M\left \| y_{k}-y_{k-1}\right \|\\
&+\eta _1\left \| x_{k-1}-x_{k}\right \|+\alpha_{1,k-1}\left \| x_{k-1}-x_{k-2}\right \| +\alpha_{2,k-1}\left \| x_{k-2}-x_{k-3}\right \|\\
\le&\left \| \nabla _xH(u_{k-1},y_{k-1})-\widetilde{\nabla}_x(u_{k-1},y_{k-1}) \right \|+(M+\eta _1)\left \| x_{k}-x_{k-1}\right \|+M\left \| y_{k}-y_{k-1}\right \|\\
&(M\gamma_{1}+\alpha_{1})\left \| x_{k-1}-x_{k-2}\right \| +(M\gamma_{2}+\alpha_{2})\left \| x_{k-2}-x_{k-3}\right \|.
\endaligned  
\end{equation}
A similar argument holds for $A_{y}^{k}$: 
\begin{equation}
\label{(3.13)}
\aligned
&\left \| A_{y}^{k} \right \|\\
\le& \left \| \nabla _yH(x_{k},y_{k})-\nabla _yH(x_{k},v_{k-1}) \right \|+\left \| \nabla _yH(x_{k},v_{k-1})-\widetilde{\nabla}_y(x_{k},v_{k-1}) \right \| \\
&+\eta _2\left \| y_{k-1}-y_{k}\right \|+\beta_{1,k-1}\left \| y_{k-1}-y_{k-2}\right \| +\beta_{2,k-1}\left \| y_{k-2}-y_{k-3}\right \|  \\
\le&\left \| \nabla _yH(x_{k},v_{k-1})-\widetilde{\nabla}_y(x_{k},v_{k-1}) \right \|+(L+\eta _2)\left \| y_{k}-y_{k-1}\right \|\\
&(L\gamma_{1}+\alpha_{1})\left \| y_{k-1}-y_{k-2}\right \| +(L\gamma_{2}+\alpha_{2})\left \| y_{k-2}-y_{k-3}\right \|.
\endaligned  
\end{equation}
Adding \eqref{(3.12)} and \eqref{(3.13)}, we get
\begin{equation*}
\aligned
&\left \| A_{x}^{k}\right \|+\left \| A_{y}^{k}\right \| \\
\le&\left \| \nabla _xH(u_{k-1},y_{k-1})-\widetilde{\nabla}_x(u_{k-1},y_{k-1}) \right \|+\left \| \nabla _yH(x_{k},v_{k-1})-\widetilde{\nabla}_y(x_{k},v_{k-1}) \right \|\\
&+2( 2N+\eta)\left \| z_{k}-z_{k-1}\right \|+2(N\gamma_1+\alpha _{1})\left \| z_{k-1}-z_{k-2}\right \| +2(N\gamma_2+\alpha _{2})\left \| z_{k-2}-z_{k-3}\right \|,
\endaligned  
\end{equation*}
where $N=\max\left \{ M,L \right \} $, $\eta =\max\left \{ \eta _1,\eta _2 \right \} $. Applying the conditional expectation operator and using \eqref{(2.2)} to bound the MSE terms, we can obtain
\begin{equation*}
\aligned
&\mathbb{E}_{k-1}\left \| (A_{x}^{k},A_{y}^{k} ) \right \|\le \mathbb{E}_{k-1}\left [ \left \| A_{x}^{k}\right \|+\left \| A_{y}^{k}\right \| \right ] \\
\le& (4N+2\eta+V_2)\mathbb{E} _{k-1}\left \| z_{k}-z_{k-1}\right \| +(2N\gamma_1+2\alpha _{1}+V_2)\left \| z_{k-1}-z_{k-2}\right \| \\
&+(2N\gamma_2+2\alpha _{2}+V_2)\left \| z_{k-2}-z_{k-3}\right \|+V_2\left \| z_{k-3}-z_{k-4}\right \|+\Gamma_{k-1}\\
\le& p\left ( \mathbb{E}_{k-1}\left \| z_{k}-z_{k-1} \right \|+\left \| z_{k-1}-z_{k-2} \right \| +\left \| z_{k-2}-z_{k-3} \right \|+\left \| z_{k-3}-z_{k-4}\right \|\right )+\Gamma_{k-1},\\
\endaligned  
\end{equation*}
where $p=2(2N+\eta+N\gamma_1+N\gamma_2+\alpha _{1}+\alpha _{2})+V_2$.
\end{proof}
\end{lemma}
Define the set of limit points of $\left \{ z_k\right \}_{k\in \mathbb{N} }$ as
$$\Omega:=\{\hat{ z}:  {\rm \ there\ exists\ a\ subsequence \ }\left \{ z_{k_l}\right \} {\rm \ of}\  \left \{ z_{k}\right \} {\rm \ such\ that\ } z_{k_l}\to\hat{ z} {\rm \ as}\ l\to \infty \}.$$
The following lemma describes properties of $\Omega$.
\begin{lemma}{\rm (limit points of $\left \{ z_k\right \}_{k\in \mathbb{N} }$)}
\label{lem33}
Suppose Assumption \ref{Assumption21} and \ref{Assumption41} hold. 
Let $\{z_k\}_{k\in \mathbb{N}}$ be a bounded sequence, which is generated by Algorithm \ref{alg1} with variance-reduced gradient estimator, let
$$\theta > L+2\alpha _1+2\alpha _2+2\sqrt{10(V_1+V_\Upsilon /\rho)+4L^2(\gamma_{1}^2+\gamma_{2}^2)}+6\epsilon.$$
where $\epsilon>0$ is small enough. Then

{\rm(1)} $\sum_{k=1}^{\infty }  \left \| z_{k}-z_{k-1} \right \|^2<\infty $ a.s., and $\left \| z_{k}-z_{k-1} \right \| \rightarrow 0$ a.s.;

{\rm(2)} $\mathbb{E} \Phi (z_k)\to \Phi ^\ast$, where $\Phi ^\ast \in [\underline{\Phi},\infty )$;

{\rm(3)} $\mathbb{E} {\rm dist}(0,\partial \Phi (z_k)) \to 0$;

{\rm(4)} the set $\Omega$ is nonempty, and for all $z^\ast \in \Omega$, $\mathbb{E} {\rm dist}(0,\partial \Phi (z^\ast)) = 0$;

{\rm(5)} ${\rm dist}(z_k ,\Omega)\to 0$ a.s.;

{\rm(6)} $\Omega$ is a.s. compact and connected;

{\rm(7)} $\mathbb{E} \Phi (z^\ast)= \Phi ^\ast$ for all $z^\ast \in \Omega$.
\begin{proof}
By Lemma \ref{lem31}, we have claim (1) holds.

\par According to \eqref{(3.2)}, the supermartingale convergence theorem ensures $\left \{ \Psi _{k}\right \} $ converges to a finite, positive random variable. Because $\left \| z_{k}-z_{k-1} \right \| \rightarrow 0$ a.s., $\left \| z_{k-1}-z_{k-2} \right \| \rightarrow 0$ a.s., $\left \| z_{k-2}-z_{k-3} \right \| \rightarrow 0$ a.s. and $\widetilde{\nabla}$ is variance-reduced so $\mathbb{E} \Upsilon_k \to 0$,
we can say 
$$\lim_{k \to \infty} \mathbb{E}\Psi _{k}=\lim_{k \to \infty} \mathbb{E}\Phi (z_{k}) \in [\underline{\Phi},\infty ),$$ 
which implys claim (2).
\par Claim (3) holds because, by Lemma \ref{lem32},
\begin{equation*}
\aligned
&\mathbb{E}\left \| (A_{x}^{k},A_{y}^{k} ) \right \|\\
\le& p \mathbb{E}\left(\left \| z_{k}-z_{k-1} \right \|+\left \| z_{k-1}-z_{k-2} \right \| +\left \| z_{k-2}-z_{k-3} \right \|+\left \| z_{k-3}-z_{k-4} \right \|\right )+\mathbb{E}\Gamma_{k-1}.
\endaligned  
\end{equation*}
We have that $\mathbb{E}\left \| z_{k}-z_{k-1} \right \| \rightarrow 0$ and $\mathbb{E} \Gamma_{k-1} \to 0$. This ensures that $\mathbb{E}\left \| (A_{x}^{k},A_{y}^{k} )\right \|\to 0$. Since $(A_{x}^{k},A_{y}^{k} )$ is one element of $\partial \Phi (z_k)$, we obtain $\mathbb{E} {\rm dist}(0,\partial \Phi (z_k))\le\mathbb{E}\left \| (A_{x}^{k},A_{y}^{k} )\right \| \to 0$. 
\par To prove claim (4), suppose $z^\ast =(x^\ast,y^\ast)$ is a limit point of the sequence $\left \{ z_k\right \}_{k\in \mathbb{N} }$ (a limit point must exist because we suppose the sequence $\left \{ z_k\right \}_{k\in \mathbb{N} }$ is bounded). This means there exists a subsequence $\left \{z_{k_j}\right \}$
satisfying $\lim_{j\to \infty} z_{k_j}= z^\ast $. Furthermore, by the variance-reduced property of $\widetilde{\nabla}(u_{k_j-1},y_{k_j-1})$, we have $\mathbb{E} \left \| \widetilde{\nabla}_x(u_{k_j-1},y_{k_j-1})-\nabla_x H(u_{k_j-1},y_{k_j-1}) \right \| ^2\to 0$.
Because $f$ and $g$ are lower semicontinuous, we have
\begin{equation}
\label{(3.14)}
\aligned
&\liminf_{j\rightarrow\infty}f(x_{k_j})\ge  f(x^\ast ),\\
&\liminf_{j\rightarrow\infty}g(y_{k_j})\ge  g(y^\ast ).
\endaligned  
\end{equation}
By the update rule for $x_{k_j}$,
letting $x=x^\ast$, we have
\begin{equation*}
\aligned
&f(x_{k_j})+\langle x _{k_j},\widetilde{\nabla}_x(u_{k_j-1},y_{k_j-1})\rangle+D_{\phi_1}(x_{k_j},x_{k_j-1})+\alpha_{1,k_j-1} \langle x_{k_j},x_{k_j-2}-x_{k_j-1}\rangle\\
&+\alpha_{2,k_j-1} \langle x_{k_j},x_{k_j-3}-x_{k_j-2}\rangle\\
\leq& f(x^\ast)+\langle x^\ast,\widetilde{\nabla}_x(u_{k_j-1},y_{k_j-1})\rangle+D_{\phi_1}(x^\ast,x_{k_j-1})+\alpha_{1,k_j-1} \langle x^\ast,x_{k_j-2}-x_{k_j-1}\rangle\\
&+\alpha_{2,k_j-1} \langle x^\ast,x_{k_j-3}-x_{k_j-2}\rangle.
\endaligned
\end{equation*}
Taking the expectation and taking the limit $j \to\infty$,
\begin{equation*}
\aligned
&\limsup_{j\rightarrow\infty}f(x_{k_j})\\
\le&\limsup_{j\rightarrow\infty}f(x^\ast)+\langle x^\ast-x_{k_j},\nabla_xH(u_{k_j-1},y_{k_j-1})\rangle+\langle x^\ast-x_{k_j},\widetilde{\nabla}_x(u_{k_j-1},y_{k_j-1})\\
&-\nabla_xH(u_{k_j-1},y_{k_j-1})\rangle+\phi_1(x^\ast)-\phi_1(x_{k_j})+\left \langle \nabla  \phi_1(x_{k_j-1}),x^\ast-x_{k_j-1} \right \rangle \\
&+\alpha_{1,k_j-1} \langle x^\ast-x_{k_j},x_{k_j-2}-x_{k_j-1}\rangle+\alpha_{2,k_j-1} \langle x^\ast-x_{k_j},x_{k_j-3}-x_{k_j-2}\rangle.
\endaligned
\end{equation*}
The second term on the right goes to zero because $x_{k_j}\to x^\ast$ and $\left \{ \nabla_xH(u_{k_j-1},y_{k_j-1})\right \}$ is bounded.  The
thrid term is zero almost surely because it is bounded above by $\left \| x^\ast-x_{k_j} \right \|^2$, and $\widetilde{\nabla}_x(u_{k_j-1},y_{k_j-1})-\nabla_xH(u_{k_j-1},y_{k_j-1})$ $\to 0$ a.s. Noting that $\phi_1$ is differentiable, so $\limsup_{j\rightarrow\infty}f(x_{k_j})\le  f(x^\ast )$ a.s., which, together with \eqref{(3.14)}, implies that $\lim_{j\rightarrow\infty}f(x_{k_j})= f(x^\ast)$ a.s. Similarly, we have
$\lim_{j\rightarrow\infty}g(y_{k_j})= g(y^\ast)$ a.s., and hence 
\begin{equation}
\label{(3.15)}
\lim_{j\rightarrow\infty}\Phi (x_{k_j},y_{k_j})=\Phi (x^\ast ,y^\ast)\ \ {\rm a.s}.
\end{equation}
Claim (3) ensures that $\mathbb{E} {\rm dist}(0,\partial \Phi (z_k)) \to 0$. Combining \eqref{(3.15)} and the fact that the subdifferential of $\Phi$ is closed, we have $\mathbb{E} {\rm dist}(0,\partial \Phi (z^\ast)) = 0$.
\par Claims (5) and (6) hold for any sequence satisfying $\left \| z_{k}-z_{k-1} \right \| \rightarrow 0$ a.s. (this fact is used in the same context in \cite{JSM,D})
\par Finally, we must show that $\Phi$ has constant expectation over $\Omega$. From claim (2), we have $\mathbb{E} \Phi (z_k)\to \Phi ^\ast$, which implies $\mathbb{E} \Phi (z_{k_j})\to \Phi ^\ast$ for every subsequence $\left \{ z_{k_j}\right \}_{j\in \mathbb{N} }$ converging to some $z^\ast \in \Omega$. In the proof of claim (4), we show that $\Phi (z_{k_j})\to \Phi(z ^\ast)$ a.s., so $\mathbb{E} \Phi (z^\ast)= \Phi ^\ast$ for all $z^\ast \in \Omega$.
\end{proof}
\end{lemma}
The following lemma is analogous to the uniformized Kurdyka--{\L}ojasiewicz property \cite{JSM}. It is a slight generalization of the K{\L} property showing that $z_k$ eventually enters a region of $\tilde{z} $ for some $\tilde{z} $ satisfying $\Phi (\tilde{z} )= \Phi(z ^\ast)$, and in this region, the K{\L} inequality holds.
\begin{lemma}
\label{lem34}
Assume that the conditions of Lemma \ref{lem33} hold and that $z_k$ is not a critical point of $\Phi$ after a finite number of iterations. Let $\Phi$ be a semialgebraic function with K{\L} exponent $\vartheta$. Then there exist an index $m$ and a desingularizing function $\varphi$ so that the following bound holds:
$$\varphi'(\mathbb{E} [\Phi (z_k)-\Phi_k ^\ast])\mathbb{E}{\rm dist}(0,\partial \Phi (z_k))\geq 1,\ \ \forall k>m,$$
where $\Phi_k ^\ast$ is a nondecreasing sequence converging to $\mathbb{E} \Phi (z^\ast)$ for all $z^\ast \in \Omega$.
\end{lemma}
\par The proof is almost the same as that of Lemma 4.5 in \cite{DT}. We omit the proof here. We now show that the iterates of Algorithm \ref{alg1} have finite length in expectation.
\begin{theorem}{\rm (finite length)}
\label{th30}
Assume that the conditions of Lemma \ref{lem33} hold and $\Phi$ is a semialgebraic function with K{\L} exponent $\vartheta \in [0,1)$. Let $\{z_k\}_{k\in \mathbb{N}}$ be a bounded sequence, which is generated by Algorithm \ref{alg1} with variance-reduced gradient estimator.

{\rm(i)} Either $z_k$ is a critical point after a finite number of iterations or $\left \{ z_k\right \}_{k\in \mathbb{N} }$ satisfies the finite-length property in expectation:
$$\sum_{k=0}^{\infty }  \mathbb{E}\left \| z_{k+1}-z_{k} \right \|<\infty, $$
and there exists an integer m so that, for all $i > m$,
\begin{equation}
\label{(3.16)}
\aligned
&\sum_{k=m}^{i}\mathbb{E}\left \| z_{k+1}-z_{k} \right \| +\sum_{k=m}^{i} \mathbb{E}\left \| z_{k}-z_{k-1} \right \|+ \sum_{k=m}^{i}\mathbb{E}\left \| z_{k-1}-z_{k-2} \right \|+ \sum_{k=m}^{i}\mathbb{E}\left \| z_{k-2}-z_{k-3} \right \|\\
\le&\sqrt{\mathbb{E}\left \| z_{m}-z_{m-1} \right \|^2} +\sqrt{\mathbb{E}\left \| z_{m-1}-z_{m-2} \right \|^2}+ \sqrt{\mathbb{E}\left \| z_{m-2}-z_{m-3} \right \|^2}+ \sqrt{\mathbb{E}\left \| z_{m-3}-z_{m-4} \right \|^2}\\
&+\frac{2\sqrt{s} }{K_1\rho } \sqrt{\mathbb{E}\Upsilon _{m-1}}+K_3\triangle _{m,i+1}, 
\endaligned  
\end{equation}
where
$$K_1=p+\frac{2\sqrt{sV_\Upsilon } }{\rho }, \ K_3=\frac{4K_1 }{K_2},\ K_2=\min\left \{ \kappa ,\epsilon,Z \right \}, $$
$p$ is as in Lemma \ref{lem32}, and $\triangle _{p,q}=(\mathbb{E}[\Psi _p-\Phi _{p}^{\ast } ]-\mathbb{E}[\Psi _q-\Phi _{q}^{\ast } ])$.

{\rm(ii)} $\left \{ z_k\right \}_{k\in \mathbb{N} }$ generated by Algorithm \ref{alg1} converge to a critical point of $\Phi$ in expectation.
\begin{proof}
(i) If $\vartheta \in (0,\frac{1}{2})$, then $\Phi$ satisfies the K{\L} property with exponent $\frac{1}{2}$, so we consider only the case $\vartheta \in [ \frac{1}{2},1)$. By Lemma \ref{lem34}, there exists a function $\varphi_0(r)=ar^{1-\vartheta }$ such that
$$\varphi_0'(\mathbb{E}[ \Phi (z_k)-\Phi_k ^\ast])\mathbb{E}{\rm dist}(0,\partial \Phi (z_k))\geq 1,\ \ \forall k>m.$$
Lemma \ref{lem32} provides a bound on $\mathbb{E}{\rm dist}(0,\partial \Phi (z_k))$.
\begin{equation}
\label{(3.17)}
\aligned
&\mathbb{E}{\rm dist}(0,\partial \Phi (z_k)) \le \mathbb{E}\left \| (A_{x}^{k},A_{y}^{k} ) \right \|\\
\le& p\mathbb{E}\left ( \left \| z_{k}-z_{k-1} \right \|+\left \| z_{k-1}-z_{k-2} \right \| +\left \| z_{k-2}-z_{k-3} \right \|+\left \| z_{k-3}-z_{k-4} \right \|\right )+\mathbb{E}\Gamma_{k-1}\\
\le& p\left ( \sqrt{\mathbb{E}\left \| z_{k}-z_{k-1} \right \|^2}+\sqrt{\mathbb{E}\left \| z_{k-1}-z_{k-2} \right \|^2} +\sqrt{\mathbb{E}\left \| z_{k-2}-z_{k-3} \right \|^2}+\sqrt{\mathbb{E}\left \| z_{k-3}-z_{k-4} \right \|^2}\right )\\
&+\sqrt{s\mathbb{E}\Upsilon _{k-1} } .
\endaligned  
\end{equation}
The final inequality is Jensen's inequality. Because $\Gamma _k=\sum_{i=1}^{s} v_{k}^{i} $ for some nonnegative random variables $v_{k}^{i} $, we can say $\mathbb{E}\Gamma _k=\mathbb{E}\sum_{i=1}^{s} v_{k}^{i} \le\mathbb{E}\sqrt{s\sum_{i=1}^{s} (v_{k}^{i} )^2}  \le \sqrt{s\mathbb{E}\Upsilon _{k} } $. We can bound the term $\sqrt{\mathbb{E}\Upsilon _{k} } $ using \eqref{(2.3)}:
\begin{equation}
\label{(3.18)}
\aligned
&\sqrt{\mathbb{E}\Upsilon _{k}}\\
\le& \sqrt{(1-\rho )\mathbb{E}\Upsilon _{k-1}+V_\Upsilon \mathbb{E}\left ( \left \| z_{k}-z_{k-1} \right \| ^{2}+\left \| z_{k-1}-z_{k-2} \right \| ^{2} +\left \| z_{k-2}-z_{k-3} \right \| ^{2}+\left \| z_{k-3}-z_{k-4} \right \| ^{2}\right )}\\
\le& \sqrt{(1-\rho )}\sqrt{\mathbb{E}\Upsilon _{k-1}} +\sqrt{V_\Upsilon} \left ( \sqrt{ \mathbb{E}\left \| z_{k}-z_{k-1} \right \| ^{2}} +\sqrt{\mathbb{E}\left \| z_{k-1}-z_{k-2} \right \| ^{2} } +\sqrt{\mathbb{E}\left \| z_{k-2}-z_{k-3} \right \| ^{2}}\right.\\
&\left. +\sqrt{\mathbb{E}\left \| z_{k-3}-z_{k-4} \right \| ^{2}}  \right )  \\
\le&(1-\frac{\rho}{2} )\sqrt{\mathbb{E}\Upsilon _{k-1}} +\sqrt{V_\Upsilon} \left ( \sqrt{  \mathbb{E}\left \| z_{k}-z_{k-1} \right \| ^{2}} +\sqrt{\mathbb{E}\left \| z_{k-1}-z_{k-2} \right \| ^{2} } +\sqrt{\mathbb{E}\left \| z_{k-2}-z_{k-3} \right \| ^{2}}\right.\\
&\left.+\sqrt{\mathbb{E}\left \| z_{k-3}-z_{k-4} \right \| ^{2}}  \right ).
\endaligned
\end{equation}
The final inequality uses the fact that $\sqrt{1-\rho } =1-\frac{\rho }{2}- \frac{\rho^2 }{8}-\cdots $. This implies that
\begin{equation}
\label{(5.18)}
\aligned
&\sqrt{s\mathbb{E}\Upsilon _{k-1}}\\
\le& \frac{2\sqrt{s} }{\rho} \left ( \sqrt{\mathbb{E}\Upsilon _{k-1}}-\sqrt{\mathbb{E}\Upsilon _{k}} \right )  +\frac{2\sqrt{sV_\Upsilon}}{\rho} \left ( \sqrt{  \mathbb{E}\left \| z_{k}-z_{k-1} \right \| ^{2}} +\sqrt{\mathbb{E}\left \| z_{k-1}-z_{k-2} \right \| ^{2} } \right .\\
&\left.+\sqrt{\mathbb{E}\left \| z_{k-2}-z_{k-3} \right \| ^{2}}  +\sqrt{\mathbb{E}\left \| z_{k-3}-z_{k-4} \right \| ^{2}}\right ).
\endaligned
\end{equation}
Then, from \eqref{(3.17)} and \eqref{(5.18)}, we have
\begin{equation*}
\aligned
&\mathbb{E}{\rm dist}(0,\partial \Phi (z_k))\\
\le&\left ( p+\frac{2\sqrt{sV_\Upsilon}  }{\rho } \right ) \left ( \sqrt{  \mathbb{E}\left \| z_{k}-z_{k-1} \right \| ^{2}} +\sqrt{\mathbb{E}\left \| z_{k-1}-z_{k-2} \right \| ^{2} } +\sqrt{\mathbb{E}\left \| z_{k-2}-z_{k-3} \right \| ^{2}} \right .\\
&\left.+\sqrt{\mathbb{E}\left \| z_{k-3}-z_{k-4} \right \| ^{2}}  \right ) +\frac{2\sqrt{s} }{\rho } \left (\sqrt{\mathbb{E}\Upsilon _{k-1} }-\sqrt{\mathbb{E}\Upsilon _{k} }  \right )\\
=&K_1\left ( \sqrt{  \mathbb{E}\left \| z_{k}-z_{k-1} \right \| ^{2}} +\sqrt{\mathbb{E}\left \| z_{k-1}-z_{k-2} \right \| ^{2} } +\sqrt{\mathbb{E}\left \| z_{k-2}-z_{k-3} \right \| ^{2}}+\sqrt{\mathbb{E}\left \| z_{k-3}-z_{k-4} \right \| ^{2}}  \right ) \\
&+\frac{2\sqrt{s} }{\rho } \left (\sqrt{\mathbb{E}\Upsilon _{k-1} }-\sqrt{\mathbb{E}\Upsilon _{k} }  \right ),
\endaligned  
\end{equation*}
where $K_1=p+\frac{2\sqrt{sV_\Upsilon}  }{\rho }$. Define $C_k$ to be the right side of this inequality:
\begin{equation*}
\aligned
C_k=&  K_1\sqrt{\mathbb{E}\left \| z_{k}-z_{k-1} \right \|^2}+ K_1\sqrt{\mathbb{E}\left \| z_{k-1}-z_{k-2} \right \|^2} + K_1\sqrt{\mathbb{E}\left \| z_{k-2}-z_{k-3} \right \|^2}\\
&+ K_1\sqrt{\mathbb{E}\left \| z_{k-3}-z_{k-4} \right \|^2}+\frac{2\sqrt{s} }{\rho } \left (\sqrt{\mathbb{E}\Upsilon _{k-1} }-\sqrt{\mathbb{E}\Upsilon _{k} }  \right ).
\endaligned  
\end{equation*}
We then have
\begin{equation}
\label{(3.20)}
\varphi_0'(\mathbb{E} [\Phi (z_k)-\Phi_k ^\ast])C_k\geq 1,\ \ \forall k>m. 
\end{equation}
By the definition of $\varphi_0$, this is equivalent to
\begin{equation}
\label{(3.21)}
\frac{a(1-\vartheta )C_k}{(\mathbb{E} [\Phi (z_k)-\Phi_k ^\ast])^\vartheta }\geq 1,\ \ \forall k>m. 
\end{equation}
\par We would like to hold the inequality above for $\Psi_k$ rather than $\Phi (z_k)$. Replace $\mathbb{E} \Phi (z_k)$ with
$\mathbb{E}\Psi_k$ by introducing a term of $\mathcal{O}\left ( \left ( \mathbb{E}\left [ \left \| z_{k}-z_{k-1} \right \|^2+\left \| z_{k-1}-z_{k-2} \right \|^2+\left \| z_{k-2}-z_{k-3} \right \|^2+\Upsilon _k \right ]   \right )^\vartheta   \right ) $ in the denominator. We show that inequality \eqref{(3.21)} still holds after this adjustment because these terms are small compared to $C_k$. Indeed, the quantity
\begin{equation*}
\aligned
C_k\ge &c_1\left ( \sqrt{\mathbb{E}\left \| z_{k}-z_{k-1} \right \|^2}+ \sqrt{\mathbb{E}\left \| z_{k-1}-z_{k-2} \right \|^2} +\sqrt{\mathbb{E}\left \| z_{k-2}-z_{k-3} \right \|^2}\right .\\
&\left.+\sqrt{\mathbb{E}\left \| z_{k-3}-z_{k-4} \right \|^2}+\sqrt{\mathbb{E}\Upsilon _{k-1} } \right )
\endaligned  
\end{equation*} 
for some constant $c_1>0$. And because $\mathbb{E}\left \| z_{k}-z_{k-1} \right \|^2 \to 0$, $\mathbb{E} \Upsilon_k \to 0$, and $\vartheta >\frac{1}{2} $, there exists an index $m$ and constants $c_2,c_3>0$ such that
\begin{equation*}
\aligned
&\left (\mathbb{E}[\Psi _k-\Phi (z_k) ]\right ) ^\vartheta \\
=&\left ( \mathbb{E}\left [\frac{1}{L\lambda\rho  } \Upsilon _k+\left (\frac{V_1+V_\Upsilon /\rho }{L\lambda }+\frac{\alpha_1+\alpha_2}{2}+\frac{2L(\gamma_{1}^2+\gamma_{2}^2)}{\lambda }+3Z  \right )\|z_{k}-z_{k-1}\|^2+\left ( \frac{V_1+V_\Upsilon /\rho }{L\lambda }\right .\right .\right .\\
&\left.\left.\left.+\frac{\alpha_2}{2}+\frac{2L\gamma_{2}^2}{\lambda }+2Z \right )\|z_{k-1}-z_{k-2}\|^2+\left ( \frac{V_1+V_\Upsilon /\rho }{L\lambda } +Z \right ) \|z_{k-2}-z_{k-3}\|^2\right ]\right ) ^\vartheta \\
\le&c_2\left ( \left ( \mathbb{E}\left [ \Upsilon _{k-1}+\left \| z_{k}-z_{k-1} \right \| ^{2} +\left \| z_{k-1}-z_{k-2} \right \| ^{2}+\left \| z_{k-2}-z_{k-3} \right \| ^{2}+\left \| z_{k-3}-z_{k-4} \right \| ^{2}\right ] \right )^\vartheta   \right )\\
\le&c_3C_k,\ \ \forall k>m.
\endaligned  
\end{equation*}
The first inequality uses \eqref{(2.3)}. Because the terms above are small compared to $C_k$, there exists a constant $d$ such that $c_3<d<+\infty $ and
\begin{equation*}
\frac{ad(1-\vartheta )C_k}{(\mathbb{E}[\Phi (z_k)-\Phi_k ^\ast])^\vartheta+\left (\mathbb{E}[\Psi _k-\Phi (z_k) ]\right ) ^\vartheta }\geq 1,\ \ \forall k>m. 
\end{equation*}
For $\vartheta \in [ \frac{1}{2},1)$, using the fact that $(a+b)^\vartheta \le a^\vartheta +b^\vartheta$  for all $a, b \ge 0$, we have
\begin{equation*}
\aligned
\frac{ad(1-\vartheta )C_k}{\left (\mathbb{E}[\Psi _k-\Phi_k ^\ast ]\right ) ^\vartheta }&=\frac{ad(1-\vartheta )C_k}{\left (\mathbb{E}[\Phi (z_k)-\Phi_k ^\ast+\Psi _k-\Phi (z_k) ]\right ) ^\vartheta }\\
&\ge\frac{ad(1-\vartheta )C_k}{\left (\mathbb{E}[\Phi (z_k)-\Phi_k ^\ast]\right ) ^\vartheta+\left (\mathbb{E}[\Psi _k-\Phi (z_k) ]\right ) ^\vartheta  } \\
&\geq 1,\ \ \forall k>m. 
\endaligned  
\end{equation*}
Therefore, with $\varphi (r)=adr^{1-\vartheta }$,
\begin{equation}
\label{(3.22)}
\varphi'(\mathbb{E}[\Psi _k-\Phi_k ^\ast])C_k\geq 1,\ \ \forall k>m.
\end{equation}
By the concavity of $\varphi$,
\begin{equation*}
\aligned
\varphi(\mathbb{E}[\Psi _k-\Phi_k ^\ast])-\varphi(\mathbb{E}[\Psi _{k+1}-\Phi_{k+1} ^\ast])&\ge \varphi'(\mathbb{E}[\Psi _k-\Phi_k ^\ast])(\mathbb{E}[\Psi _k-\Phi_k ^\ast+\Phi_{k+1} ^\ast-\Psi _{k+1}])\\
&\ge \varphi'(\mathbb{E}[\Psi _k-\Phi_k ^\ast])(\mathbb{E}[\Psi _k-\Psi _{k+1}]),
\endaligned  
\end{equation*}
where the last inequality follows from the fact that $\Phi_k ^\ast$ is nondecreasing. With $\triangle _{p,q}=\varphi(\mathbb{E}[\Psi _p-\Phi _{p}^{\ast } ])-\varphi(\mathbb{E}[\Psi _q-\Phi _{q}^{\ast } ])$, we have shown
$$\triangle _{k,k+1}C_k\ge \mathbb{E}[\Psi _k-\Psi _{k+1}],\ \forall k>m.$$
Using Lemma \ref{lem31}, we can bound $\mathbb{E}[\Psi _k-\Psi _{k+1}]$ below by both $\mathbb{E}\left \| z_{k+1}-z_{k} \right \|^2$, $\mathbb{E}\left \| z_{k}-z_{k-1} \right \|^2$, $\mathbb{E}\left \| z_{k-1}-z_{k-2} \right \|^2$ and $\mathbb{E}\left \| z_{k-2}-z_{k-3} \right \|^2$. Specifically,
\begin{equation}
\label{(3.23)}
\aligned 
\triangle _{k,k+1}C_k&\ge \kappa\mathbb{E}\left \| z_{k+1}-z_{k}\right \|^2+\epsilon\mathbb{E}\left \| z_{k}-z_{k-1} \right \|^2+\epsilon\mathbb{E}\left \| z_{k-1}-z_{k-2}\right \|^2+Z\mathbb{E}\left \| z_{k-2}-z_{k-3}\right \|^2\\
&\ge K_2\mathbb{E}\left \| z_{k+1}-z_{k}\right \|^2+K_2\mathbb{E}\left \| z_{k}-z_{k-1} \right \|^2+K_2\mathbb{E}\left \| z_{k-1}-z_{k-2}\right \|^2+K_2\mathbb{E}\left \| z_{k-2}-z_{k-3}\right \|^2,
\endaligned 
\end{equation}
where $K_2=\min\left \{ \kappa ,\epsilon,Z \right \}>0$, $\kappa$, $\lambda$, $\epsilon$ and $Z$ are set as in Lemma \ref{lem31}. Let us use the first of these inequalities to begin. Applying Young's inequality to \eqref{(3.23)} yields
\begin{equation}
\label{(3.24)}
\aligned
&\sqrt{\mathbb{E}\left \| z_{k+1}-z_{k} \right \|^2} +\sqrt{\mathbb{E}\left \| z_{k}-z_{k-1} \right \|^2} +\sqrt{\mathbb{E}\left \| z_{k-1}-z_{k-2} \right \|^2}+\sqrt{\mathbb{E}\left \| z_{k-2}-z_{k-3} \right \|^2}\\
\le&2\sqrt{\mathbb{E}\left \| z_{k+1}-z_{k} \right \|^2+\mathbb{E}\left \| z_{k}-z_{k-1} \right \|^2+\mathbb{E}\left \| z_{k-1}-z_{k-2} \right \|^2+\mathbb{E}\left \| z_{k-2}-z_{k-3} \right \|^2}\\
\le& 2\sqrt{K_2^{-1}C_k\triangle _{k,k+1}} \le \frac{C_k}{2K_1}+\frac{2K_1\triangle _{k,k+1}}{K_2}\\
\le& \frac{ 1}{2}\sqrt{\mathbb{E}\left \| z_{k}-z_{k-1} \right \|^2} +\frac{1}{2}\sqrt{\mathbb{E}\left \| z_{k-1}-z_{k-2} \right \|^2} +\frac{1}{2}\sqrt{\mathbb{E}\left \| z_{k-2}-z_{k-3} \right \|^2}+\frac{1}{2}\sqrt{\mathbb{E}\left \| z_{k-3}-z_{k-4} \right \|^2}\\
&+\frac{\sqrt{s} }{K_1\rho } \left (\sqrt{\mathbb{E}\Upsilon _{k-1} }-\sqrt{\mathbb{E}\Upsilon _{k} }  \right )+\frac{2K_1\triangle _{k,k+1}}{K_2}.
\endaligned
\end{equation}
Summing inequality \eqref{(3.24)} from $k=m$ to $k=i$, set 
\begin{equation}
\label{(3.25)}
\aligned
T_m^{i}=&\sum_{k=m}^{i} \sqrt{\mathbb{E}\left \| z_{k+1}-z_{k} \right \|^2}+\sum_{k=m}^{i}\sqrt{\mathbb{E}\left \| z_{k}-z_{k-1} \right \|^2} +\sum_{k=m}^{i}\sqrt{\mathbb{E}\left \| z_{k-1}-z_{k-2} \right \|^2}\\
&+\sum_{k=m}^{i}\sqrt{\mathbb{E}\left \| z_{k-2}-z_{k-3} \right \|^2}.
\endaligned  
\end{equation}
Then
\begin{equation*}
\label{(3.26)}
\aligned
T_m^{i}\le \frac{1}{2}T_{m-1}^{i-1}+\frac{\sqrt{s} }{K_1\rho }\left (\sqrt{\mathbb{E}\Upsilon _{m-1} }-\sqrt{\mathbb{E}\Upsilon _{i} }  \right )+\frac{2K_1}{K_2}\triangle _{m,i+1},
\endaligned  
\end{equation*}
which implies that
\begin{equation*}
\label{(3.27)}
\aligned
\frac{1}{2}T_m^{i}\le&\frac{1}{2} \sqrt{\mathbb{E}\left \| z_{m}-z_{m-1} \right \|^2}+\frac{1}{2}\sqrt{\mathbb{E}\left \| z_{m-1}-z_{m-2} \right \|^2}+\frac{1}{2} \sqrt{\mathbb{E}\left \| z_{m-2}-z_{m-3} \right \|^2} \\
&+\frac{1}{2} \sqrt{\mathbb{E}\left \| z_{m-3}-z_{m-4} \right \|^2}+\frac{\sqrt{s} }{K_1\rho } \left (\sqrt{\mathbb{E}\Upsilon _{m-1} }-\sqrt{\mathbb{E}\Upsilon _{i} }  \right )+\frac{2K_1}{K_2}\triangle _{m,i+1}.
\endaligned  
\end{equation*}
Dropping the nonpositive term $-\sqrt{\mathbb{E}\Upsilon _{i} }$, this show that
\begin{equation}
\label{(3.302)}
\aligned
T_m^{i}
\le&\sqrt{\mathbb{E}\left \| z_{m}-z_{m-1} \right \|^2}+\sqrt{\mathbb{E}\left \| z_{m-1}-z_{m-2} \right \|^2}+ \sqrt{\mathbb{E}\left \| z_{m-2}-z_{m-3} \right \|^2} \\
&+ \sqrt{\mathbb{E}\left \| z_{m-3}-z_{m-4} \right \|^2}+\frac{2\sqrt{s} }{K_1\rho } \sqrt{\mathbb{E}\Upsilon _{m-1} }+K_3\triangle _{m,i+1}.
\endaligned  
\end{equation}
where $K_3=\frac{4K_1}{K_2}$. Applying Jensen's inequality to the terms on the left gives
\begin{equation*}
\label{(3.16)}
\aligned
&\sum_{k=m}^{i}\mathbb{E}\left \| z_{k+1}-z_{k} \right \| +\sum_{k=m}^{i} \mathbb{E}\left \| z_{k}-z_{k-1} \right \|+ \sum_{k=m}^{i}\mathbb{E}\left \| z_{k-1}-z_{k-2} \right \|+ \sum_{k=m}^{i}\mathbb{E}\left \| z_{k-2}-z_{k-3} \right \|\le T_m^{i}\\
\le&\sqrt{\mathbb{E}\left \| z_{m}-z_{m-1} \right \|^2}+\sqrt{\mathbb{E}\left \| z_{m-1}-z_{m-2} \right \|^2}+\sqrt{\mathbb{E}\left \| z_{m-2}-z_{m-3} \right \|^2}+\sqrt{\mathbb{E}\left \| z_{m-3}-z_{m-4} \right \|^2} \\
&+\frac{2\sqrt{s} }{K_1\rho } \sqrt{\mathbb{E}\Upsilon _{m-1} }+K_3\triangle _{m,i+1}.
\endaligned  
\end{equation*}
The term $\lim_{i \to \infty} \triangle _{m,i+1}$ is bounded because $\mathbb{E}\Psi_k$ is bounded due to Lemma \ref{lem31}. Letting $i \to \infty$, we prove the assertion.
\par (ii) An immediate consequence of claim (i) is that the sequence $\left \{ z_k\right \}_{k\in \mathbb{N} }$ converges in expectation to a critical point. This is because, for any $p,q \in\mathbb{N}$ with $p \ge q$, $\mathbb{E}\left \| z_{p}-z_{q} \right \|=\mathbb{E}\left \| \sum_{k=q}^{p-1}( z_{k+1}-z_{k}) \right \|\le\sum_{k=q}^{p-1} \mathbb{E}\left \| z_{k+1}-z_{k} \right \|  $, and the finite length property implies this final sum converges to zero. This proves claim (ii).
\end{proof}
\end{theorem}

\begin{theorem}
\label{th31}
{\it  Assume that the conditions of Lemma \ref{lem33} hold and $\Phi$ is a semialgebraic function with K{\L} exponent $\vartheta\in [0, 1)$. Let $\{z_k\}_{k\in \mathbb{N}}$ be a bounded sequence, which is generated by Algorithm \ref{alg1} with variance-reduced gradient estimator. The following convergence rates hold:

{\rm (i)} If $\vartheta\in (0, \frac{1}{2} ]$, then there exist $d_1 > 0$ and $\tau \in [1 - \rho,1)$ such that $\mathbb{E} \left \| z_k-z^\ast  \right \| \le d_1\tau ^k$.

{\rm (ii)} If $\vartheta \in (\frac{1}{2} ,1)$, then there exists a constant $d_2 > 0$ such that $\mathbb{E} \left \| z_k-z^\ast  \right \| \le d_2k ^{-\frac{1-\vartheta }{2\vartheta -1} }$.

{\rm (iii)} If $\vartheta = 0$, then there exists an $m \in \mathbb{N}$ such that $\mathbb{E} \Phi (z_k)=\mathbb{E} \Phi (z^\ast )$ for all $k \ge m$.
}
\end{theorem}
\begin{proof}
As in the proof of Theorem \ref{th30}, if $\vartheta\in (0, \frac{1}{2} )$, then $\Phi$ satisfies the K{\L} property with exponent $\frac{1}{2} $, so we consider only the case $\vartheta \in [\frac{1}{2} ,1)$.\\
Let
\begin{equation*}
\aligned
T_m=&\sum_{k=m}^{\infty} \sqrt{\mathbb{E}\left \| z_{k+1}-z_{k} \right \|^2}+\sum_{k=m}^{\infty}\sqrt{\mathbb{E}\left \| z_{k}-z_{k-1} \right \|^2} +\sum_{k=m}^{\infty}\sqrt{\mathbb{E}\left \| z_{k-1}-z_{k-2} \right \|^2}\\
&+\sum_{k=m}^{\infty}\sqrt{\mathbb{E}\left \| z_{k-2}-z_{k-3} \right \|^2}.
\endaligned  
\end{equation*}
Substituting the desingularizing function $\varphi (r)=ar^{1-\vartheta }$ into \eqref{(3.302)}, let $i\to \infty$, then we have
\begin{equation}
\label{(3.32)}
\aligned
T_m\le&\sqrt{\mathbb{E}\left \| z_{m}-z_{m-1} \right \|^2}+\sqrt{\mathbb{E}\left \| z_{m-1}-z_{m-2} \right \|^2}+\sqrt{\mathbb{E}\left \| z_{m-2}-z_{m-3} \right \|^2}+\sqrt{\mathbb{E}\left \| z_{m-3}-z_{m-4} \right \|^2} \\
&+\frac{2\sqrt{s} }{K_1\rho } \sqrt{\mathbb{E}\Upsilon _{m-1} }+aK_3(\mathbb{E}[\Psi _m-\Phi _{m}^{\ast } ])^{1-\vartheta }.\\
\endaligned  
\end{equation}
Because $\Psi _m=\Phi(z_m)+\mathcal{O}(\left \| z_{m}-z_{m-1} \right \|^2+\left \| z_{m-1}-z_{m-2} \right \|^2+\left \| z_{m-2}-z_{m-3} \right \|^2+\Upsilon _m) $, we can rewrite the final term as $\Phi(z_m)-\Phi _{m}^{\ast }$.
\begin{equation}
\label{(3.320)}
\aligned
&(\mathbb{E}[\Psi _m-\Phi _{m}^{\ast } ])^{1-\vartheta }\\
=&\left (\mathbb{E}\left [\Phi(z_m)-\Phi _{m}^{\ast }+ \frac{1}{L\lambda\rho  } \Upsilon _k+\left (\frac{V_1+V_\Upsilon /\rho }{L\lambda }+\frac{\alpha_1+\alpha_2}{2}+\frac{2L(\gamma_{1}^2+\gamma_{2}^2)}{\lambda }+3Z  \right )\|z_{m}-z_{m-1}\|^2 \right .\right .\\
&\left.\left.+\left ( \frac{V_1+V_\Upsilon /\rho }{L\lambda }+\frac{\alpha_2}{2}+\frac{2L\gamma_{2}^2}{\lambda }+2Z \right )\|z_{m-1}-z_{m-2}\|^2+\left ( \frac{V_1+V_\Upsilon /\rho }{L\lambda } +Z \right ) \right .\right .\\
&\left.\left. \|z_{m-2}-z_{m-3}\|^2\right ] \right )^{1-\vartheta }\\
\overset{(1)}{\le }& \left (\mathbb{E}[\Phi(z_m)-\Phi _{m}^{\ast }]\right )^{1-\vartheta }+\left (\frac{1}{L\lambda\rho  } \mathbb{E}\Upsilon _m\right )^{1-\vartheta }+\left (\left ( \frac{V_1+V_\Upsilon /\rho }{L\lambda }+\frac{\alpha_1+\alpha_2}{2}+\frac{2L(\gamma_{1}^2+\gamma_{2}^2)}{\lambda }+3Z  \right )\right .\\
&\left. \mathbb{E}\|z_{m}-z_{m-1}\|^2\right )^{1-\vartheta }+\left (\left( \frac{V_1+V_\Upsilon /\rho }{L\lambda }+\frac{\alpha_2}{2}+\frac{2L\gamma_{2}^2}{\lambda }+2Z \right ) \mathbb{E}\|z_{m-1}-z_{m-2}\|^2\right )^{1-\vartheta }\\
&+\left (\left( \frac{V_1+V_\Upsilon /\rho }{L\lambda } +Z \right ) \mathbb{E}\|z_{m-2}-z_{m-3}\|^2\right )^{1-\vartheta }.
\endaligned  
\end{equation}
Inequality (1) is due to the fact that $(a+b) ^{1-\vartheta }\le a^{1-\vartheta }+b^{1-\vartheta }$. 
Applying the K{\L} inequality \eqref{(2.011)},
\begin{equation}
\label{(3.33)}
aK_3\left (\mathbb{E}[\Phi(z_m)-\Phi _{m}^{\ast }]\right )^{1-\vartheta }\le aK_4\left (\mathbb{E}\left \| \xi _m \right \| \right )^{\frac{1-\vartheta}{\vartheta }  }  
\end{equation}
for all $\xi _m\in \partial \Phi (z_m)$ and we have absorbed the constant $C$ into $K_4$. Inequality \eqref{(3.17)} provides a bound on the norm of the subgradient:
\begin{equation*}
\aligned
\left (\mathbb{E}\left \| \xi _m \right \| \right )^{\frac{1-\vartheta}{\vartheta }  } \le&\left ( p\left ( \sqrt{\mathbb{E}\left \| z_{m}-z_{m-1} \right \|^2}+\sqrt{\mathbb{E}\left \| z_{m-1}-z_{m-2} \right \|^2} +\sqrt{\mathbb{E}\left \| z_{m-2}-z_{m-3} \right \|^2}\right.\right.\\
&\left.\left.+\sqrt{\mathbb{E}\left \| z_{m-3}-z_{m-4} \right \|^2}\right )+\sqrt{s\mathbb{E}\Upsilon _{m-1} } \right )^{\frac{1-\vartheta}{\vartheta }  }.  
\endaligned  
\end{equation*}
Let
\begin{equation*}
\aligned
\Theta _{m}=&p\left ( \sqrt{\mathbb{E}\left \| z_{m}-z_{m-1} \right \|^2}+\sqrt{\mathbb{E}\left \| z_{m-1}-z_{m-2} \right \|^2} +\sqrt{\mathbb{E}\left \| z_{m-2}-z_{m-3} \right \|^2}\right.\\
&\left.+\sqrt{\mathbb{E}\left \| z_{m-3}-z_{m-4} \right \|^2}\right )+\sqrt{s\mathbb{E}\Upsilon _{m-1} }.
\endaligned  
\end{equation*}
Therefore, it follows from \eqref{(3.32)}-\eqref{(3.33)} that
\begin{equation}
\label{(3.34)}
\aligned
T_m\le&\sqrt{\mathbb{E}\left \| z_{m}-z_{m-1} \right \|^2}+\sqrt{\mathbb{E}\left \| z_{m-1}-z_{m-2} \right \|^2}+\sqrt{\mathbb{E}\left \| z_{m-2}-z_{m-3} \right \|^2}+\sqrt{\mathbb{E}\left \| z_{m-3}-z_{m-4} \right \|^2} \\
&+\frac{2\sqrt{s} }{K_1\rho } \sqrt{\mathbb{E}\Upsilon _{m-1} }+aK_4\Theta _{m}^{\frac{1-\vartheta}{\vartheta } }+aK_3\left (\frac{1}{L\lambda\rho  } \mathbb{E}\Upsilon _m\right )^{1-\vartheta }\\
&+aK_3\left (\left ( \frac{V_1+V_\Upsilon /\rho }{L\lambda }+\frac{\alpha_1+\alpha_2}{2}+\frac{2L(\gamma_{1}^2+\gamma_{2}^2)}{\lambda }+3Z  \right )\mathbb{E}\|z_{m}-z_{m-1}\|^2\right )^{1-\vartheta }\\
&+aK_3\left (\left( \frac{V_1+V_\Upsilon /\rho }{L\lambda }+\frac{\alpha_2}{2}+\frac{2L\gamma_{2}^2}{\lambda }+2Z\right ) \mathbb{E}\|z_{m-1}-z_{m-2}\|^2\right )^{1-\vartheta }\\
&+aK_3\left (\left( \frac{V_1+V_\Upsilon /\rho }{L\lambda } +Z\right ) \mathbb{E}\|z_{m-2}-z_{m-3}\|^2\right )^{1-\vartheta }.
\endaligned  
\end{equation}

(i) If $\vartheta = \frac{1} {2}$, then $\left (\mathbb{E}\left \| \xi _m \right \| \right )^{\frac{1-\vartheta}{\vartheta }  }=\mathbb{E}\left \| \xi _m \right \|$. Equation \eqref{(3.34)} then gives 
\begin{equation}
\label{(3.39)}
\aligned
T_m\le&\sqrt{\mathbb{E}\left \| z_{m}-z_{m-1} \right \|^2}+\sqrt{\mathbb{E}\left \| z_{m-1}-z_{m-2} \right \|^2}+\sqrt{\mathbb{E}\left \| z_{m-2}-z_{m-3} \right \|^2}+\sqrt{\mathbb{E}\left \| z_{m-3}-z_{m-4} \right \|^2} \\
&+\frac{2\sqrt{s} }{K_1\rho } \sqrt{\mathbb{E}\Upsilon _{m-1} }+aK_4\left ( p\left ( \sqrt{\mathbb{E}\left \| z_{m}-z_{m-1} \right \|^2}+\sqrt{\mathbb{E}\left \| z_{m-1}-z_{m-2} \right \|^2} \right .\right .\\
&\left.\left.+\sqrt{\mathbb{E}\left \| z_{m-2}-z_{m-3} \right \|^2}+\sqrt{\mathbb{E}\left \| z_{m-3}-z_{m-4} \right \|^2} \right )+\sqrt{s\mathbb{E}\Upsilon _{m-1} } \right )+aK_3\sqrt{\frac{1}{L\lambda\rho  }}\sqrt{ \mathbb{E}\Upsilon _m}\\
&+\left (aK_3\sqrt{ \frac{V_1+V_\Upsilon /\rho }{L\lambda }+\frac{\alpha_1+\alpha_2}{2}+\frac{2L(\gamma_{1}^2+\gamma_{2}^2)}{\lambda }+3Z }\right )\sqrt{ \mathbb{E}\|z_{m}-z_{m-1}\|^2}\\
&+\left (aK_3\sqrt{ \frac{V_1+V_\Upsilon /\rho }{L\lambda }+\frac{\alpha_2}{2}+\frac{2L\gamma_{2}^2}{\lambda }+2Z }\right )\sqrt{ \mathbb{E}\|z_{m-1}-z_{m-2}\|^2}\\
&+\left (aK_3\sqrt{ \frac{V_1+V_\Upsilon /\rho }{L\lambda } +Z }\right )\sqrt{ \mathbb{E}\|z_{m-2}-z_{m-3}\|^2}\\
\le&\left ( 1+aK_5\left ( p+\sqrt{ \frac{V_1+V_\Upsilon /\rho }{L\lambda }+\frac{\alpha_1+\alpha_2}{2}+\frac{2L(\gamma_{1}^2+\gamma_{2}^2)}{\lambda }+3Z }  \right )  \right ) \left ( \sqrt{\mathbb{E}\left \| z_{m}-z_{m-1} \right \|^2}\right .\\
&\left.+\sqrt{\mathbb{E}\left \| z_{m-1}-z_{m-2} \right \|^2} +\sqrt{\mathbb{E}\left \| z_{m-2}-z_{m-3} \right \|^2}+\sqrt{\mathbb{E}\left \| z_{m-3}-z_{m-4} \right \|^2}\right )\\
&+\left ( \frac{2\sqrt{s} }{K_1\rho }+aK_5 \sqrt{s}\right )\sqrt{\mathbb{E}\Upsilon _{m-1} } +aK_5\sqrt{\frac{1}{L\lambda\rho  }}\sqrt{ \mathbb{E}\Upsilon _m},
\endaligned  
\end{equation}
where $K_5=\max\left \{ K_3,K_4 \right \} $. Using \eqref{(3.18)}, we have that, for any constant $c > 0$,
\begin{equation*}
\aligned
0\le&-c\sqrt{\mathbb{E}\Upsilon _{k}}+c(1-\frac{\rho}{2} )\sqrt{\mathbb{E}\Upsilon _{k-1}} +c\sqrt{V_\Upsilon} \left ( \sqrt{  \mathbb{E}\left \| z_{k}-z_{k-1} \right \| ^{2}} +\sqrt{\mathbb{E}\left \| z_{k-1}-z_{k-2} \right \| ^{2} } \right .\\
&\left.+\sqrt{\mathbb{E}\left \| z_{k-2}-z_{k-3} \right \| ^{2}} +\sqrt{\mathbb{E}\left \| z_{k-3}-z_{k-4} \right \| ^{2}} \right ).
\endaligned
\end{equation*}
Combining this inequality with \eqref{(3.39)},
\begin{equation*}
\aligned
T_m\le&\left ( 1+aK_5\left ( p+\sqrt{ \frac{V_1+V_\Upsilon /\rho }{L\lambda }+\frac{\alpha_1+\alpha_2}{2}+\frac{2L(\gamma_{1}^2+\gamma_{2}^2)}{\lambda }+3Z }+c\sqrt{V_\Upsilon}  \right )  \right ) \left ( \sqrt{\mathbb{E}\left \| z_{m}-z_{m-1} \right \|^2} \right .\\
&\left.+\sqrt{\mathbb{E}\left \| z_{m-1}-z_{m-2} \right \|^2}+\sqrt{\mathbb{E}\left \| z_{m-2}-z_{m-3} \right \|^2}+\sqrt{\mathbb{E}\left \| z_{m-3}-z_{m-4} \right \|^2}\right )\\
&+c\left ( 1-\frac{\rho}{2}+\frac{2\sqrt{s} }{K_1\rho c }+\frac{aK_5 \sqrt{s}}{c}\right )\sqrt{\mathbb{E}\Upsilon _{m-1} } -c\left (1-\frac{aK_5 }{c}\sqrt{\frac{1}{L\lambda\rho  }}\right )\sqrt{ \mathbb{E}\Upsilon _m}.
\endaligned  
\end{equation*}
Defining $A=1+aK_5\left ( p+\sqrt{ \frac{V_1+V_\Upsilon /\rho }{L\lambda }+\frac{\alpha_1+\alpha_2}{2}+\frac{2L(\gamma_{1}^2+\gamma_{2}^2)}{\lambda }+3Z }+c\sqrt{V_\Upsilon}\right )$, we have shown
\begin{equation*}
\aligned
&T_m+c\left (1-\frac{aK_5 }{c}\sqrt{\frac{1}{L\lambda\rho  }}\right )\sqrt{ \mathbb{E}\Upsilon _m}\\
\le&A \left (T_{m-1}-T_m\right )+c\left ( 1-\frac{\rho}{2}+\frac{2\sqrt{s} }{K_1\rho c }+\frac{aK_5 \sqrt{s}}{c}\right )\sqrt{\mathbb{E}\Upsilon _{m-1} }.
\endaligned  
\end{equation*}
Then, we get
\begin{equation*}
\aligned
&(1+A)T_m+c\left (1-\frac{aK_5 }{c}\sqrt{\frac{1}{L\lambda\rho  }}\right )\sqrt{ \mathbb{E}\Upsilon _m}\\
\le& AT_{m-1}+c\left ( 1-\frac{\rho}{2}+\frac{2\sqrt{s} }{K_1\rho c }+\frac{aK_5 \sqrt{s}}{c}\right )\sqrt{\mathbb{E}\Upsilon _{m-1} }.
\endaligned  
\end{equation*}
This implies
\begin{equation*}
\aligned
&T_m+\sqrt{ \mathbb{E}\Upsilon _m}\\
\le&\max\left \{ \frac{A}{1+A},\left ( 1-\frac{\rho}{2}+\frac{2\sqrt{s} }{K_1\rho c }+\frac{aK_5 \sqrt{s}}{c}\right )\left (1-\frac{aK_5 }{c}\sqrt{\frac{1}{L\lambda\rho  }}\right )^{-1}  \right \}  \left (T_{m-1}+\sqrt{\mathbb{E}\Upsilon _{m-1} }\right ).
\endaligned  
\end{equation*}
For large $c$, the second coefficient in the above expression approaches $1-\frac{\rho}{2}$. So there exist $\tau \in [1 - \rho,1)$ such that  
$$\sum_{k=m}^{\infty}\sqrt{\mathbb{E}\left \| z_{k}-z_{k-1} \right \|^2}\le\tau ^k\left (T_{0}+\sqrt{\mathbb{E}\Upsilon _{0} }\right ) \le d_1\tau ^k$$
for some constnt $d_1$. Then using the fact that 
$\mathbb{E}\left \| z_{m}-z^{\ast} \right \|=\mathbb{E}\left \|\sum_{k=m+1}^{\infty} (z_{k}-z_{k-1}) \right \|\le \sum_{k=m}^{\infty}\mathbb{E}\left \|  z_{k}-z_{k-1} \right \|$, we proves claim (i).

(ii) Suppose $\vartheta \in (\frac{1}{2} ,1)$. Each term on the right side of \eqref{(3.34)} converges to zero, but at different rates. Because 
\begin{equation*}
\aligned
\Theta_m =& \mathcal{O}\left ( \sqrt{\mathbb{E}\left \| z_{m}-z_{m-1} \right \|^2}+\sqrt{\mathbb{E}\left \| z_{m-1}-z_{m-2} \right \|^2} +\sqrt{\mathbb{E}\left \| z_{m-2}-z_{m-3} \right \|^2}\right.\\
&\left.+\sqrt{\mathbb{E}\left \| z_{m-3}-z_{m-4} \right \|^2}+\sqrt{s\mathbb{E}\Upsilon _{m-1} }   \right ), 
\endaligned  
\end{equation*}
and $\vartheta$ satisfies $\frac{1-\vartheta}{\vartheta }< 1$, the term $\Theta _{m}^{\frac{1-\vartheta}{\vartheta } }$ dominates the first five terms on the right side of \eqref{(3.34)} for large $m$. Also, because $\frac{1-\vartheta}{2\vartheta }< 1-\vartheta$, $\Theta _{m}^{\frac{1-\vartheta}{\vartheta } }$ dominates the final four terms as well. Combining these facts, there exists a natural number $M_1$ such that for all
$m \ge M_1$,
\begin{equation}
\label{(3.35)}
\aligned
T_m\le P\Theta _m
\endaligned
\end{equation}
for some constant $P>(aK_3)^{\frac{\vartheta}{1-\vartheta } }$. The bound of \eqref{(5.18)} implies
\begin{equation*}
\aligned
&2\sqrt{s\mathbb{E}\Upsilon _{m-1}}\\
\le& \frac{4\sqrt{s} }{\rho} \left ( \sqrt{\mathbb{E}\Upsilon _{m-1}}-\sqrt{\mathbb{E}\Upsilon _{m}}   +\sqrt{V_\Upsilon}\left ( \sqrt{  \mathbb{E}\left \| z_{m}-z_{m-1} \right \| ^{2}} +\sqrt{\mathbb{E}\left \| z_{m-1}-z_{m-2} \right \| ^{2} }\right . \right .\\ 
&\left. \left.+\sqrt{\mathbb{E}\left \| z_{m-2}-z_{m-3} \right \| ^{2}} +\sqrt{\mathbb{E}\left \| z_{m-3}-z_{m-4} \right \| ^{2}} \right )\right ).
\endaligned
\end{equation*}
Therefore,
\begin{equation}
\label{(3.36)}
\aligned
\Theta_m = &p\left ( \sqrt{\mathbb{E}\left \| z_{m}-z_{m-1} \right \|^2}+\sqrt{\mathbb{E}\left \| z_{m-1}-z_{m-2} \right \|^2} +\sqrt{\mathbb{E}\left \| z_{m-2}-z_{m-3} \right \|^2} \right .\\ 
&\left.+\sqrt{\mathbb{E}\left \| z_{m-3}-z_{m-4} \right \|^2} \right )+\left (2\sqrt{s\mathbb{E}\Upsilon _{m-1} }-\sqrt{s\mathbb{E}\Upsilon _{m-1} } \right )\\
\le &\left ( p+ \frac{4\sqrt{sV_\Upsilon } }{\rho}\right ) \left ( \sqrt{\mathbb{E}\left \| z_{m}-z_{m-1} \right \|^2}+\sqrt{\mathbb{E}\left \| z_{m-1}-z_{m-2} \right \|^2} +\sqrt{\mathbb{E}\left \| z_{m-2}-z_{m-3} \right \|^2} \right .\\ 
&\left.+\sqrt{\mathbb{E}\left \| z_{m-3}-z_{m-4} \right \|^2} \right )+\frac{4\sqrt{s} }{\rho}\left (\sqrt{\mathbb{E}\Upsilon _{m-1} }-\sqrt{\mathbb{E}\Upsilon _{m} } \right )-\sqrt{s\mathbb{E}\Upsilon _{m-1} }.\\
\endaligned
\end{equation}
Furthermore, because ${\frac{\vartheta}{1-\vartheta } }>1$ and $\mathbb{E}\Upsilon _m\to 0$, for large enough $m$, we have $\left ( \sqrt{\mathbb{E}\Upsilon _m} \right )^{\frac{\vartheta}{1-\vartheta } }  \ll \sqrt{\mathbb{E}\Upsilon _m} $. This ensures that there exists a natural number $M_2$ such that for every $m \ge M_2$,
\begin{equation}
\label{(3.37)}
\left (\frac{4\sqrt{s} (1-\rho /4)}{\rho (p+4\sqrt{sV_\Upsilon}  /\rho )}  \sqrt{\mathbb{E}\Upsilon _m} \right )^{\frac{\vartheta}{1-\vartheta } }  \le P\sqrt{s\mathbb{E}\Upsilon _m} .
\end{equation}
The constant appearing on the left was chosen to simplify later arguments. Therefore, \eqref{(3.35)} implies
\begin{equation*}
\aligned
&\left (  T_m+\frac{4\sqrt{s} (1-\rho /4)}{\rho (p+4\sqrt{sV_\Upsilon}  /\rho )}  \sqrt{\mathbb{E}\Upsilon _m}\right )^{\frac{\vartheta}{ 1-\vartheta} }\\
\overset{(1)}{\le } &\frac{2^{{\frac{\vartheta}{1-\vartheta } }}}{2}\left ( T_m\right )^{\frac{\vartheta}{ 1-\vartheta} }+\frac{2^{{\frac{\vartheta}{1-\vartheta } }}}{2}\left (  \frac{4\sqrt{s} (1-\rho /4)}{\rho (p+4\sqrt{sV_\Upsilon}  /\rho )}  \sqrt{\mathbb{E}\Upsilon _m}\right )^{\frac{\vartheta}{ 1-\vartheta} }
\overset{(2)}{\le } \frac{2^{{\frac{\vartheta}{1-\vartheta } }}}{2}\left (  T_m\right )^{\frac{\vartheta}{ 1-\vartheta} }+\frac{2^{{\frac{\vartheta}{1-\vartheta } }}}{2}\left ( P\sqrt{s\mathbb{E}\Upsilon _m}\right )\\
\overset{(3)}{\le } &\frac{2^{{\frac{\vartheta}{1-\vartheta } }}}{2}\left ( P\left ( p+ \frac{4\sqrt{sV_\Upsilon } }{\rho}\right ) \left ( \sqrt{\mathbb{E}\left \| z_{m}-z_{m-1} \right \|^2}+\sqrt{\mathbb{E}\left \| z_{m-1}-z_{m-2} \right \|^2} +\sqrt{\mathbb{E}\left \| z_{m-2}-z_{m-3} \right \|^2} \right . \right .\\
&\left.\left.+\sqrt{\mathbb{E}\left \| z_{m-3}-z_{m-4} \right \|^2} \right )+\frac{4\sqrt{s}P }{\rho}\left (\sqrt{\mathbb{E}\Upsilon _{m-1} }-\sqrt{\mathbb{E}\Upsilon _{m} } \right )-P\sqrt{s\mathbb{E}\Upsilon _{m-1} }  \right )+\frac{2^{{\frac{\vartheta}{1-\vartheta } }}}{2}\left ( P\sqrt{s\mathbb{E}\Upsilon _m}\right )\\
\le&\frac{2^{{\frac{\vartheta}{1-\vartheta } }}}{2}\left ( P\left ( p+ \frac{4\sqrt{sV_\Upsilon } }{\rho}\right ) \left ( \sqrt{\mathbb{E}\left \| z_{m}-z_{m-1} \right \|^2}+\sqrt{\mathbb{E}\left \| z_{m-1}-z_{m-2} \right \|^2} +\sqrt{\mathbb{E}\left \| z_{m-2}-z_{m-3} \right \|^2} \right . \right .\\
&\left.\left.+\sqrt{\mathbb{E}\left \| z_{m-3}-z_{m-4} \right \|^2} \right )+\frac{4\sqrt{s}P(1-\rho/4) }{\rho}\left (\sqrt{\mathbb{E}\Upsilon _{m-1} }-\sqrt{\mathbb{E}\Upsilon _{m} } \right ) \right ).\\
\endaligned
\end{equation*}
Here, (1) follows by convexity of the function $x^{\frac{\vartheta}{1-\vartheta }}$ for $\vartheta \in [1/2, 1)$ and $x \ge  0$, (2) is \eqref{(3.37)}, and (3) is \eqref{(3.35)} combined with \eqref{(3.36)}. We absorb the constant $\frac{2^{{\frac{\vartheta}{1-\vartheta } }}}{2}$ into $P$. Define
\begin{equation*}
\aligned
S_m=T_m+\frac{4\sqrt{s} (1-\rho /4)}{\rho (p+4\sqrt{sV_\Upsilon}  /\rho )}  \sqrt{\mathbb{E}\Upsilon _m}.
\endaligned
\end{equation*}
$S_m$ is bounded for all $m$ because $\sum_{k=m}^{\infty} \sqrt{\mathbb{E}\left \| z_{k+1}-z_{k} \right \|^2}$ is bounded by \eqref{(3.32)}. Hence, we have shown
\begin{equation}
\label{(3.38)}
S_{m}^{\frac{\vartheta}{1-\vartheta }} \le  P\left ( p+ \frac{4\sqrt{sV_\Upsilon } }{\rho}\right )(S_{m-1}-S_m).
\end{equation}
The rest of the proof is almost the same as it mentioned in \cite{AB,DT}. We omit the proof here.

(iii) When $\vartheta = 0$, the K{\L} property \eqref{(2.011)} implies that exactly one of the following two scenarios holds: either $\mathbb{E} \Phi (z_k)\ne \Phi _{k}^{\ast }$ and
\begin{equation}
\label{(3.50)}
0<C\le\mathbb{E}\left \| \xi _k \right \| ,\ \ \forall \xi _k\in \partial \Phi (z_k)
\end{equation}
or $\mathbb{E} \Phi (z_k)= \Phi _{k}^{\ast }$. We show that the above inequality can hold only for a finite number of iterations.
\par Using the subgradient bound \eqref{(3.9)}, the first scenario implies
\begin{equation*}
\aligned
C^2\le&\left ( \mathbb{E}\left \| \xi _k \right \|  \right )^2 \\
\le &\left ( p\left (\mathbb{E}\left \| z_{k}-z_{k-1} \right \|+\mathbb{E}\left \| z_{k-1}-z_{k-2} \right \| +\mathbb{E}\left \| z_{k-2}-z_{k-3} \right \|+\mathbb{E}\left \| z_{k-3}-z_{k-4} \right \|\right )+\Gamma_{k-1} \right )^2\\
\le&5p^2 \left ( \mathbb{E}\left \| z_{k}-z_{k-1} \right \|  \right ) ^2+5p^2 \left ( \mathbb{E}\left \| z_{k-1}-z_{k-2} \right \|  \right ) ^2+5p^2 \left ( \mathbb{E}\left \| z_{k-2}-z_{k-3} \right \|  \right ) ^2\\
&+5p^2 \left ( \mathbb{E}\left \| z_{k-3}-z_{k-4} \right \|  \right ) ^2+5(\mathbb{E} \Gamma_{k-1})^2\\
\le&5p^2 \left ( \mathbb{E}\left \| z_{k}-z_{k-1} \right \|  \right ) ^2+5p^2 \left ( \mathbb{E}\left \| z_{k-1}-z_{k-2} \right \|  \right ) ^2+5p^2 \left ( \mathbb{E}\left \| z_{k-2}-z_{k-3} \right \|  \right ) ^2\\
&+5p^2 \left ( \mathbb{E}\left \| z_{k-3}-z_{k-4} \right \|  \right ) ^2+5s\mathbb{E} \Upsilon_{k-1},
\endaligned
\end{equation*}
where we have used the inequality $(a_1+a_2+\cdots +a_s)^2\le s (a_1^2+a_2^2+\cdots +a_s^2)$ and Jensen's inequality. Applying this inequality to the decrease of $\Psi_ k$ \eqref{(3.2)}, we obtain
\begin{equation*}
\aligned
&\mathbb{E}_k\Psi _{k} \\
\le& \mathbb{E}_k\Psi _{k-1}-\kappa \left \| z_{k+1}-z_{k} \right \|^2-\epsilon \left \| z_{k}-z_{k-1} \right \|^2- \epsilon \left \| z_{k-1}-z_{k-2} \right \|^2- Z \left \| z_{k-2}-z_{k-3} \right \|^2\\
\le&\mathbb{E}_k\Psi _{k-1}-C^2+\mathcal{O}\left (  \left \| z_{k+1}-z_{k} \right \|^2 \right )+\mathcal{O} \left ( \left \| z_{k}-z_{k-1} \right \|^2 \right ) +\mathcal{O} \left ( \left \| z_{k-1}-z_{k-2} \right \|^2 \right )\\
&+\mathcal{O} \left ( \left \| z_{k-2}-z_{k-3} \right \|^2 \right )+\mathcal{O} \left ( \mathbb{E} \Upsilon_{k-1} \right ) 
\endaligned
\end{equation*}
for some constant $C^2$. Because the final five terms go to zero as $k \to \infty$, there exists an index $M_4$ so that the sum of these five terms is bounded above by $\frac{C^2}{2}$ for all $k \ge M_4$. Therefore,
$$\mathbb{E}_k\Psi _{k}\le\mathbb{E}_k\Psi-\frac{C^2}{2},\ \ \forall k\ge M_4.$$
Because $\Psi_k$ is bounded below for all $k$, this inequality can only hold for $N < \infty$ steps. After $N$ steps, it is no longer possible for the bound \eqref{(3.50)} to hold, so it must be that $\mathbb{E} \Phi (z_k)= \Phi _{k}^{\ast }$. Because $\Phi _{k}^{\ast }<\Phi(z^{\ast })$, $\Phi _{k}^{\ast }<\mathbb{E} \Phi (z_k)$, and both $\mathbb{E} \Phi (z_k)$, $\Phi _{k}^{\ast }$ converge to $\mathbb{E}\Phi(z^{\ast })$, we must have $\Phi _{k}^{\ast }=\mathbb{E} \Phi (z_k)=\mathbb{E}\Phi(z^{\ast })$.
\end{proof}
\section{Numerical experiments}\label{sect5}
\ \par In this section, to demonstrate the advantages of STiBPALM (Algorithm \ref{alg1}), we present our numerical study on the practical performance of the proposed STiBPALM with three different stochastic gradient estimators, i.e. SGD estimator \cite{RS} (STiBPALM-SGD), SAGA gradient \cite{AFS} estimator (STiBPALM-SAGA), and SARAH gradient \cite{LM} estimator (STiBPALM-SARAH), compared with PALM \cite{JSM}, iPALM \cite{TS}, TiPALM \cite{GZ}, SPRING \cite{DT} and SiPALM \cite{JG} algorithms. We refer to SPRING with SGD, SAGA, and SARAH gradient estimators as SPRING-SGD, SPRING-SAGA, and SPRING-SARAH; SiPALM using the SGD, SAGA, and SARAH gradient estimators as SiPALM-SGD, SiPALM-SAGA, and SiPALM-SARAH, respectively. Two applications are considered here for comparison: sparse nonnegative matrix factorization (S-NMF) and blind image-deblurring (BID). 
\par Since the proposed algorithm is based on the stochastic gradient estimator, we report the average results (over 10 independent runs) of objective values for all algorithms. The initial point is also the same for all algorithms. In addition, we choose step-size which is suggested in \cite{JSM} for PALM and in \cite{TS} for iPALM, respectively and the same step-size based on \cite{DT} for all stochastic algorithms for simplicity.
\subsection{Sparse nonnegative matrix factorization}\label{sect51}
\ \par Given a matrix $A$, sparse nonnegative matrix factorization (S-NMF) \cite{DH,JN,FF} problem can be formulated as the following model
\begin{equation}
\label{(5.1)}
\aligned
\underset{X,Y}{\min}  \left \{ \frac{\eta}{2}\left \| A-XY \right \| _{F}^{2} : \ X,Y\ge 0,\ \left \| X_i \right \| _0\le s,\ i=1,2,\dots ,r\right \}.
\endaligned
\end{equation}
In dictionary learning and sparse coding, $X$ is called the learned dictionary with coefficients $Y$. In this formulation, the sparsity on $X$ is restricted $75\%$ of the entries to be 0.
\par We use the extended Yale-B dataset and the ORL dataset, which are standard facial recognition benchmarks consisting of human face images\footnote{ \href{http://www.cad.zju.edu.cn/home/dengcai/Data/FaceData.html}{http://www.cad.zju.edu.cn/home/dengcai/Data/FaceData.html.}}. For solving this S-NMF problem \eqref{(5.1)}, \cite{GCH,TS} gave the details on how to solve the X-subproblems and Y-subproblems. The extended Yale-B dataset contains 2414 cropped images of size $32 \times 32$, while the ORL dataset contains 400 images sized $64 \times 64$, see Figure \ref{fig1}. In the experiment for the Yale dataset, we extract 49 sparse basis images for the dataset. For the ORL dataset we extract 25 sparse basis images. In each iteration of the stochastic algorithms, we randomly subsample $5\%$ of the full batch as a minibatch. Here for SARAH gradient estimator we set $p=\frac{1}{20}$.
\ \par In STiBPALM, let $\phi_1(X)=\frac{\theta_1 }{2} \left \| X\right \|^{2}$, $\phi_2(Y)=\frac{\theta_2 }{2} \left \|Y \right \|^{2}$. In numerical experiment, we choose $\eta=3$ and calculate $\theta_1$ and $\theta_2$ by computing the largest eigenvalues of $\eta YY^T$ and $\eta X^TX$ at $k$-th iteration, respectively. We choose $\alpha_{1k}=\beta _{1k}=\gamma_{1k}=\mu_{1k}=\frac{k-1}{k+2}$, $\alpha_{2k}=\beta _{2k}=\gamma_{2k}=\mu_{2k}=\frac{k-1}{k+2}$ in TiPALM and STiBPALM and $\alpha_{1k}=\beta_{1k}=\gamma_{1k}=\mu_{1k}=\frac{k-1}{k+2}$ in iPALM and SiPALM. We use BTiPALM and BSTiPALM to denote TiPALM and STiBPALM with $\phi_1(X)=\frac{\theta_1^2 }{4} \left \| X\right \|^{4}$, $\phi_2(Y)=\frac{\theta_2 }{2} \left \|Y \right \|^{2}$, respectively. We refer to BSTiPALM using the SGD, SAGA, and SARAH gradient estimators as BSTiPALM-SGD, BSTiPALM-SAGA, and BSTiPALM-SARAH, respectively.
\par In Figure \ref{fig2} and Figure \ref{fig4}, we report the numerical results for Yale-B dataset. A similar result for the ORL dataset is plotted in Figure \ref{fig3} and Figure \ref{fig5}. One can observe from these four figures that the STiBPALM can get slightly lower values than the other algorithms within almost the same computation time. In addition, STiBPALM can get better performance than the SPRING and SiPALM stochastic algorithm with epoch changes. 
The stochastic algorithms can improve the numerical results compared with the corresponding deterministic method. Furthermore, compared with the stochastic gradient algorithm without variance reduction (SGD), the variance reduced stochastic gradient (SAGA, SARAH) algorithm can get better numerical results.
\par The numerical results applying different Bregman distances under Yale-B dataset and ORL dataset are reported in Figure \ref{fig6} and Figure \ref{fig7}, respectively. We can observe that BSTiPALM algorithm can obtain better numerical results compared to STiBPALM algorithm, where SARAH gradient estimator can get the best performance with epoch changes. 
\par We also compare STiBPALM with SGD, SAGA, and SARAH for different sparsity settings (the value of $s$). The results of the basis images are shown in Figure \ref{fig8}. One can observe from Figure \ref{fig8} that for smaller values of $s$, the four algorithms lead to more compact representations. This might improve the generalization capabilities of the representation.
\begin{figure*}[!t]
    \centering
    \subfloat{\includegraphics[width=6.0in]{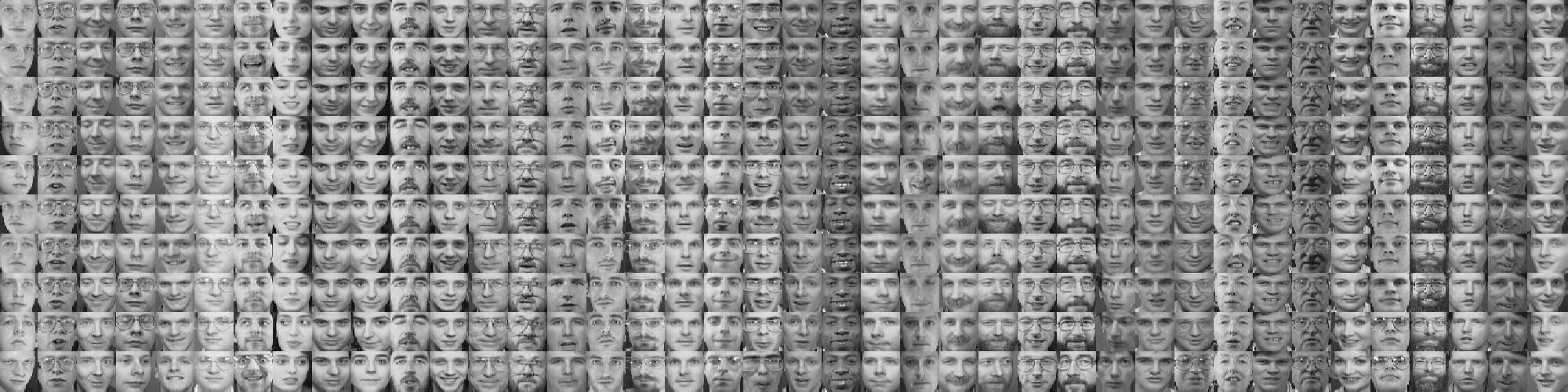}
    \label{400 normalized cropped frontal faces}}
    \caption{ORL face database which includes 400 normalized cropped frontal faces which we used in our S-NMF example.}
    \label{fig1}
\end{figure*}

\begin{figure*}[!t]
    \centering
    \subfloat[Epoch counts comparison]{\includegraphics[width=2.0in]{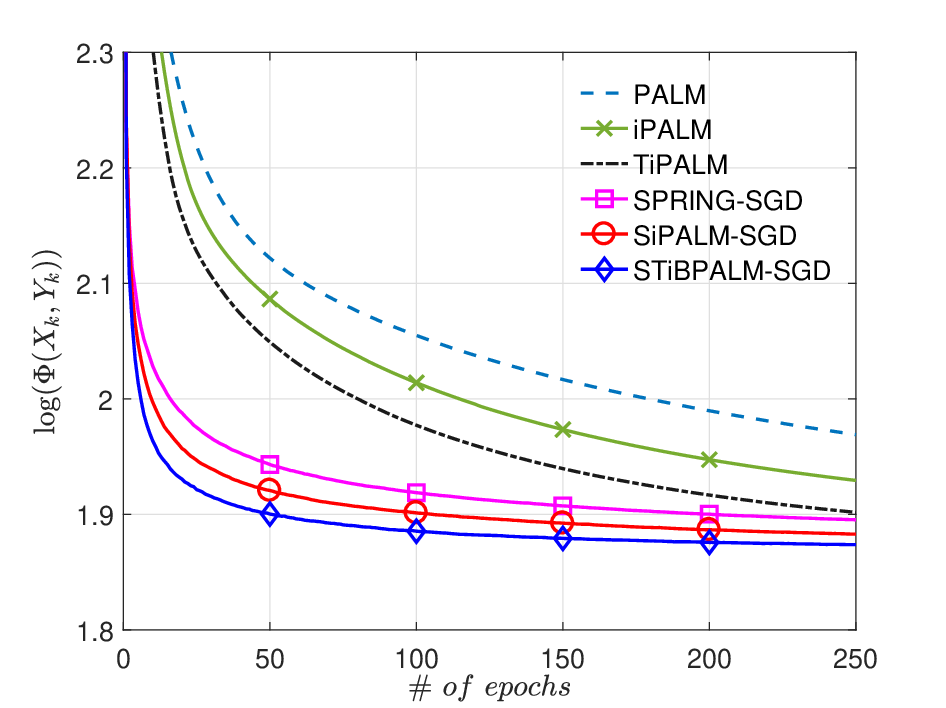}
    \label{Epoch counts comparison}}
\hfil
\subfloat[Epoch counts comparison]{\includegraphics[width=2.0in]{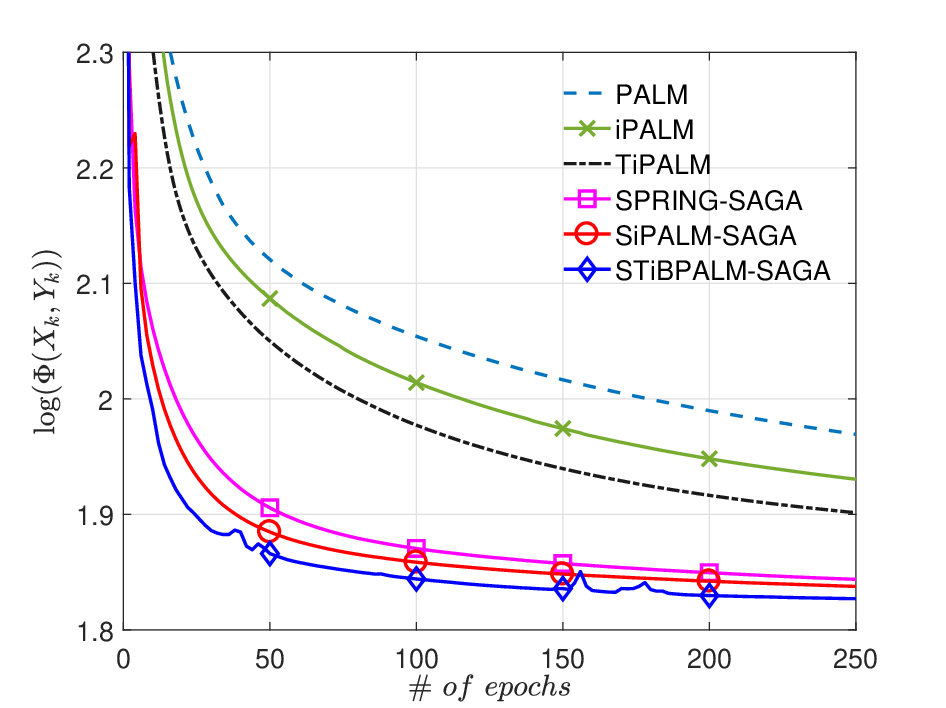}%
    \label{Wall-clock time comparison}}
\hfil
    \subfloat[Epoch counts comparison]{\includegraphics[width=2.0in]{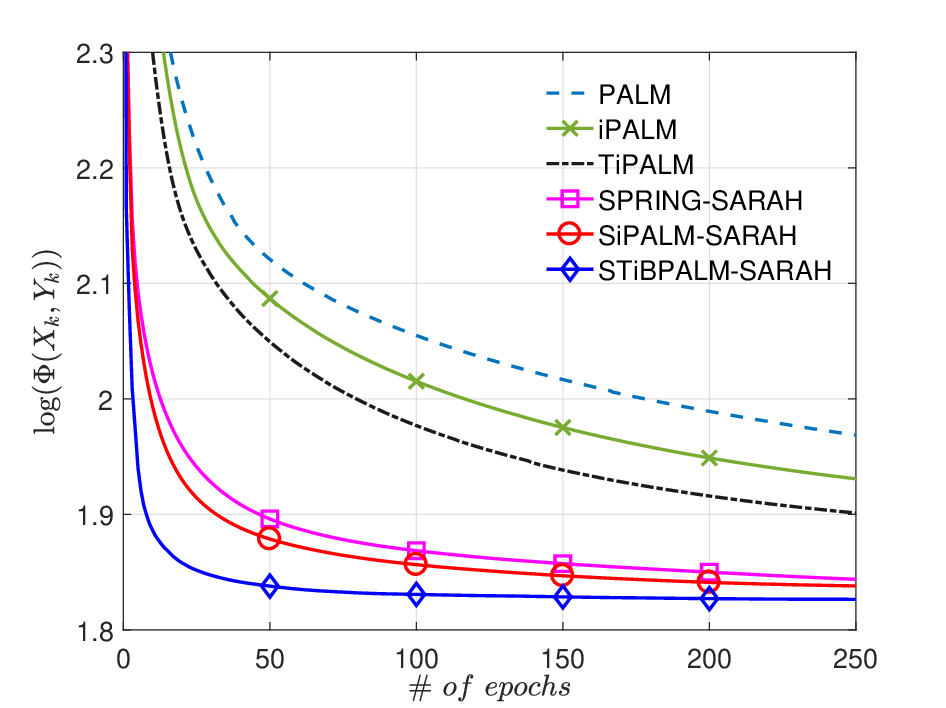}
    \label{Epoch counts comparison}}
\hfil
\subfloat[Wall-clock time comparison]{\includegraphics[width=2.0in]{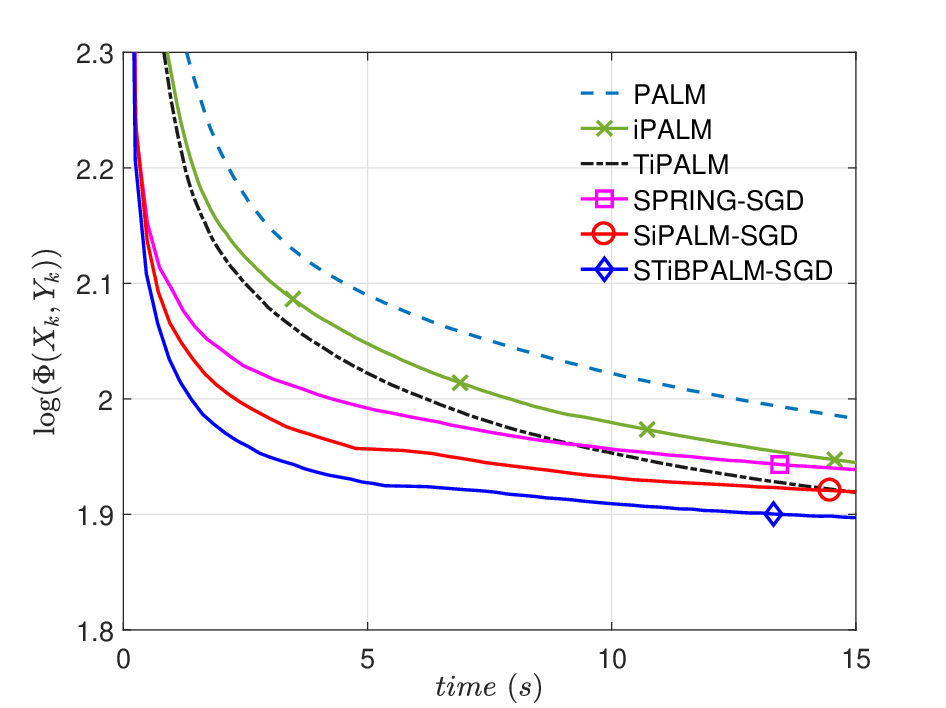}%
    \label{Wall-clock time comparison}}
\hfil
    \subfloat[Wall-clock time comparison]{\includegraphics[width=2.0in]{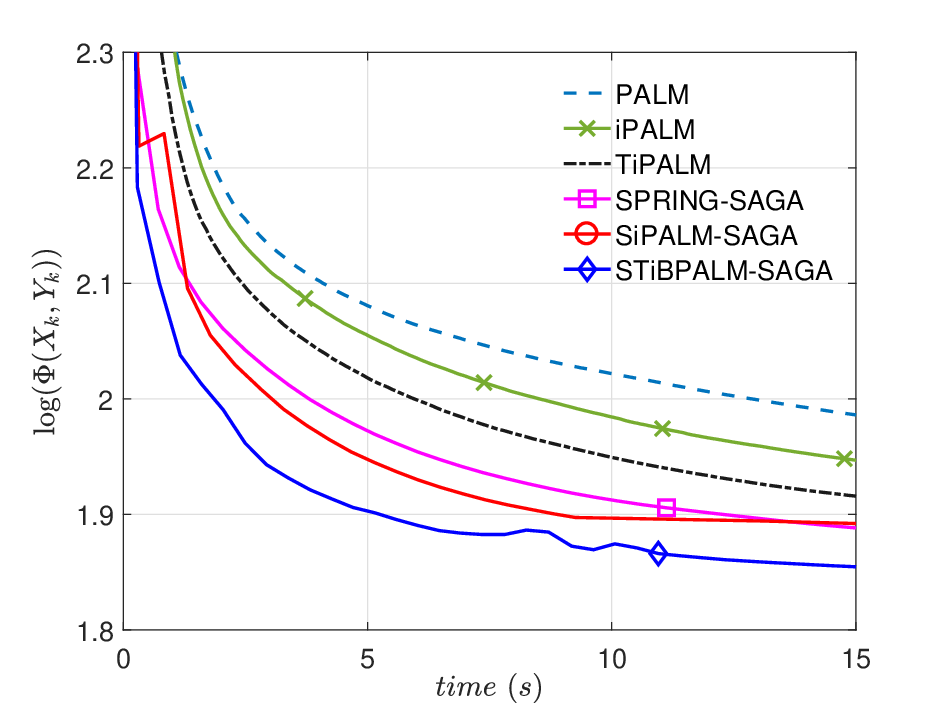}
    \label{Epoch counts comparison}}
\hfil
\subfloat[Wall-clock time comparison]{\includegraphics[width=2.0in]{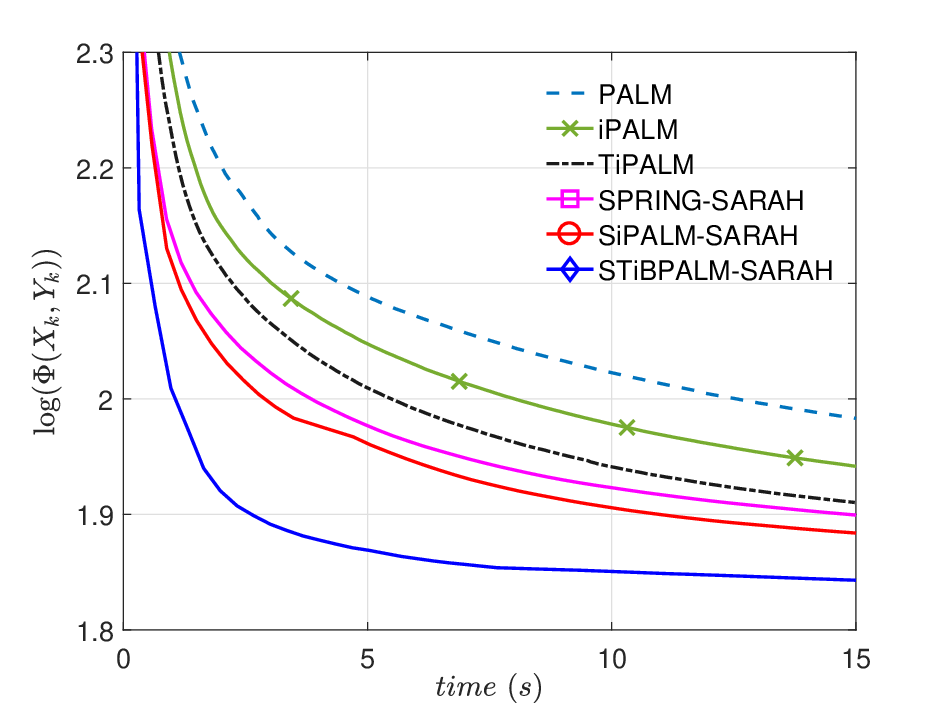}%
    \label{Wall-clock time comparison}}
    \caption{Objective decrease comparison of S-NMF with $s = 25\%$ on Yale dataset. From left column to right column are the results of SGD, SAGA and SARAH, respectively.}
    \label{fig2}
\end{figure*}

\begin{figure*}[!t]
    \centering
    \subfloat[Epoch counts comparison]{\includegraphics[width=3.0in]{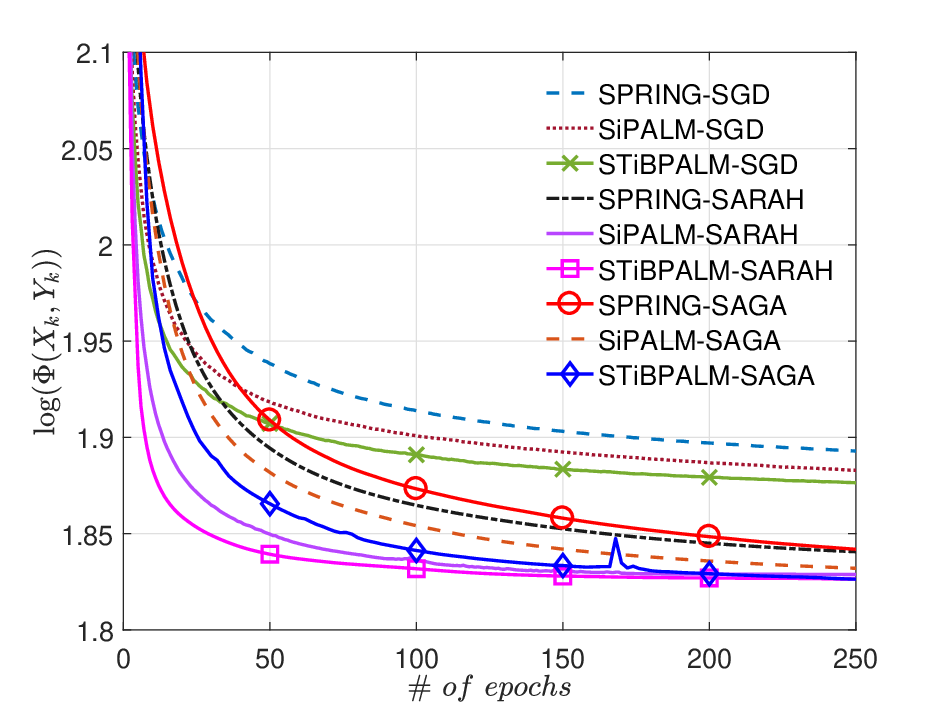}
    \label{Epoch counts comparison}}
\hfil
\subfloat[Wall-clock time comparison]{\includegraphics[width=3.0in]{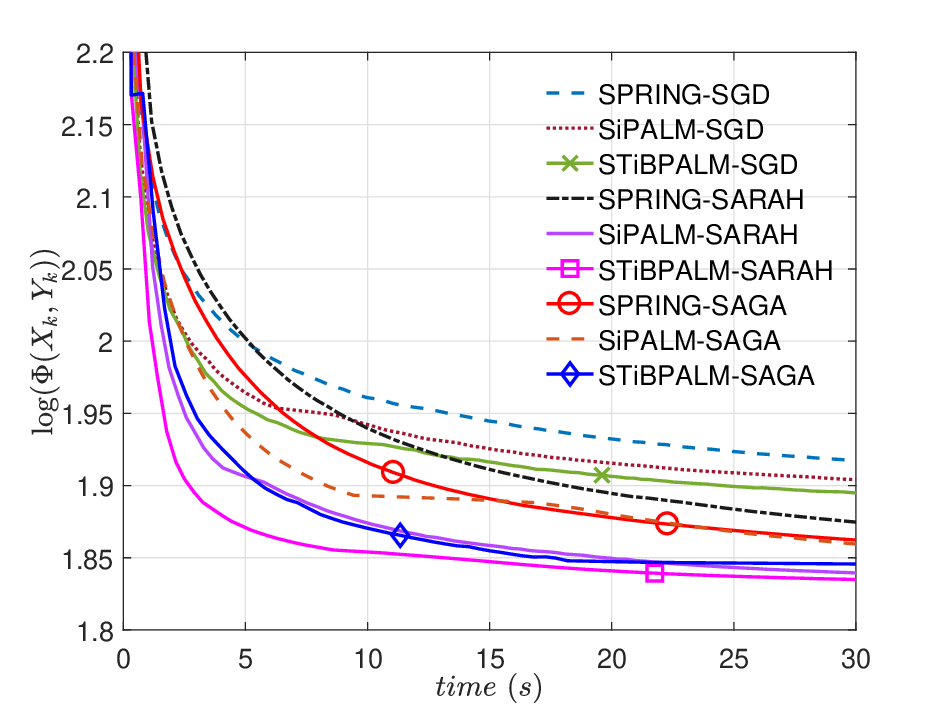}%
    \label{Wall-clock time comparison}}

    \caption{Objective decrease comparison of S-NMF with $s = 25\%$ on Yale dataset.}
    \label{fig4}
\end{figure*}

\begin{figure*}[!t]
    \centering
    \subfloat[Epoch counts comparison]{\includegraphics[width=2.0in]{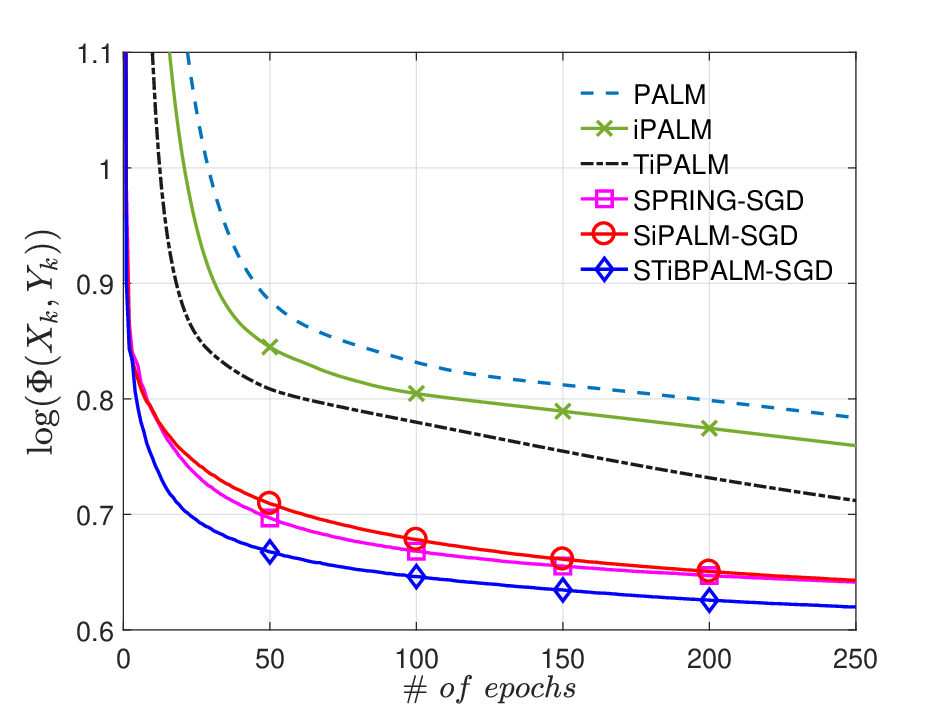}
    \label{Epoch counts comparison}}
\hfil
\subfloat[Epoch counts comparison]{\includegraphics[width=2.0in]{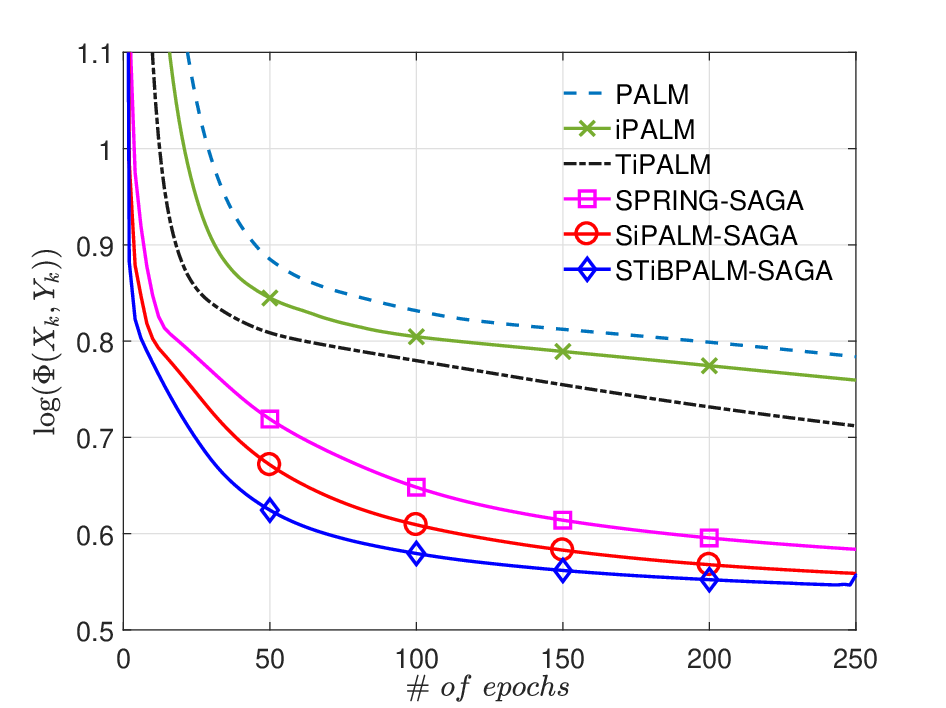}%
    \label{Wall-clock time comparison}}
\hfil
    \subfloat[Epoch counts comparison]{\includegraphics[width=2.0in]{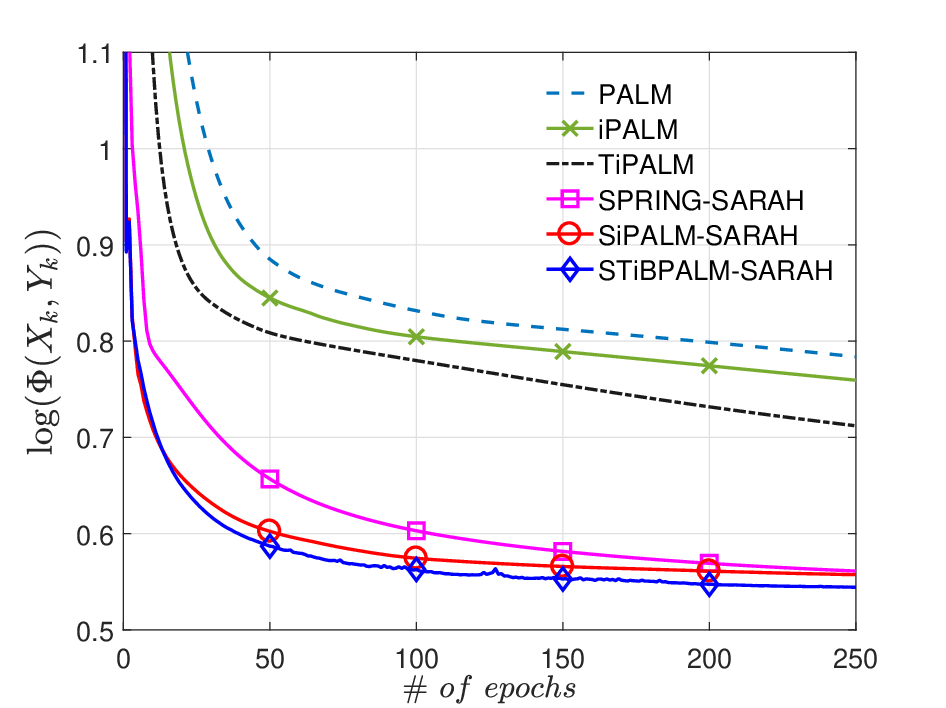}
    \label{Epoch counts comparison}}
\hfil
\subfloat[Wall-clock time comparison]{\includegraphics[width=2.0in]{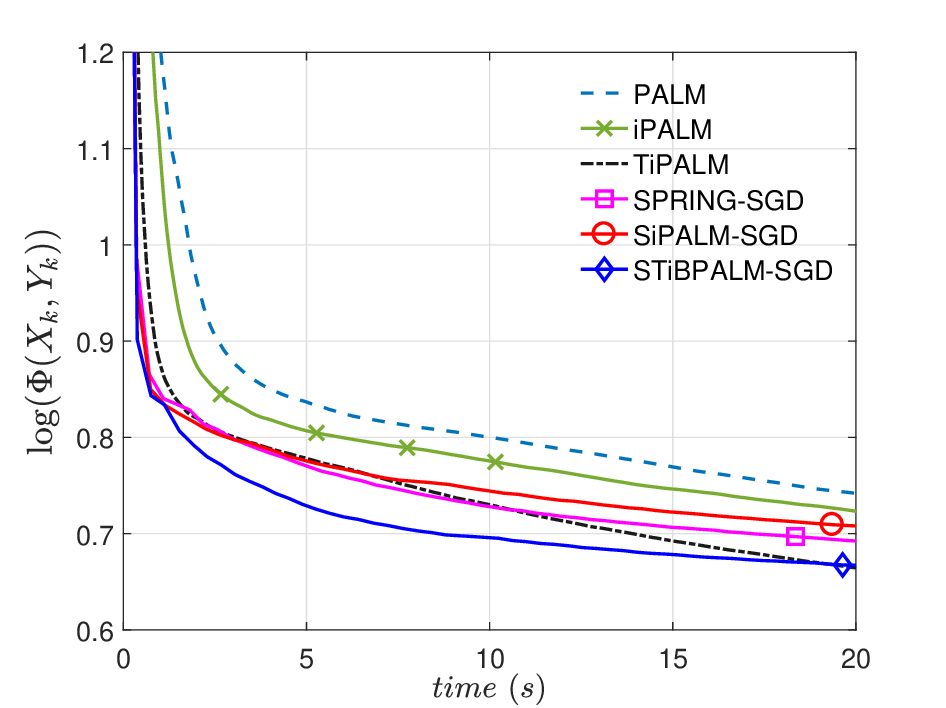}%
    \label{Wall-clock time comparison}}
\hfil
    \subfloat[Wall-clock time comparison]{\includegraphics[width=2.0in]{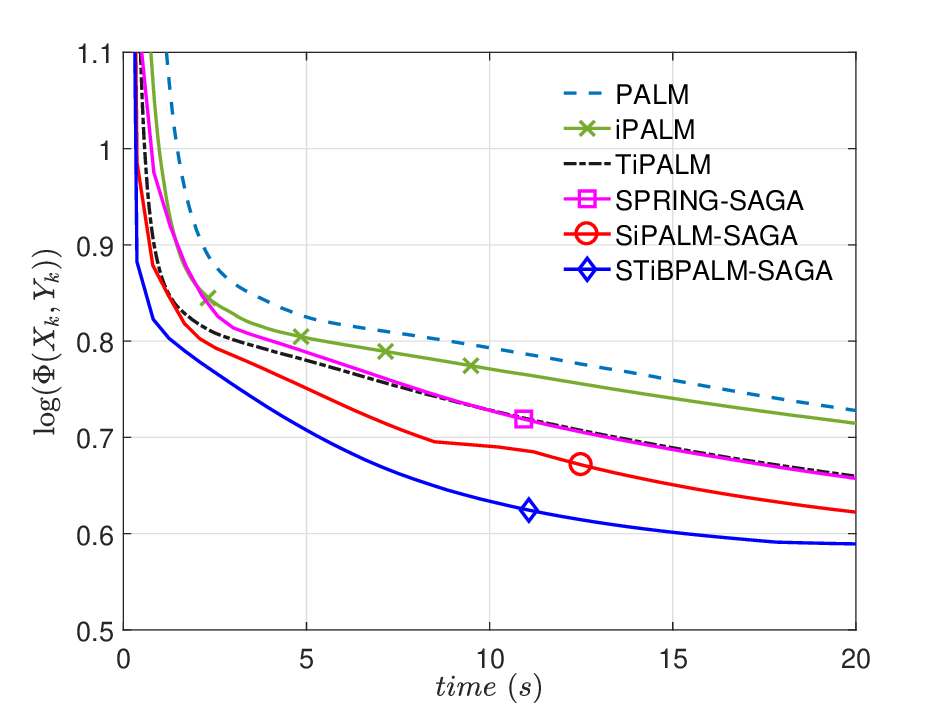}
    \label{Epoch counts comparison}}
\hfil
\subfloat[Wall-clock time comparison]{\includegraphics[width=2.0in]{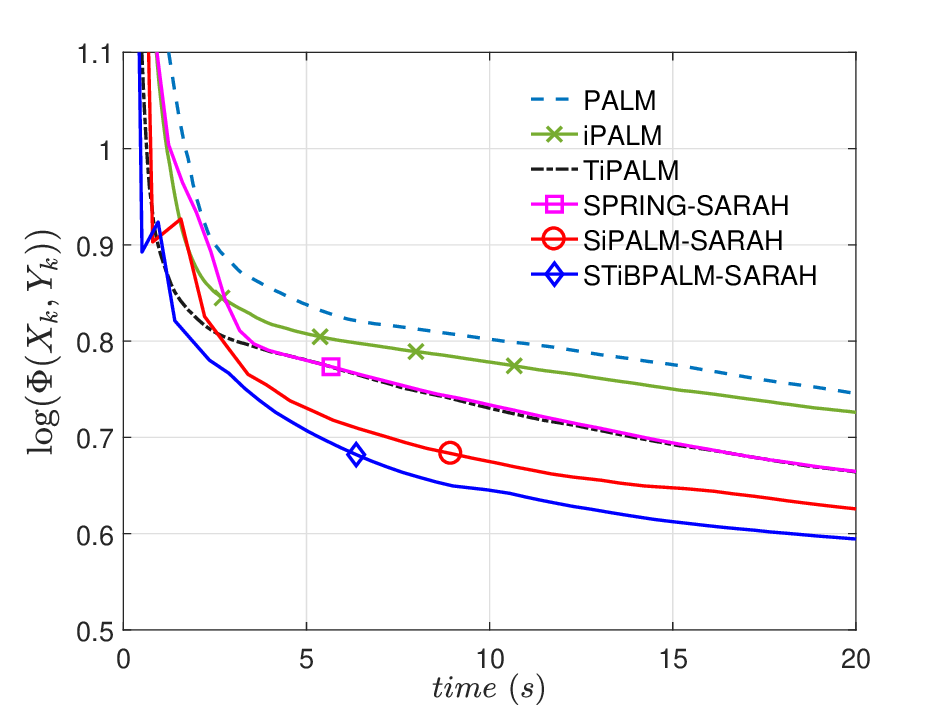}%
    \label{Wall-clock time comparison}}
    \caption{Objective decrease comparison of S-NMF with $s = 25\%$ on ORL dataset. From left column to right column are the results of SGD, SAGA and SARAH, respectively.}
    \label{fig3}
\end{figure*}

\begin{figure*}[!t]
    \centering
    \subfloat[Epoch counts comparison]{\includegraphics[width=3.0in]{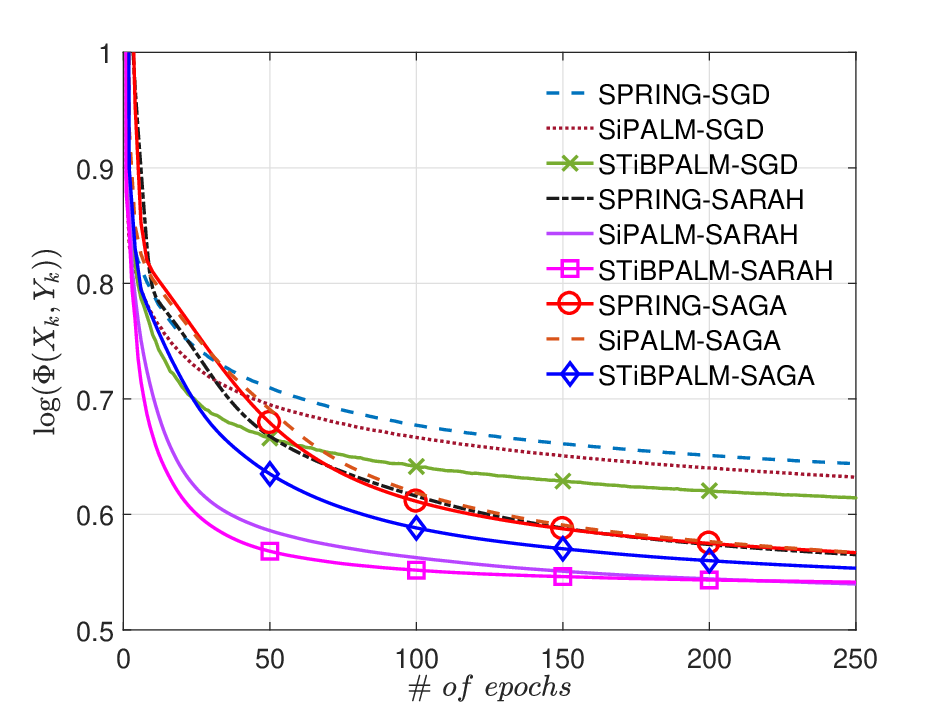}
    \label{Epoch counts comparison}}
\hfil
\subfloat[Wall-clock time comparison]{\includegraphics[width=3.0in]{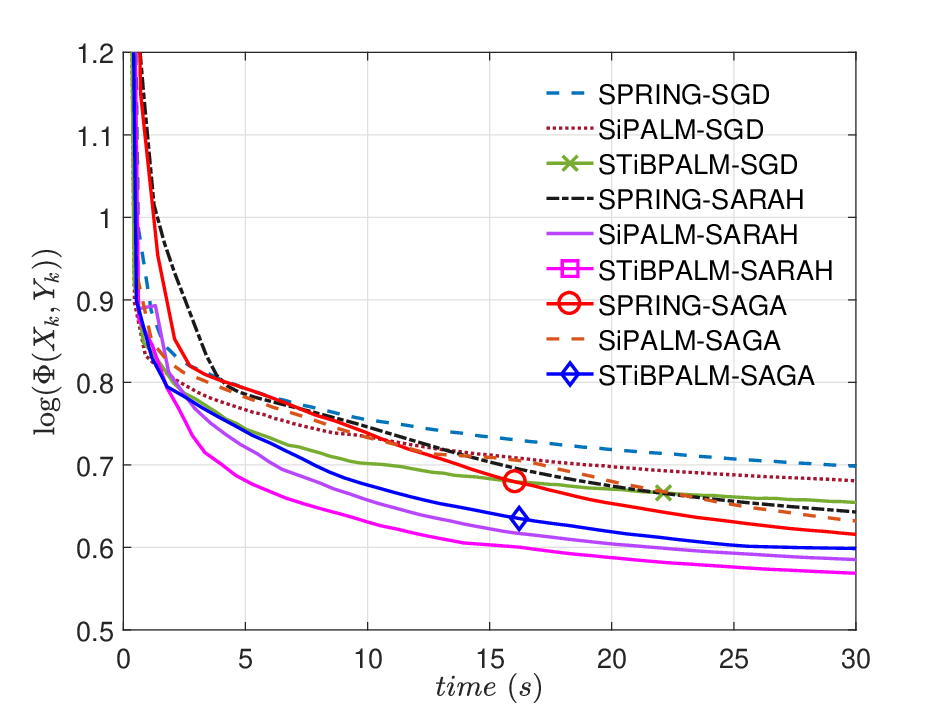}%
    \label{Wall-clock time comparison}}

    \caption{Objective decrease comparison of S-NMF with $s = 25\%$ on ORL dataset.}
    \label{fig5}
\end{figure*}

\begin{figure*}[!t]
    \centering
    \subfloat[Epoch counts comparison]{\includegraphics[width=3.0in]{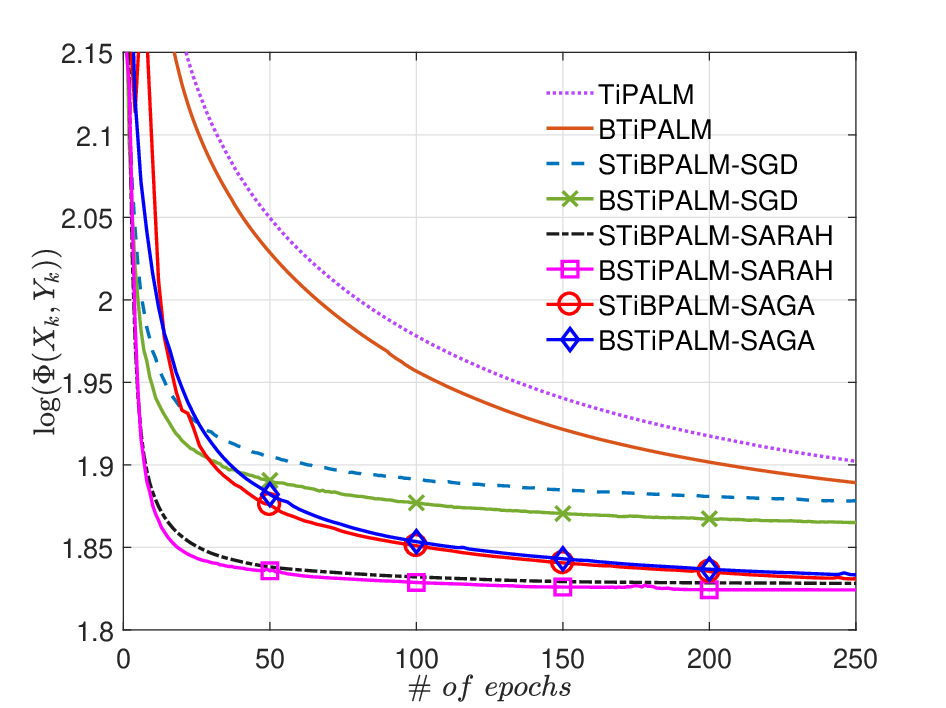}
    \label{Epoch counts comparison}}
\hfil
\subfloat[Wall-clock time comparison]{\includegraphics[width=3.0in]{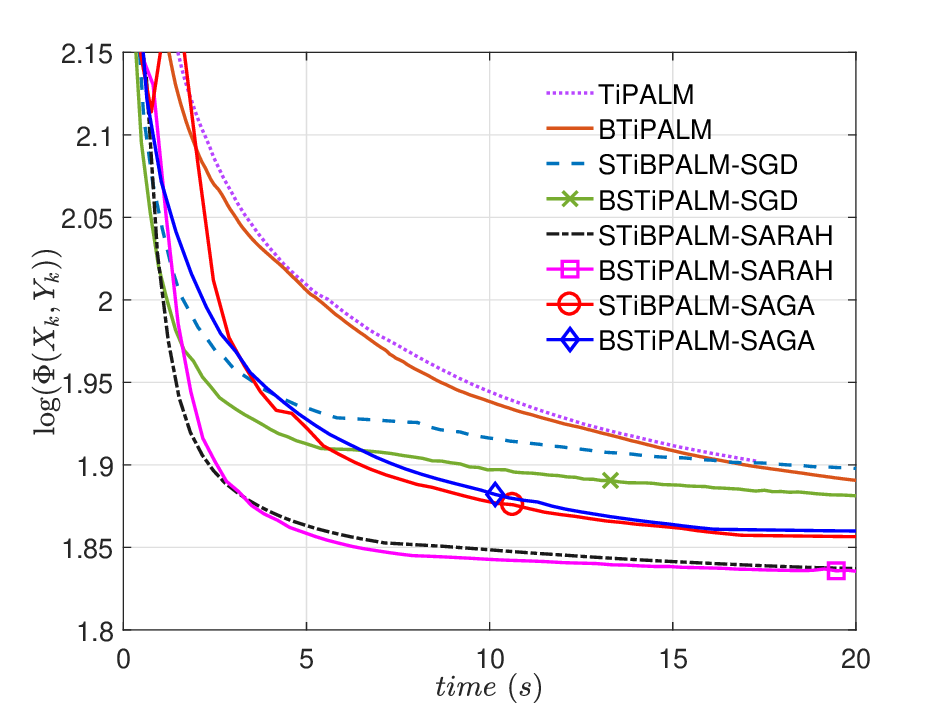}%
    \label{Wall-clock time comparison}}

    \caption{Objective decrease comparison of S-NMF with $s = 25\%$ on Yale dataset with different Brengman distance.}
    \label{fig6}
\end{figure*}

\begin{figure*}[!t]
    \centering
    \subfloat[Epoch counts comparison]{\includegraphics[width=3.0in]{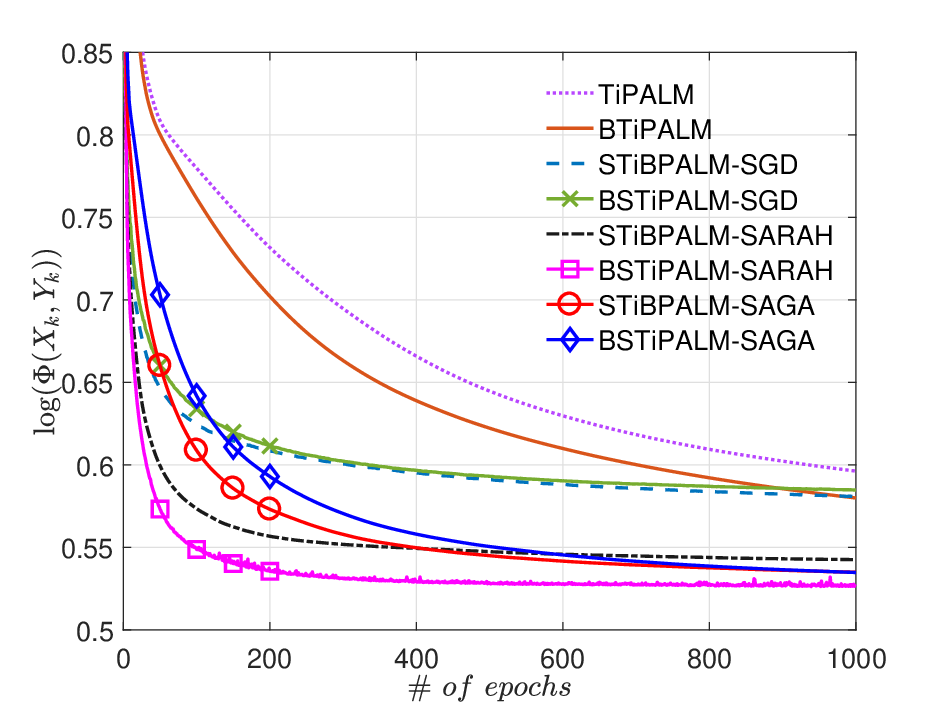}
    \label{Epoch counts comparison}}
\hfil
\subfloat[Wall-clock time comparison]{\includegraphics[width=3.0in]{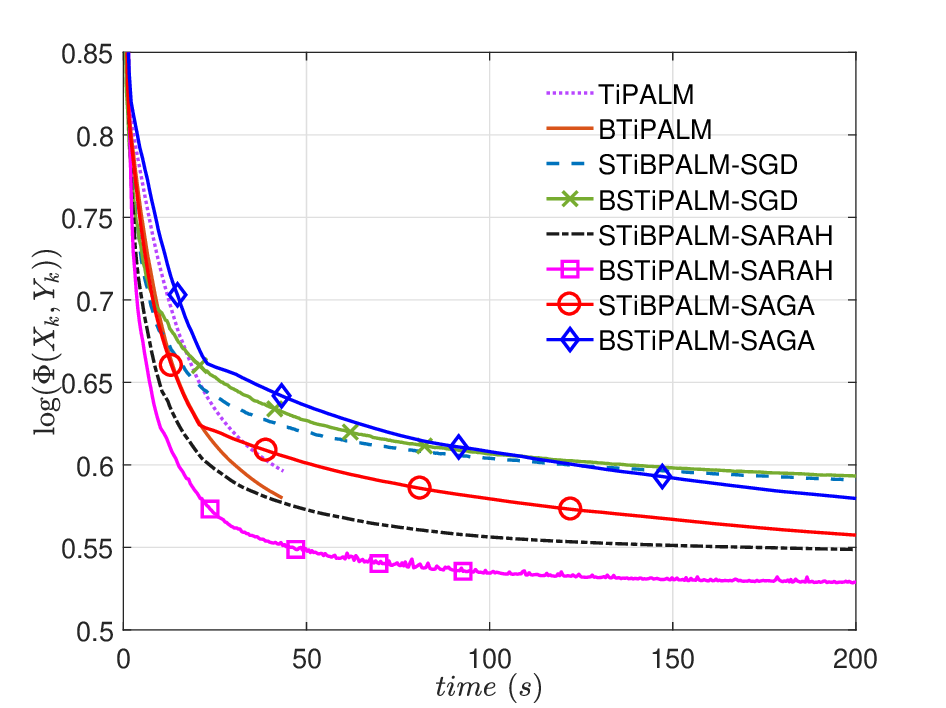}%
    \label{Wall-clock time comparison}}

    \caption{Objective decrease comparison of S-NMF with $s = 25\%$ on ORL dataset with different Brengman distance.}
    \label{fig7}
\end{figure*}

\begin{figure*}[!t]
    \centering
     \subfloat{\includegraphics[width=1.5in]{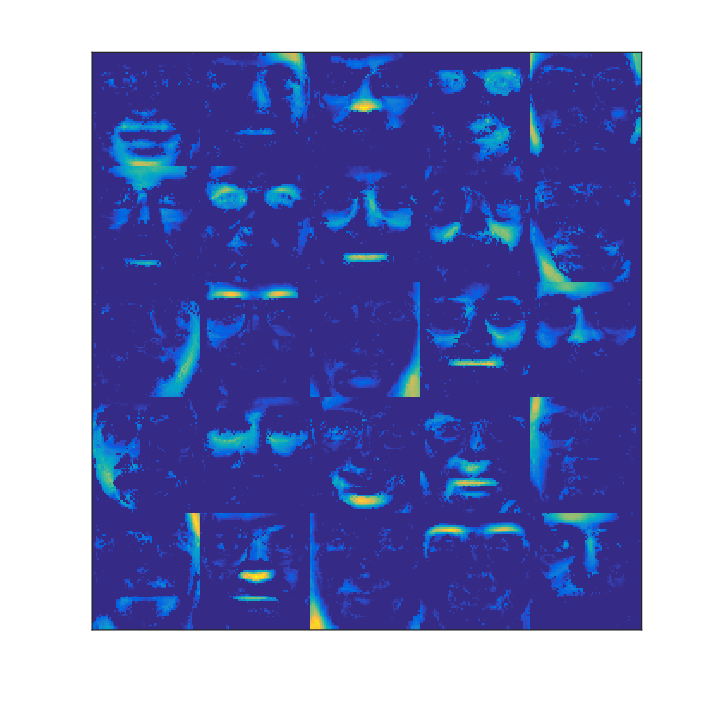}
    \label{fig_first_case}}
\hfil
\subfloat{\includegraphics[width=1.5in]{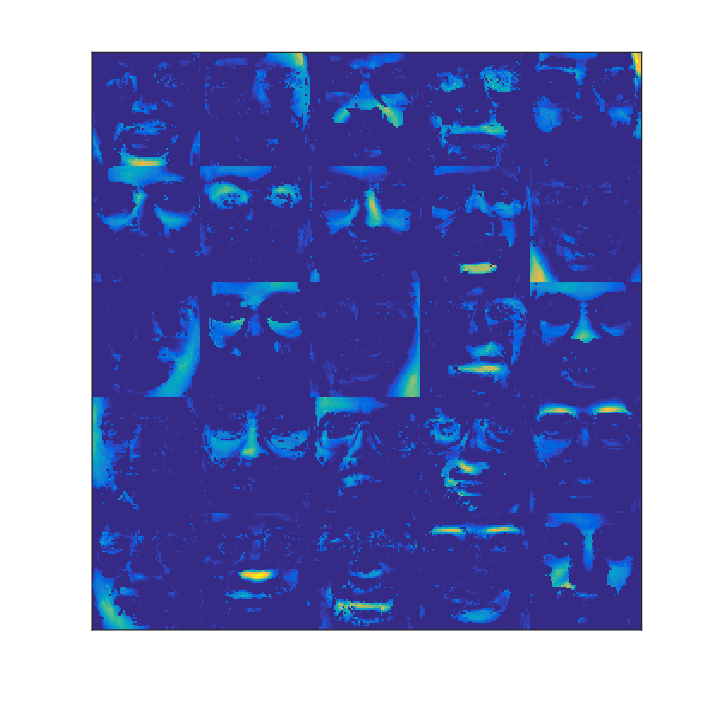}%
    \label{fig_second_case}}
\hfil
\subfloat{\includegraphics[width=1.5in]{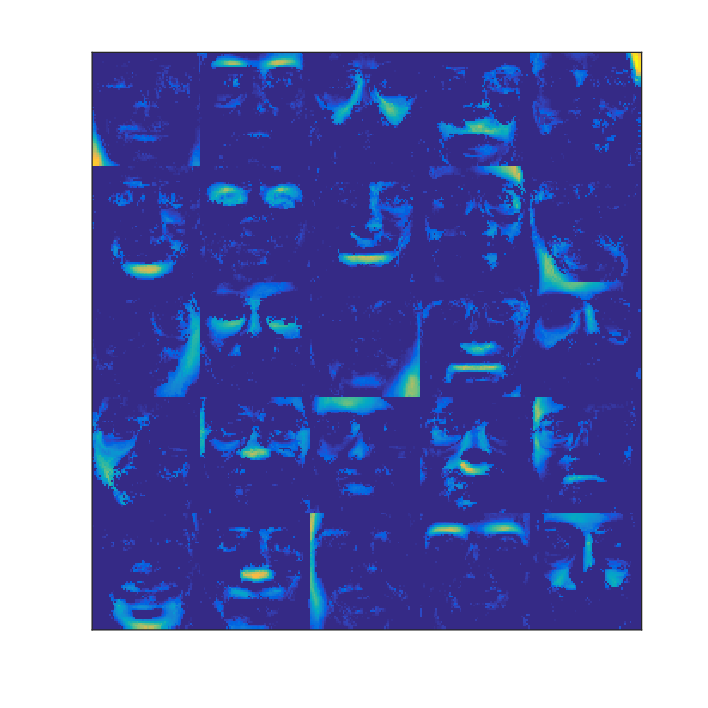}%
    \label{fig_thrid_case}}
\hfil    
    \subfloat{\includegraphics[width=1.5in]{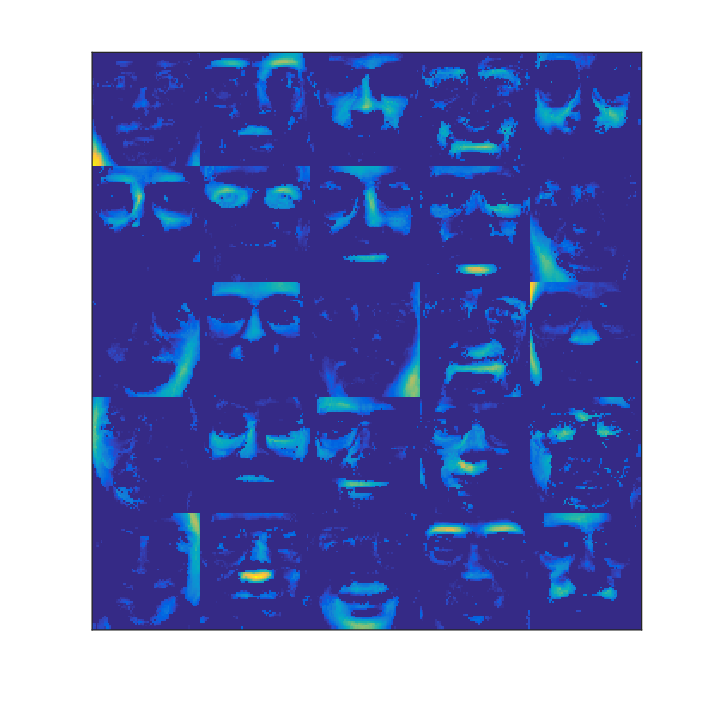}%
    \label{fig_four_case}}
\hfil
\subfloat{\includegraphics[width=1.5in]{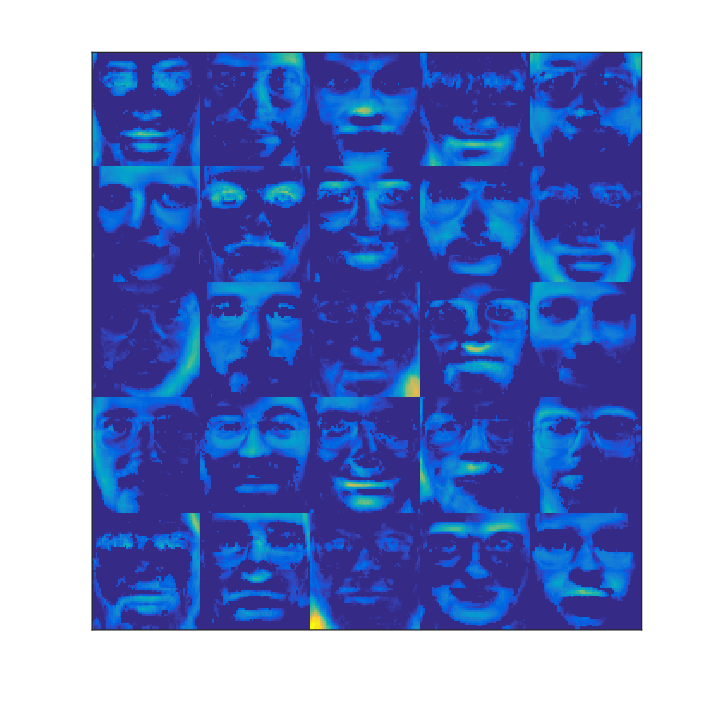}%
    \label{fig_five_case}}
\hfil
\subfloat{\includegraphics[width=1.5in]{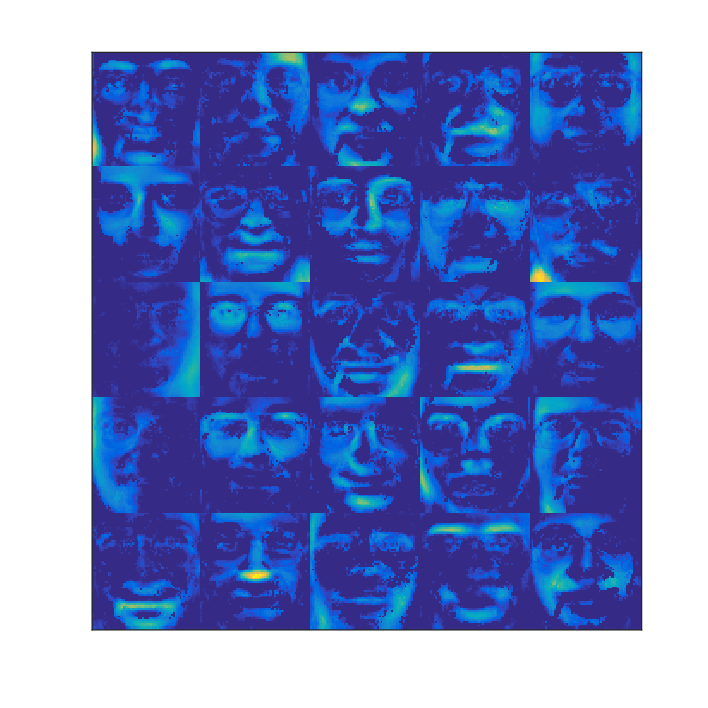}%
    \label{fig_six_case}}
    \hfil
\subfloat{\includegraphics[width=1.5in]{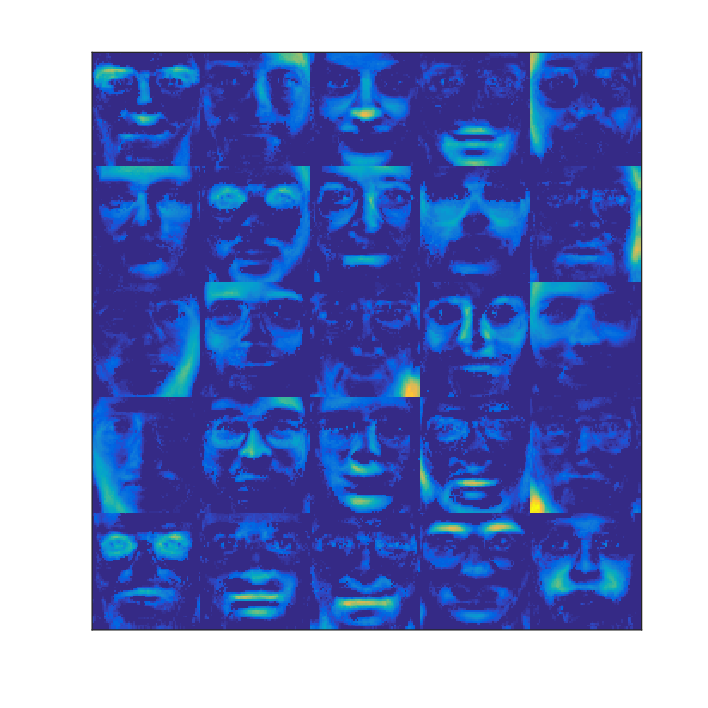}%
    \label{fig_seven_case}}
    \hfil
\subfloat{\includegraphics[width=1.5in]{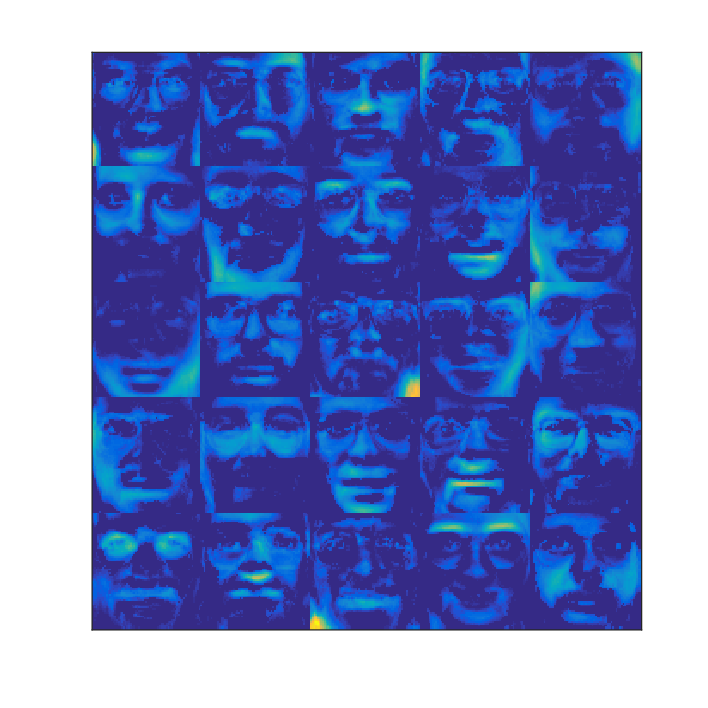}%
    \label{fig_eight_case}}   
    \caption{ The results for 25 basis faces using different sparsity settings. From left column to right column are the results of TiPALM, STiBPALM-SGD, STiBPALM-SAGA, and STiBPALM-SARAH, respectively. From top row to bottom row are the result of $s = 25\%$ and $s = 50\%$, respectively. }
    \label{fig8}
\end{figure*}
\subsection{Blind image-deblurring}\label{sect52}
\ \par Let $A$ be a blurred image, the problem of blind deconvolution is given by
\begin{equation}
\label{(5.2)}
\underset{X,Y}{\min}  \left \{\frac{1}{2} \left \| A-X\odot Y \right \| _{F}^{2}+\eta\sum_{r=1}^{2d} R([D(X)]_r) : \ 0\le X\le 1,\ 0\le Y\le 1,\ \left \| Y \right \| _1\le 1\right \}.
\end{equation}
In numerical experiment, we choose $R(v)= \log(1 + \sigma v^2)$ as in \cite{TS}, where $\sigma=10^3$ and $\eta=5\times 10^{-5}$. 
\par We consider two images, Kodim08 and Kodim15, of size $256 \times 256$ for testing. For each image, two blur kernels---linear motion blur and out-of-focus blur---are considered with additional additive Gaussian noise. In this numerical experiment, we mainly use SARAH gradient estimator, and set $p=\frac{1}{64}$. We take $\alpha_{1k}=\beta _{1k}=\gamma_{1k}=\mu_{1k}=\frac{k-1}{k+2}$, $\alpha_{2k}=\beta _{2k}=\gamma_{2k}=\mu_{2k}=\frac{k-1}{k+2}$ in TiPALM and STiBPALM and $\alpha_{1k}=\beta _{1k}=\gamma_{1k}=\mu_{1k}=\frac{k-1}{k+2}$ in iPALM.
\par The convergence comparisons of the algorithms for both images with motion blur are provided in Figure \ref{fig9} and Figure \ref{fig10}, from which we observe STiBPALM-SARAH is faster than the other methods. Figure \ref{fig11} and Figure \ref{fig12} provide comparisons of the recovered image and blur kernel. We observe superior performance of stochastic algorithms over deterministic algorithms in these figures as well.  In particular, when comparing the estimated blur kernels of the two algorithms every 20 epochs, we clearly see that STiBPALM-SARAH more quickly recovers more accurate solutions than TiPALM.
\begin{figure*}[!t]
    \centering
    \subfloat[Epoch counts comparison]{\includegraphics[width=3.0in]{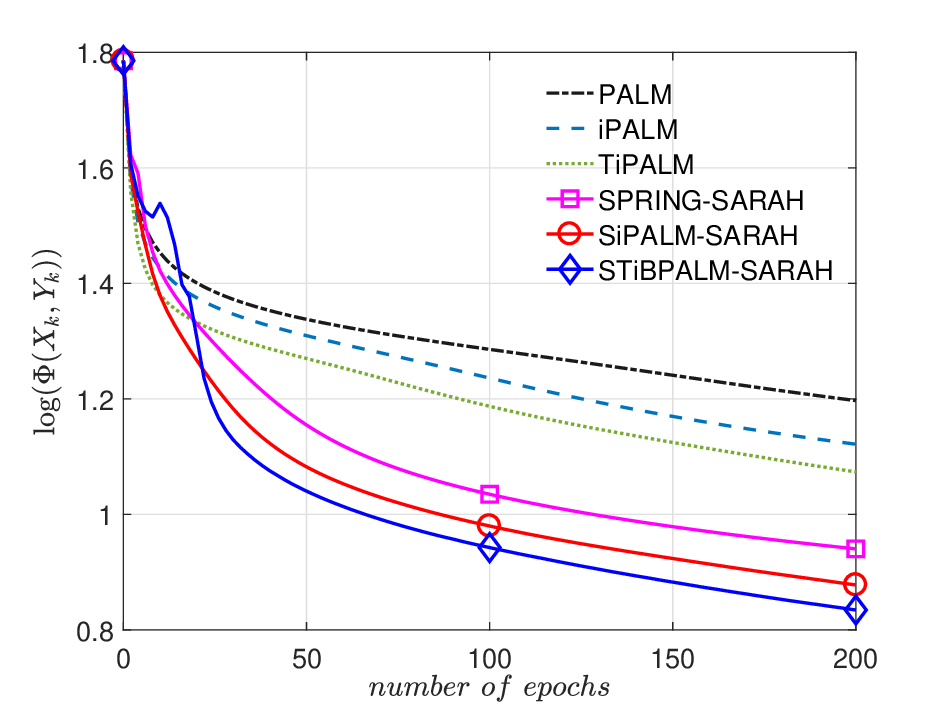}
    \label{Epoch counts comparison}}
\hfil
\subfloat[Wall-clock time comparison]{\includegraphics[width=3.0in]{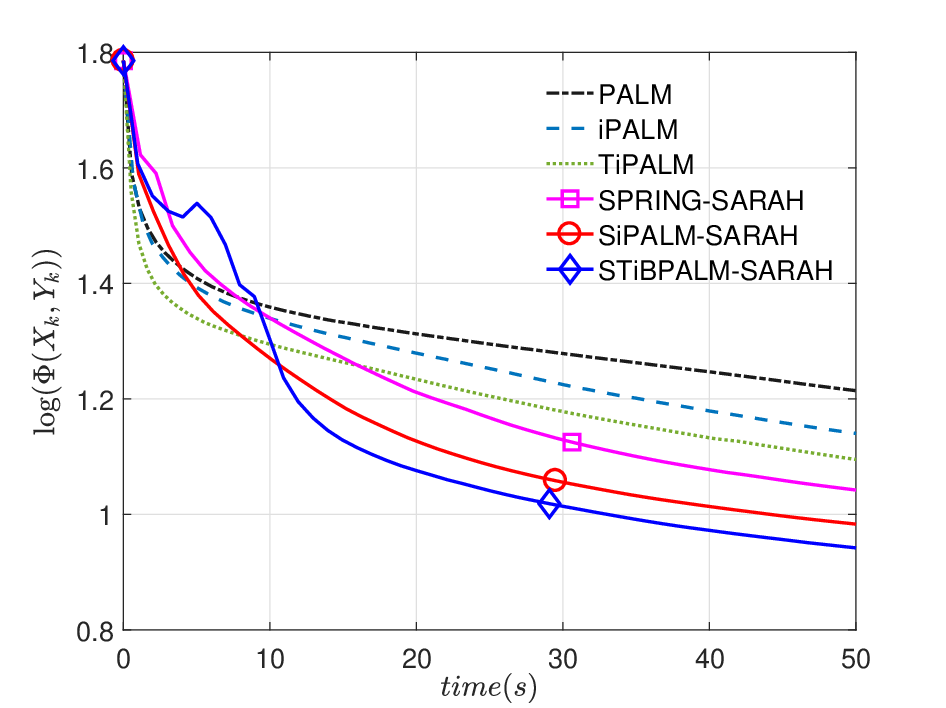}%
    \label{Wall-clock time comparison}}

    \caption{Objective decrease comparison (epoch counts) of blind image-deconvolution experiment on Kodim08 image using an $11 \times 11$ motion blur kernel.}
    \label{fig9}
\end{figure*}

\begin{figure*}[!t]
    \centering
    \subfloat[Epoch counts comparison]{\includegraphics[width=3.0in]{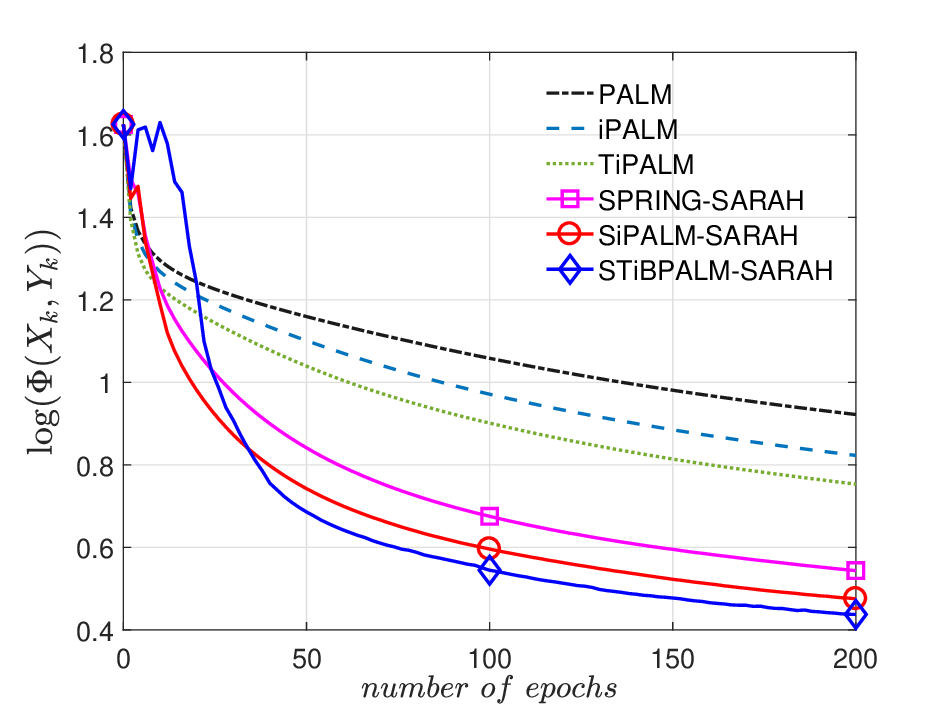}
    \label{Epoch counts comparison}}
\hfil
\subfloat[Wall-clock time comparison]{\includegraphics[width=3.0in]{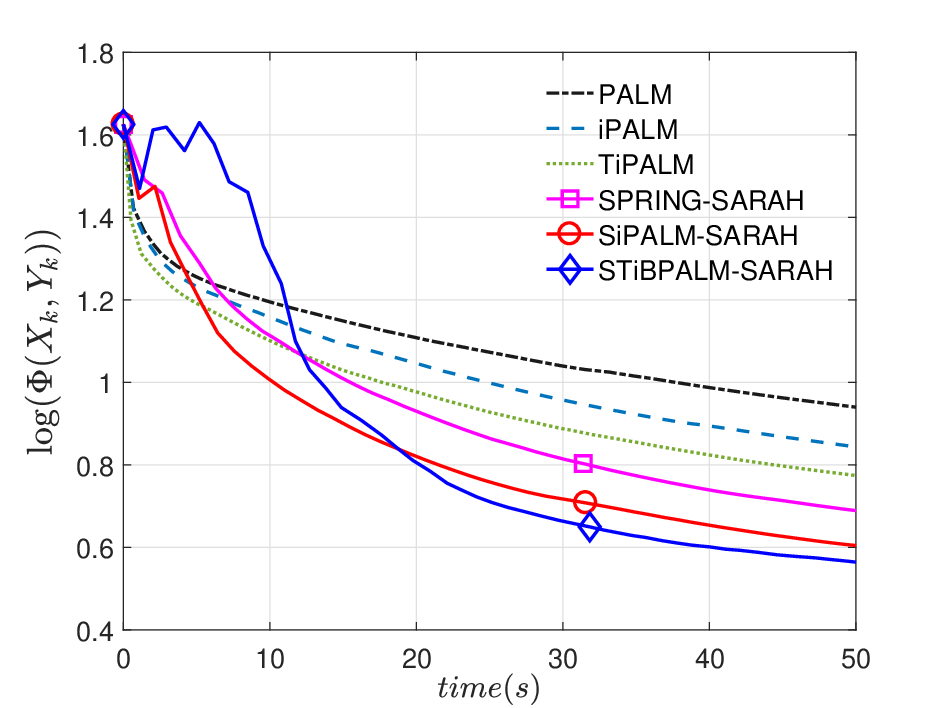}%
    \label{Wall-clock time comparison}}

    \caption{Objective decrease comparison (epoch counts) of blind image-deconvolution experiment on Kodim15 image using an $11 \times 11$ motion blur kernel.}
    \label{fig10}
\end{figure*}

\begin{figure*}[!t]
    \centering
     \subfloat[Original image and kernel]{\includegraphics[width=1.5in]{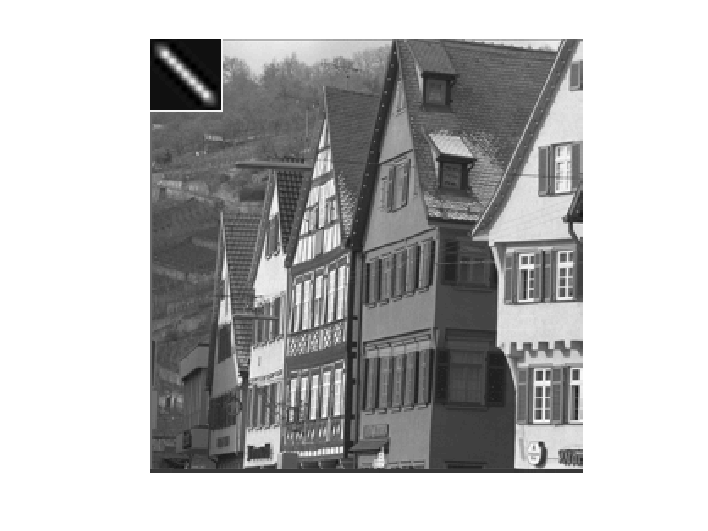}
    \label{fig_first_case}}
\hfil
\subfloat[Blurred image]{\includegraphics[width=1.5in]{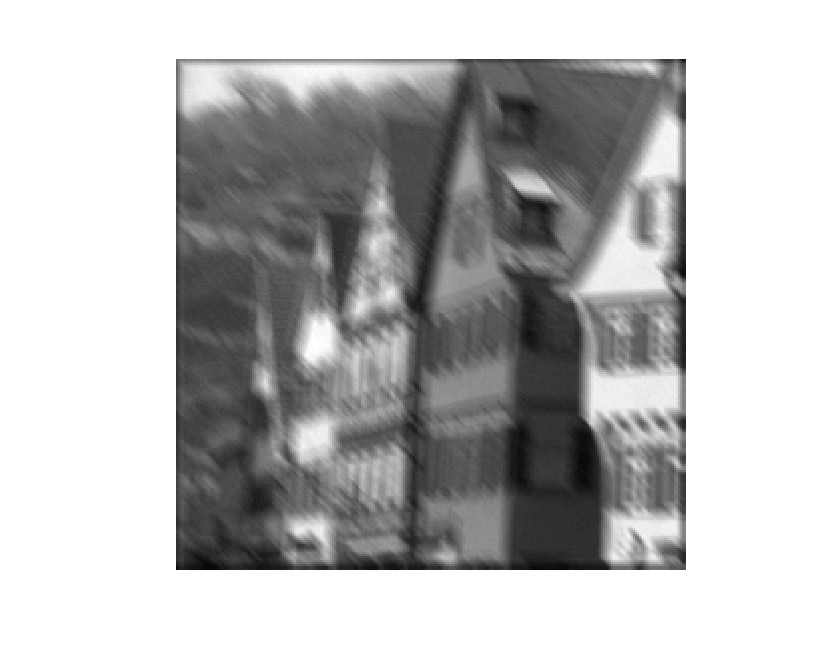}%
    \label{fig_second_case}}
\hfil
\subfloat[Recovered by PALM]{\includegraphics[width=1.5in]{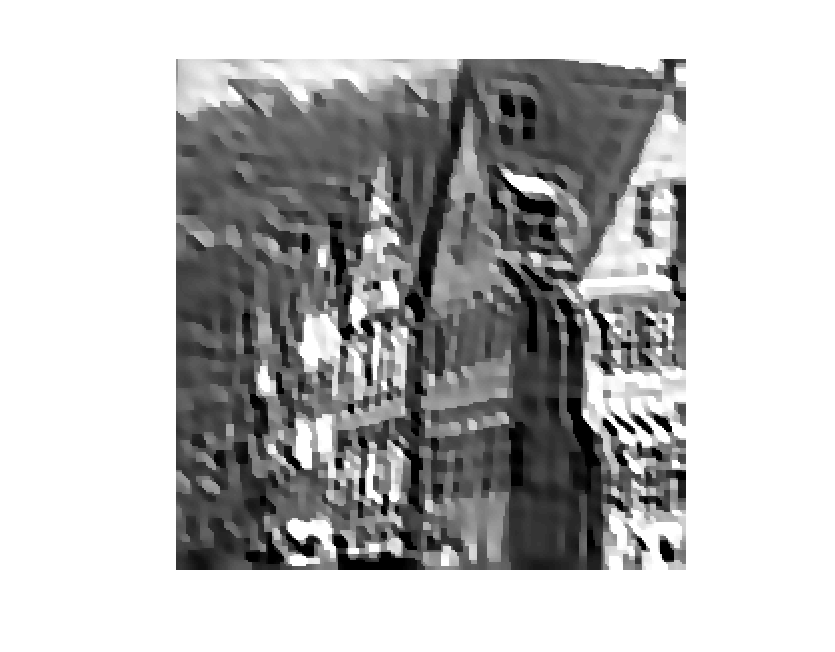}%
    \label{fig_thrid_case}}
\hfil    
    \subfloat[Recovered by iPALM]{\includegraphics[width=1.5in]{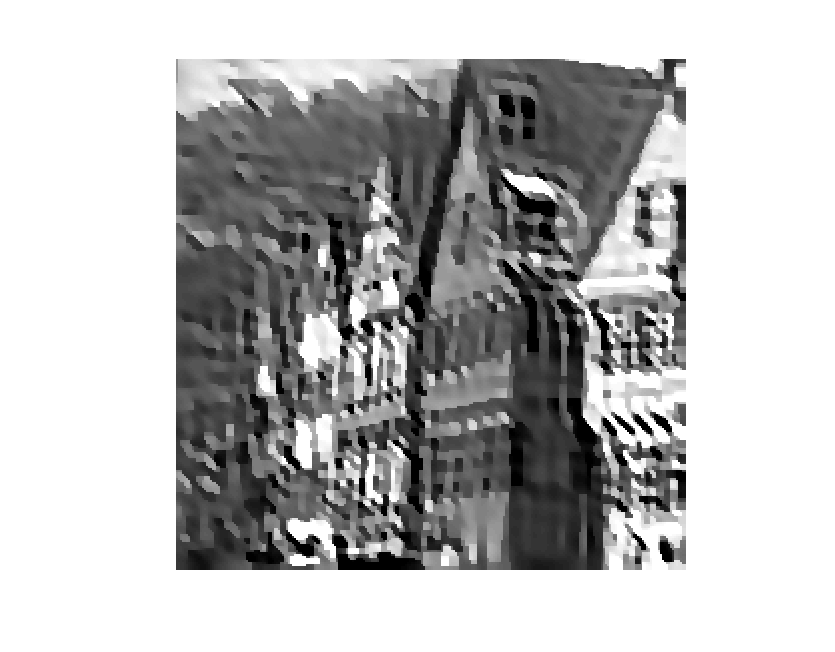}%
    \label{fig_four_case}}
\hfil
\subfloat[Recovered by TiPALM]{\includegraphics[width=1.5in]{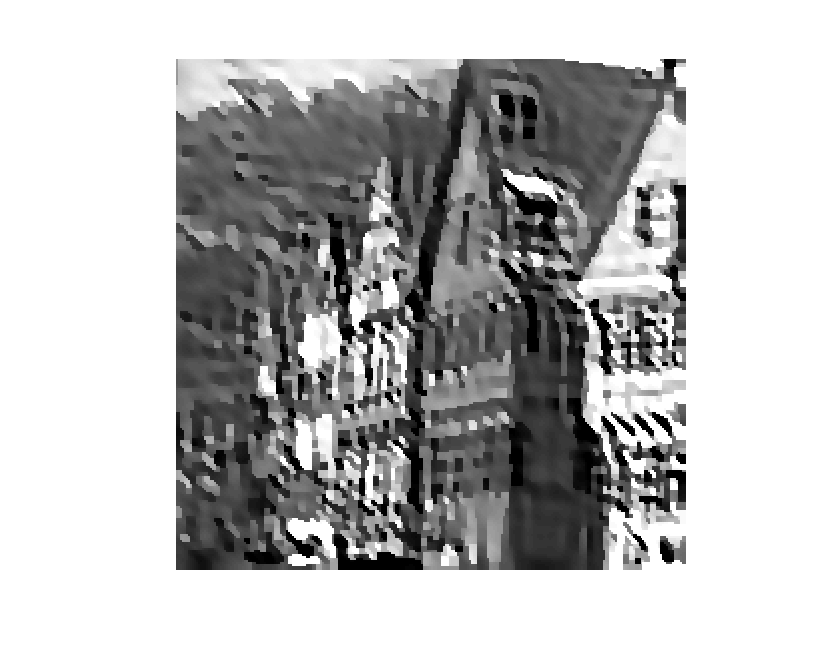}%
    \label{fig_five_case}}
\hfil
\subfloat[Recovered by SPRING-SARAH]{\includegraphics[width=1.5in]{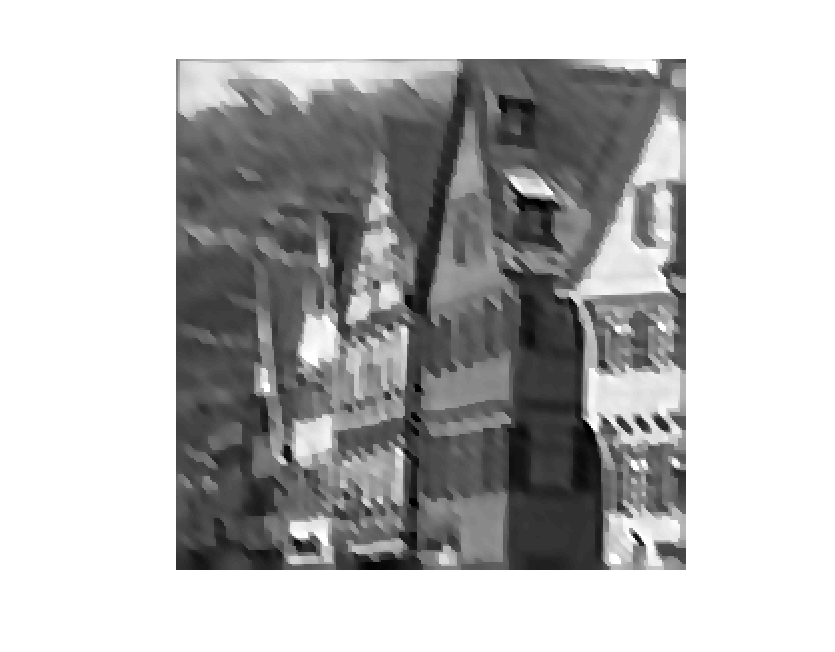}%
    \label{fig_six_case}}
    \hfil
\subfloat[Recovered by SiPALM-SARAH]{\includegraphics[width=1.5in]{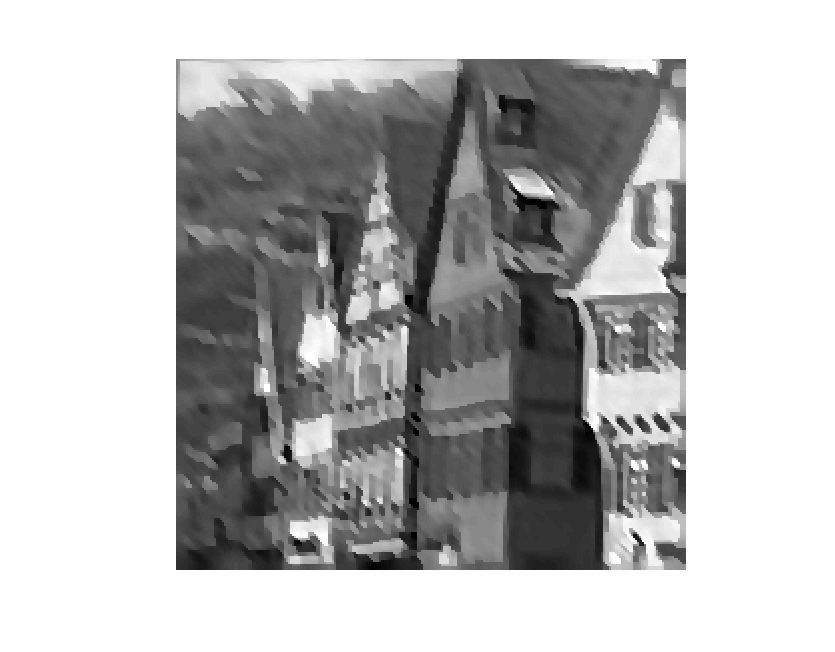}%
    \label{fig_seven_case}}
    \hfil
\subfloat[Recovered by STiB\\ PALM-SARAH]{\includegraphics[width=1.5in]{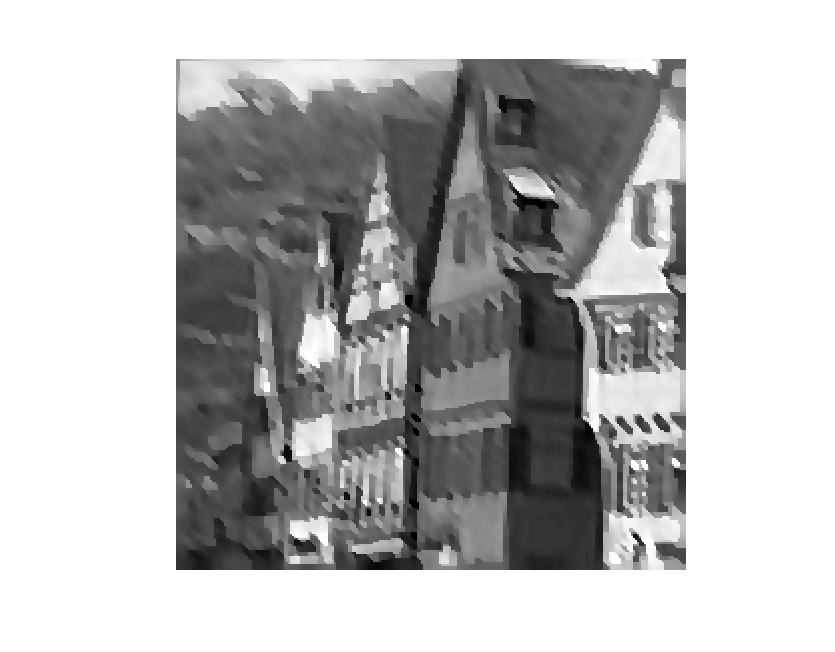}%
    \label{fig_eight_case}}  
   \hfil
\subfloat[ Estimated kernel by TiPALM]{\includegraphics[width=3in]{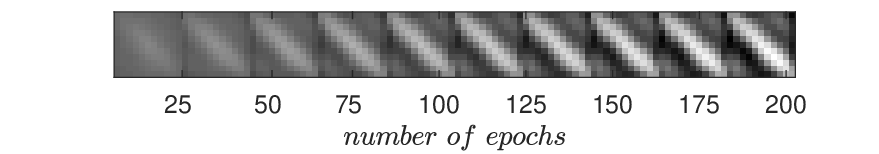}%
    \label{fig_nine_case}}     
       \hfil
\subfloat[ Estimated kernel by STiBPALM-SARAH]{\includegraphics[width=3in]{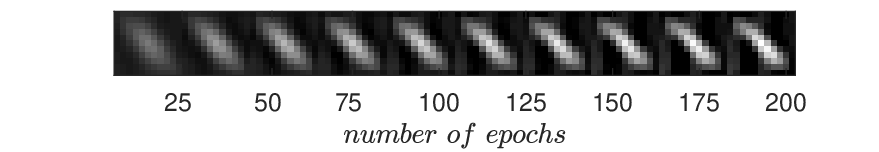}%
    \label{fig_ten_case}}  
    \caption{Image and kernel reconstructions from the blind image-deconvolution experiment on the Kodim08 image using an $11 \times 11$ motion blur kernel.}
    \label{fig11}
\end{figure*}

\begin{figure*}[!t]
    \centering
     \subfloat[Original image and kernel]{\includegraphics[width=1.5in]{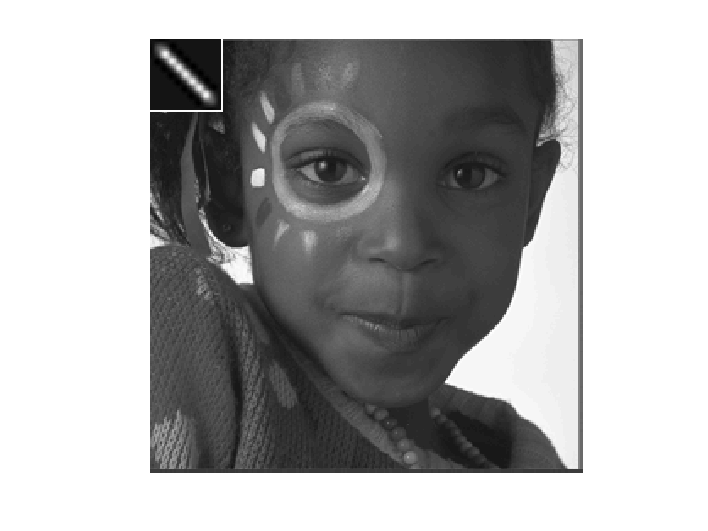}
    \label{fig_first_case}}
\hfil
\subfloat[Blurred image]{\includegraphics[width=1.5in]{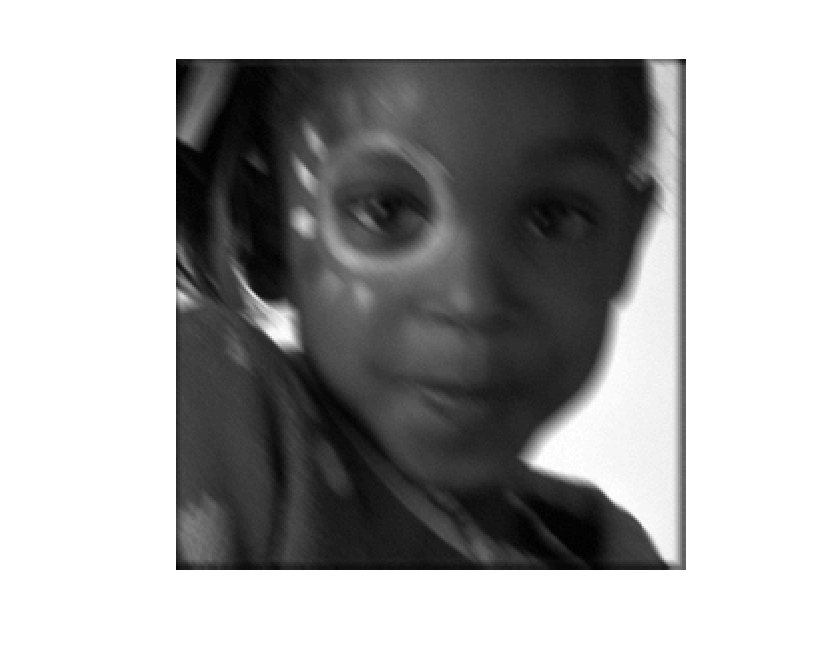}%
    \label{fig_second_case}}
\hfil
\subfloat[Recovered by PALM]{\includegraphics[width=1.5in]{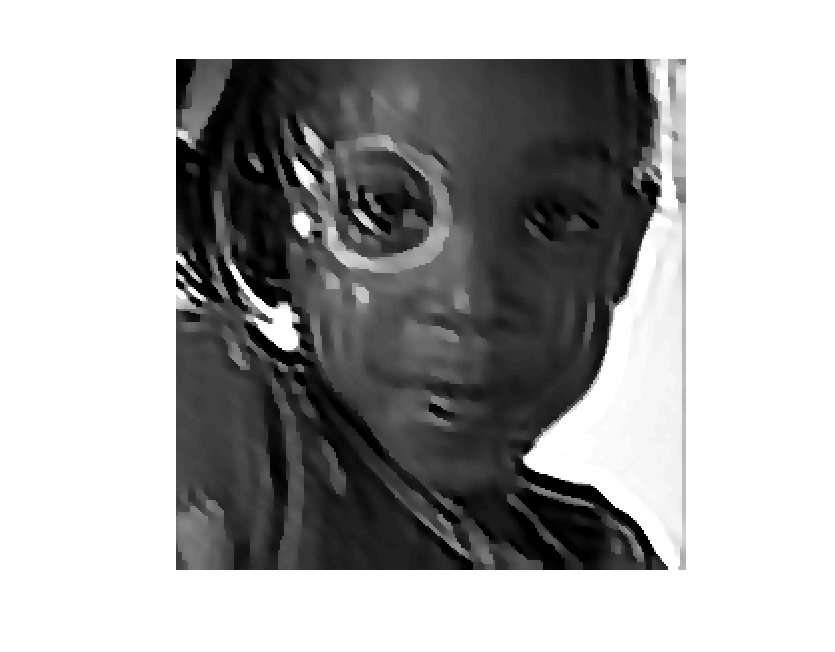}%
    \label{fig_thrid_case}}
\hfil    
    \subfloat[Recovered by iPALM]{\includegraphics[width=1.5in]{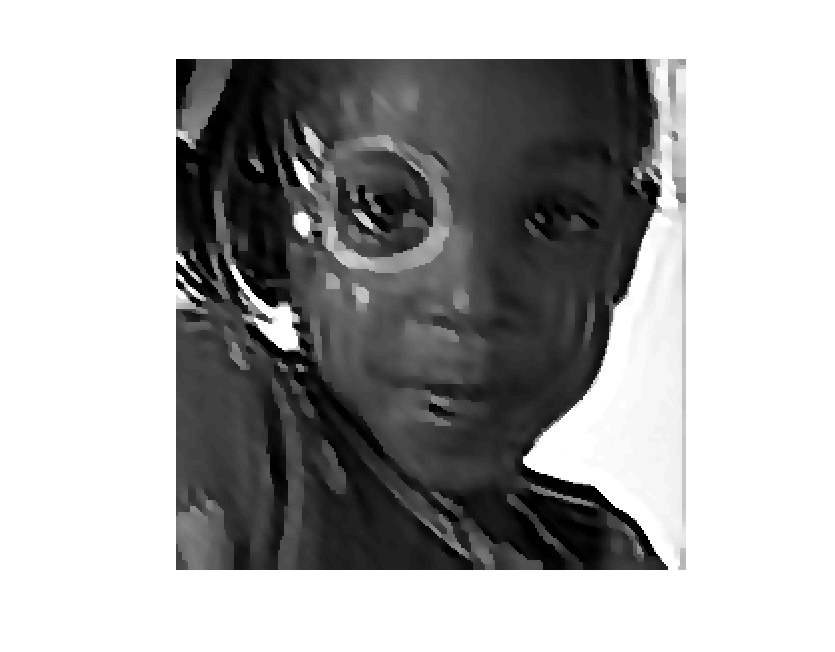}%
    \label{fig_four_case}}
\hfil
\subfloat[Recovered by TiPALM]{\includegraphics[width=1.5in]{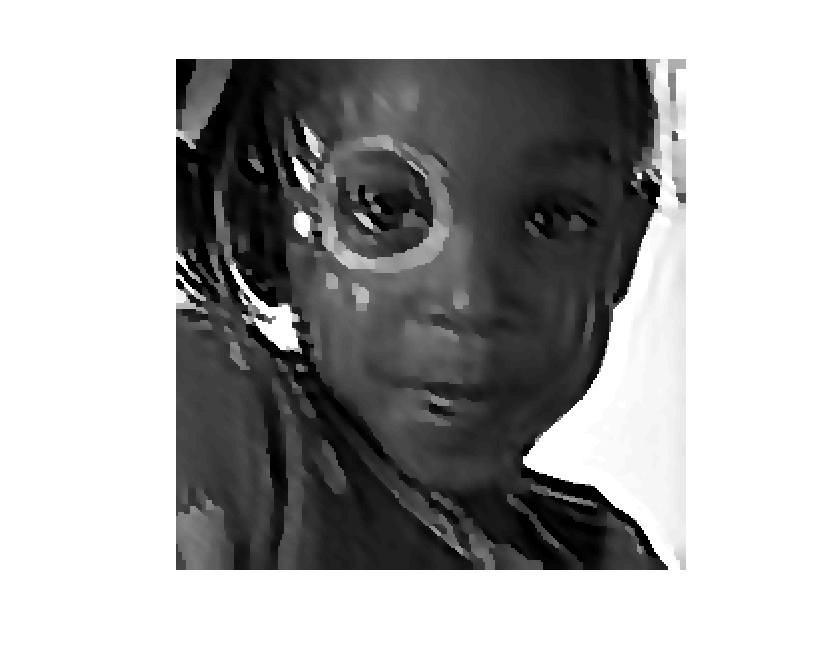}%
    \label{fig_five_case}}
\hfil
\subfloat[Recovered by SPRING-SARAH]{\includegraphics[width=1.5in]{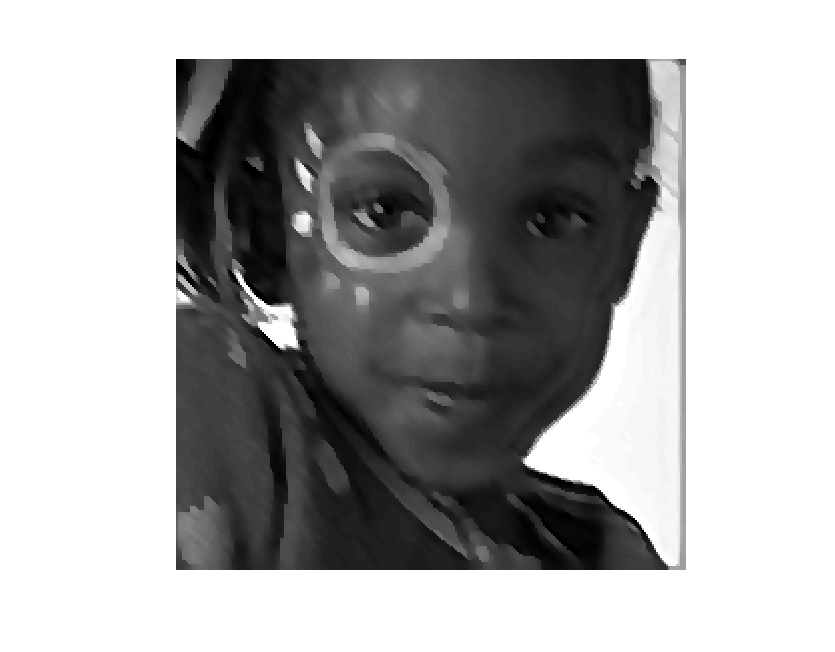}%
    \label{fig_six_case}}
    \hfil
\subfloat[Recovered by SiPALM-SARAH]{\includegraphics[width=1.5in]{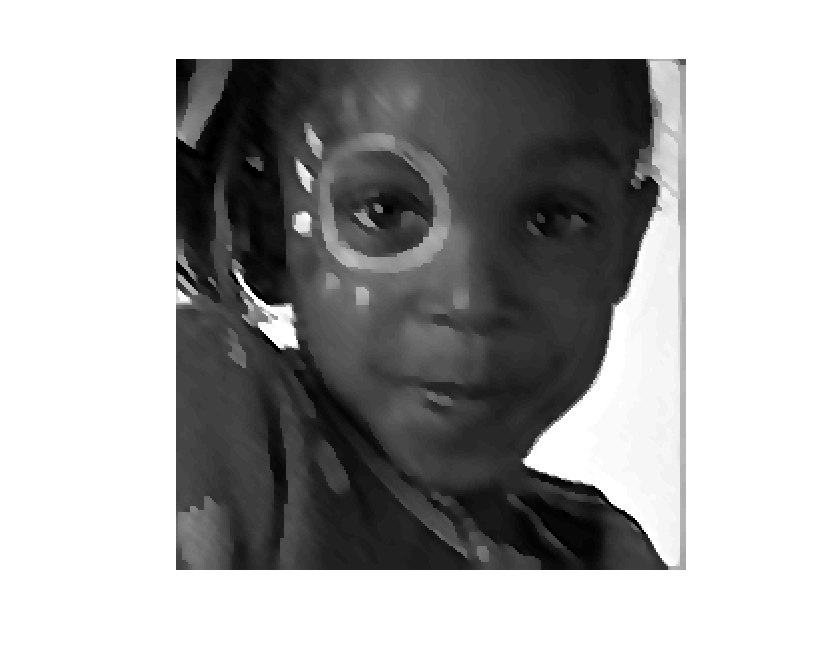}%
    \label{fig_seven_case}}
    \hfil
\subfloat[Recovered by STiB\\  PALM-SARAH]{\includegraphics[width=1.5in]{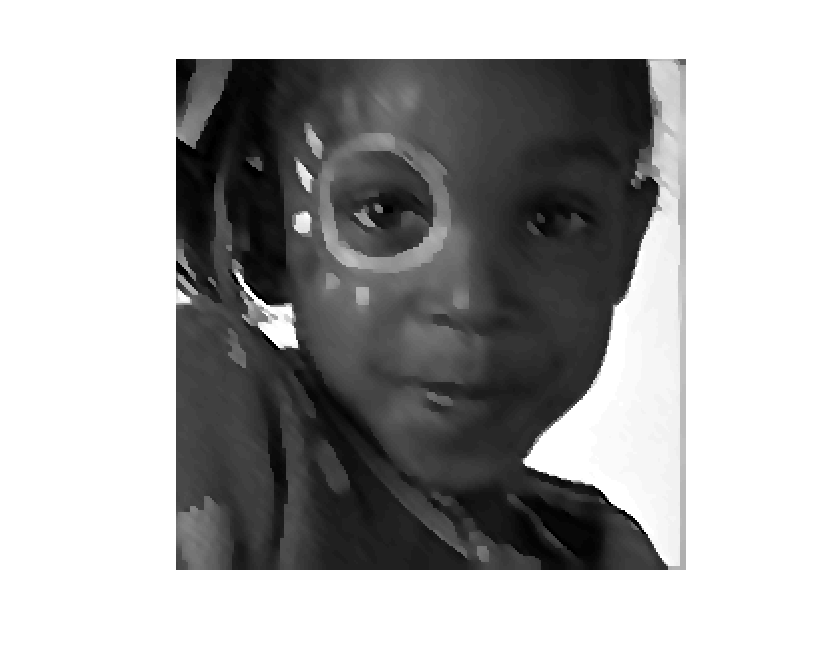}%
    \label{fig_eight_case}}  
    \hfil 
    \subfloat[ Estimated kernel by TiPALM]{\includegraphics[width=3in]{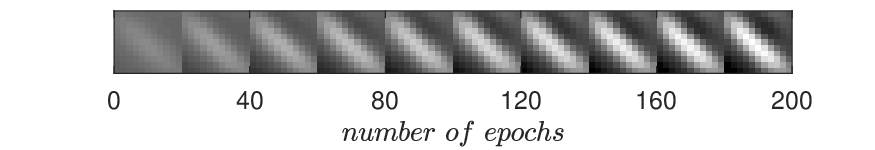}%
    \label{fig_nine_case}}     
       \hfil
\subfloat[ Estimated kernel by STiBPALM-SARAH]{\includegraphics[width=3in]{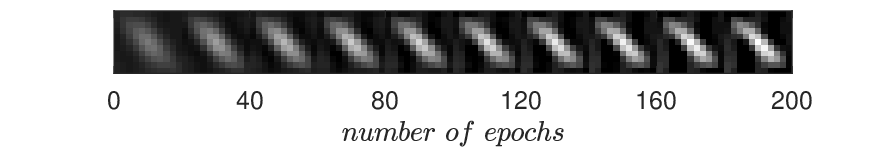}%
    \label{fig_ten_case}}  
    \caption{Image and kernel reconstructions from the blind image-deconvolution experiment on the Kodim08 image using an $11 \times 11$ motion blur kernel.}
    \label{fig12}
\end{figure*}
\section{Conclusion}
{In this paper, we propose a stochastic two-step inertial Bregman proximal alternating linearized minimization (STiBPALM) algorithm with the variance-reduced gradient estimator to solve a class of nonconvex nonsmooth optimization problems. Under some mild conditions, we analyze the convergence properties of STiBPALM when using a variety of variance-reduced gradient estimators, and prove specific convergence rates using the SAGA and SARAH estimators. We also implement the STiBPALM algorithm to sparse nonnegative matrix factorization and blind image-deblurring problems, and perform some numerical experiments to demonstrate the effectiveness of the proposed algorithm.}\\
\noindent {\bf Conflict of Interest:} The authors declared that they have no conflict of interest.
\\
\\
\\
\\
\\
\\
\\
\\
\\
\section*{Appendix}\label{sect6}
\setcounter{equation}{0}
\setcounter{lemma}{0}
\setcounter{subsection}{0}
\renewcommand\theequation{A.\arabic{equation}}
\renewcommand\thelemma{A.\arabic{lemma}}
\renewcommand\thesubsection{A}
\subsection{SAGA Variance Bound}\label{71}
We define the SAGA gradient estimators $\widetilde{\nabla}_x(u_k,y_k)$ and $\widetilde{\nabla}_y(x_{k+1},v_k)$ as follows:
\begin{equation}
\label{(A.1)}
\aligned
\widetilde{\nabla}_x(u_k,y_k)=&\frac{1}{b}\sum_{i\in I_k^x}\left (  \nabla _xH_i(u_k,y_k)- \nabla _xH_i(\varphi _{k}^{i},y_{k}) \right ) + \frac{1}{n}\sum_{j=1}^n\nabla _xH_j(\varphi _{k}^{j},y_{k}), \\
\widetilde{\nabla}_y(x_{k+1},v_k)=&\frac{1}{b}\sum_{i\in I_k^y}\left (  \nabla _yH_i(x_{k+1},v_k)- \nabla _yH_i(x_{k+1},\xi _{k}^{i}) \right ) + \frac{1}{n}\sum_{j=1}^n\nabla _yH_j(x_{k+1},\xi _{k}^{j}),
\endaligned
\end{equation}
where $I_k^x$ and $I_k^y$ are mini-batches containing $b$ indices. The variables $\varphi _{k}^{i}$ and $\xi _{k}^{i}$ follow the update rules $\varphi _{k+1}^{i}=u_k$
if $i\in I_k^x$ and $\varphi _{k+1}^{i}=\varphi _{k}^{i}$ otherwise, and $\xi _{k+1}^{i}=v_k$ if $i\in I_k^y$ and $\xi _{k+1}^{i}=\xi _{k}^{i}$ otherwise.
\par To prove our variance bounds, we require the following lemma.
\begin{lemma}
\label{lem51}
Suppose $X_1,\cdots ,X_t$ are independent random variables satisfying $\mathbb{E}_{k}X_i=0$ for $1\le i\le t$. Then
\begin{equation}
\label{(A.2)}
\mathbb{E}_{k}\left \| X_1+\cdots +X_t  \right \| ^2=\mathbb{E}_{k}\left [  \left \| X_1 \right \|^2 +\cdots +\left \| X_t  \right \|  ^2 \right ].
\end{equation}
\begin{proof}
Our hypotheses on these random variables imply $\mathbb{E}_{k}\left \langle X_i,X_j \right \rangle =0$ for $i\ne j$. Therefore,
\begin{equation*}
\mathbb{E}_{k}\left \| X_1+\cdots +X_t  \right \| ^2= \mathbb{E}_{k}\sum_{i,j=1}^{t}\left \langle X_i,X_j \right \rangle =\mathbb{E}_{k}\left [  \left \| X_1 \right \|^2 +\cdots +\left \| X_t  \right \|  ^2 \right ].
\end{equation*}
\end{proof}
\end{lemma}
We are now prepared to prove that the SAGA gradient estimator is variance-reduced.
\begin{lemma}{} 
\label{lem52}
The SAGA gradient estimator satisfies
\begin{equation}
\label{(A.3)}
\aligned
&\mathbb{E}_{k}\left \| \widetilde{\nabla}_x(u_k,y_k)-\nabla _xH(u_k,y_k) \right \| ^2\le\frac{1}{bn}\sum_{j=1}^n\left \| \nabla _xH_j(u_k,y_k)- \nabla _xH_j(\varphi _{k}^{j},y_{k}) \right \| ^2,\\
&\mathbb{E}_{k}\left \| \widetilde{\nabla}_y(x_{k+1},v_k)-\nabla _yH(x_{k+1},v_k) \right \| ^2\le\frac{4}{bn}\sum_{j=1}^n\left \| \nabla _yH_j(x_k,v_k)- \nabla _yH_j(x_{k},\xi _{k}^{j}) \right \| ^2\\
&+\frac{16N^2\gamma^2}{b}\left(\mathbb{E}_{k}\left \|z_{k+1}-z_{k}\right \| ^2+\left \|z_{k}-z_{k-1}\right \| ^2+\left\|z_{k-1}-z_{k-2}\right \| ^2\right),
\endaligned
\end{equation}
as well as
\begin{equation}
\label{(A.4)}
\aligned
&\mathbb{E}_{k}\left \| \widetilde{\nabla}_x(u_k,y_k)-\nabla _xH(u_k,y_k) \right \| \le\frac{1}{\sqrt{bn}}\sum_{j=1}^n\left \| \nabla _xH_j(u_k,y_k)- \nabla _xH_j(\varphi _{k}^{j},y_{k}) \right \|,\\
&\mathbb{E}_{k}\left \| \widetilde{\nabla}_y(x_{k+1},v_k)-\nabla _yH(x_{k+1},v_k) \right \| \le\frac{2}{\sqrt{bn}}\sum_{j=1}^n\left \| \nabla _yH_j(x_k,v_k)- \nabla _yH_j(x_{k},\xi _{k}^{j}) \right \| \\
&+\frac{4N\gamma}{\sqrt{b}}\left(\mathbb{E}_{k}\left \|z_{k+1}-z_{k}\right \|+\left \|z_{k}-z_{k-1}\right \|+\left\|z_{k-1}-z_{k-2}\right \| \right),
\endaligned
\end{equation}
where $N=\max\left \{M,L \right \} $, $\gamma=\max\left \{ \gamma_1,\gamma_2 \right \} $.
\begin{proof}
According to \eqref{(A.1)}, we have
\begin{equation}
\label{(A.5)}
\aligned
&\mathbb{E}_{k}\left \| \widetilde{\nabla}_x(u_k,y_k)-\nabla _xH(u_k,y_k) \right \| ^2\\
=&\mathbb{E}_{k}\left \|\frac{1}{b} \sum_{i\in I_k^x}\left ( \nabla _xH_i(u_k,y_k)- \nabla _xH_i(\varphi _{k}^{i},y_{k}) \right ) -\nabla _xH(u_k,y_k)+ \frac{1}{n}\sum_{j=1}^n\nabla _xH_j(\varphi _{k}^{j},y_{k})  \right \| ^2\\
\overset{(1)}{\le } &\frac{1}{b^2}\mathbb{E}_{k}\sum_{i\in I_k^x}\left \| \nabla _xH_i(u_k,y_k)- \nabla _xH_i(\varphi _{k}^{i},y_{k}) \right \| ^2\\
=&\frac{1}{bn}\sum_{j=1}^n\left \| \nabla _xH_j(u_k,y_k)- \nabla _xH_j(\varphi _{k}^{j},y_{k}) \right \| ^2.
\endaligned
\end{equation}
Inequality (1) follows from Lemma \ref{lem51}. By the Jensen's inequality, we can say that
\begin{equation}
\label{(A.6)}
\aligned
\mathbb{E}_{k}\left \| \widetilde{\nabla}_x(u_k,y_k)-\nabla _xH(u_k,y_k) \right \|
\le&\sqrt{\mathbb{E}_{k}\left \| \widetilde{\nabla}_x(u_k,y_k)-\nabla _xH(u_k,y_k) \right \| ^2}\\
\le &\frac{1}{\sqrt{bn}}\sqrt{\sum_{j=1}^n\left \| \nabla _xH_j(u_k,y_k)- \nabla _xH_j(\varphi _{k}^{j},y_{k}) \right \| ^2}\\
\le &\frac{1}{\sqrt{bn}}\sum_{j=1}^n\left \| \nabla _xH_j(u_k,y_k)- \nabla _xH_j(\varphi _{k}^{j},y_{k}) \right \| .
\endaligned
\end{equation}
We use an analogous argument for $\widetilde{\nabla}_y(x_{k+1},v_k)$.  Let $\mathbb{E}_{k,x}$ denote the expectation conditional on the first $k$ iterations and $I_k^x$. By the same reasoning as in \eqref{(A.5)}, applying the Lipschitz continuity of $\nabla _yH_j$, we obtain that
\begin{equation*}
\aligned
&\mathbb{E}_{k,x}\left \| \widetilde{\nabla}_y(x_{k+1},v_k)-\nabla _yH(x_{k+1},v_k) \right \| ^2\\
\le&\frac{1}{bn}\sum_{j=1}^n\left \| \nabla _yH_j(x_{k+1},v_k)- \nabla _yH_j(x_{k+1},\xi _{k}^{j}) \right \| ^2\\
\le&\frac{4}{bn}\sum_{j=1}^n\left \| \nabla _yH_j(x_{k+1},v_k)- \nabla _yH_j(x_{k},y_{k}) \right \| ^2+\frac{4}{bn}\sum_{j=1}^n\left \| \nabla _yH_j(x_{k},y_k)- \nabla _yH_j(x_{k},v_{k}) \right \| ^2\\
&+\frac{4}{bn}\sum_{j=1}^n\left \| \nabla _yH_j(x_{k},v_k)- \nabla _yH_j(x_{k},\xi _{k}^{j}) \right \| ^2+\frac{4}{bn}\sum_{j=1}^n\left \| \nabla _yH_j(x_{k},\xi _{k}^{j})- \nabla _yH_j(x_{k+1},\xi _{k}^{j}) \right \| ^2\\
\le&\frac{4M^2}{b}\left \|x_{k+1}-x_{k}\right \| ^2+\frac{4M^2}{b}\left \|v_{k}-y_{k}\right \| ^2+\frac{4L^2}{b}\left \|y_{k}-v_{k}\right \| ^2\\
&+\frac{4}{bn}\sum_{j=1}^n\left \| \nabla _yH_j(x_{k},v_k)- \nabla _yH_j(x_{k},\xi _{k}^{j}) \right \| ^2+\frac{4M^2}{b}\left \|x_{k+1}-x_{k}\right \| ^2\\
\le&\frac{4}{bn}\sum_{j=1}^n\left \| \nabla _yH_j(x_{k},v_k)- \nabla _yH_j(x_{k},\xi _{k}^{j}) \right \| ^2+\frac{8M^2}{b}\left \|x_{k+1}-x_{k}\right \| ^2\\
&+\frac{4(M^2+L^2)}{b}\left (2\gamma_1^2\left \|y_{k}-y_{k-1}\right \| ^2+2\gamma_2^2\left\|y_{k-1}-y_{k-2}\right \| ^2\right)\\
\endaligned
\end{equation*}
\begin{equation}
\label{(A.7)}
\aligned
\le&\frac{4}{bn}\sum_{j=1}^n\left \| \nabla _yH_j(x_{k},v_k)- \nabla _yH_j(x_{k},\xi _{k}^{j}) \right \| ^2+\frac{16N^2\gamma^2}{b}\left(\left \|z_{k+1}-z_{k}\right \| ^2+\left \|z_{k}-z_{k-1}\right \| ^2\right.\\
&\left.+\left\|z_{k-1}-z_{k-2}\right \| ^2\right),
\endaligned
\end{equation}
where $N=\max\left \{M,L \right \} $, $\gamma=\max\left \{ \gamma_1,\gamma_2 \right \} $. Also, by the same reasoning as in \eqref{(A.6)},
\begin{equation}
\label{(A.8)}
\aligned
&\mathbb{E}_{k,x}\left \| \widetilde{\nabla}_y(x_{k+1},v_k)-\nabla _yH(x_{k+1},v_k) \right \|\\
\le&\sqrt{\mathbb{E}_{k,x}\left \| \widetilde{\nabla}_y(x_{k+1},v_k)-\nabla _yH(x_{k+1},v_k) \right \| ^2}\\
\le&\frac{2}{\sqrt{bn}}\sum_{j=1}^n\left \| \nabla _yH_j(x_{k},v_k)- \nabla _yH_j(x_{k},\xi _{k}^{j}) \right \|+\frac{4N\gamma}{\sqrt{b}}\left(\left \|z_{k+1}-z_{k}\right \|+\left \|z_{k}-z_{k-1}\right \| \right.\\
&\left.+\left\|z_{k-1}-z_{k-2}\right \| \right),
\endaligned
\end{equation}
Applying the operator $\mathbb{E}_{k}$ to \eqref{(A.7)} and \eqref{(A.8)}, we get the desired result.
\end{proof}
\end{lemma}
Now define
\begin{equation}
\label{(A.9)}
\aligned
\Upsilon _{k+1}=&\frac{1}{bn}\sum_{j=1}^n \left (\left \| \nabla _xH_j(u_{k+1},y_{k+1})- \nabla _xH_j(\varphi _{k+1}^{j},y_{k+1}) \right \| ^2 \right.\\
&\left .+4\left \| \nabla _yH_j(x_{k+1},v_{k+1})- \nabla _yH_j(x_{k+1},\xi _{k+1}^{j}) \right \| ^2 \right),\\
\Gamma _{k+1}=&\frac{1}{\sqrt{bn}}\sum_{j=1}^n \left (\left \| \nabla _xH_j(u_{k+1},y_{k+1})- \nabla _xH_j(\varphi _{k+1}^{j},y_{k+1}) \right \| ^2 \right.\\
&\left .+2\left \| \nabla _yH_j(x_{k+1},v_{k+1})- \nabla _yH_j(x_{k+1},\xi _{k+1}^{j}) \right \| ^2 \right).
\endaligned
\end{equation}
By  Lemma \ref{lem52}, we have
\begin{equation*}
\aligned
&\mathbb{E}_k\left [  \left \| \widetilde{\nabla}_x(u_k,y_k)-\nabla _xH(u_k,y_k) \right \| ^2+\left \| \widetilde{\nabla}_y(x_{k+1},v_k)-\nabla _yH(x_{k+1},v_k) \right \| ^2\right ]\\
\le &\Upsilon _k+V_1\left (\mathbb{E}_k\left \| z_{k+1}-z_{k} \right \| ^{2}+\left \| z_{k}-z_{k-1} \right \| ^{2} +\left \| z_{k-1}-z_{k-2} \right \| ^{2}\right ),
\endaligned
\end{equation*}
and
\begin{equation*}
\aligned
&\mathbb{E}_k\left [  \left \| \widetilde{\nabla}_x(u_k,y_k)-\nabla _xH(u_k,y_k) \right \| +\left \| \widetilde{\nabla}_y(x_{k+1},v_k)-\nabla _yH(x_{k+1},v_k) \right \| \right ]\\
\le& \Gamma_k+V_2\left (\mathbb{E}_k\left \| z_{k+1}-z_{k} \right \|+\left \| z_{k}-z_{k-1} \right \| +\left \| z_{k-1}-z_{k-2} \right \|\right ).
\endaligned
\end{equation*}
This is exactly the MSE bound, where $V_{1}=\frac{16N^2\gamma^2}{b}$ and $V_{2}=\frac{4N\gamma}{\sqrt{b}}$.

\begin{lemma}{\rm (Geometric decay)} 
\label{lem53}
Let $\Upsilon _{k}$ be defined as in \eqref{(A.9)}, then we can establish the geometric decay property:
\begin{equation}
\label{(A.10)}
\aligned
\mathbb{E}_{k}\Upsilon _{k+1}\le\left ( 1-\rho  \right )\Upsilon _{k}+V_{\Upsilon}\left (\mathbb{E}_k\left \| z_{k+1}-z_{k} \right \| ^{2}+\left \| z_{k}-z_{k-1} \right \| ^{2} +\left \| z_{k-1}-z_{k-2} \right \| ^{2}\right ),
\endaligned
\end{equation}
where $\rho=\frac{b}{2n}$, $V_{\Upsilon}=\frac{408nN^2(1+2\gamma_1^2+\gamma_2^2)}{b^2}$.
\begin{proof}
We show that $\mathbb{E}_{k}\Upsilon _{k+1}$ is decreasing at a geometric rate. By applying the inequality $\left \| a-c \right \| ^2\le (1+\varepsilon )\left \| a-b \right \| ^2+(1+\varepsilon ^{-1} )\left \| b-c \right \| ^2$ twice, it follows that
\begin{equation}
\label{(A.11)}
\aligned
&\frac{1}{bn}\sum_{j=1}^n \mathbb{E}_{k}\left \| \nabla _xH_j(u_{k+1},y_{k+1})- \nabla _xH_j(\varphi _{k+1}^{j},y_{k+1}) \right \| ^2\\
\le&\frac{1+\varepsilon }{bn}\sum_{j=1}^n \mathbb{E}_{k}\left \| \nabla _xH_j(u_{k},y_{k})- \nabla _xH_j(\varphi _{k+1}^{j},y_{k+1}) \right \| ^2+\frac{1+\varepsilon ^{-1}}{bn}\sum_{j=1}^n\mathbb{E}_{k} \left \| \nabla _xH_j(u_{k+1},y_{k+1})- \nabla _xH_j(u_{k},y_{k})\right \| ^2\\
\le&\frac{(1+\varepsilon)^2 }{bn}\sum_{j=1}^n \mathbb{E}_{k}\left \| \nabla _xH_j(u_{k},y_{k})- \nabla _xH_j(\varphi _{k+1}^{j},y_{k}) \right \| ^2\\
&+\frac{(1+\varepsilon)(1+\varepsilon ^{-1} ) }{bn}\sum_{j=1}^n \mathbb{E}_{k}\left \|\nabla _xH_j(\varphi _{k+1}^{j},y_{k})- \nabla _xH_j(\varphi _{k+1}^{j},y_{k+1}) \right \| ^2\\
&+\frac{1+\varepsilon ^{-1}}{bn}\sum_{j=1}^n \mathbb{E}_{k}\left \| \nabla _xH_j(u_{k+1},y_{k+1})- \nabla _xH_j(u_{k},y_{k})\right \| ^2\\
\le&\frac{(1+\varepsilon)^2 (1-b/n)}{bn}\sum_{j=1}^n \left \| \nabla _xH_j(u_{k},y_{k})- \nabla _xH_j(\varphi _{k}^{j},y_{k}) \right \| ^2+\frac{(1+\varepsilon)(1+\varepsilon ^{-1} )M^2 }{b}\mathbb{E}_{k}\left \|y_{k}- y_{k+1} \right \| ^2\\
&+\frac{(1+\varepsilon ^{-1})M^2}{b}\mathbb{E}_{k}\left (\left \|u_{k+1}-u_{k}\right \| ^2+\left \|y_{k+1}-y_{k}\right \|^2\right)\\
\le&\frac{(1+\varepsilon)^2 (1-b/n)}{bn}\sum_{j=1}^n \left \| \nabla _xH_j(u_{k},y_{k})- \nabla _xH_j(\varphi _{k}^{j},y_{k}) \right \| ^2+\frac{(2+\varepsilon)(1+\varepsilon ^{-1} )M^2 }{b}\mathbb{E}_{k}\left \|y_{k+1}- y_{k} \right \| ^2\\
&+\frac{(1+\varepsilon ^{-1})M^2}{b}\mathbb{E}_{k}\left (3\left \|u_{k+1}-x_{k+1}\right \| ^2+3\left \|x_{k+1}-x_{k}\right \|^2+3\left \|x_{k}-u_{k}\right \|^2\right)\\
\le&\frac{(1+\varepsilon)^2 (1-b/n)}{bn}\sum_{j=1}^n \left \| \nabla _xH_j(u_{k},y_{k})- \nabla _xH_j(\varphi _{k}^{j},y_{k}) \right \| ^2+\frac{(2+\varepsilon)(1+\varepsilon ^{-1} )M^2 }{b}\mathbb{E}_{k}\left \|y_{k+1}- y_{k} \right \| ^2\\
&+\frac{3M^2(1+\varepsilon ^{-1})(1+2\gamma_{1}^2)}{b}\mathbb{E}_{k}\left \|x_{k+1}-x_{k}\right \| ^2+\frac{6M^2(1+\varepsilon ^{-1})(\gamma_{1}^2+\gamma_{2}^2)}{b}\left \|x_{k}-x_{k-1}\right \|^2\\
&+\frac{6M^2(1+\varepsilon ^{-1})\gamma_{2}^2}{b}\left \|x_{k-1}-x_{k-2}\right \|^2.
\endaligned
\end{equation}
Similarly,
\begin{equation}
\label{(A.12)}
\aligned
&\frac{1}{bn}\sum_{j=1}^n \mathbb{E}_{k}\left \| \nabla _yH_j(x_{k+1},v_{k+1})- \nabla _yH_j(x_{k+1},\xi _{k+1}^{j}) \right \| ^2\\
\le&\frac{1+\varepsilon }{bn}\sum_{j=1}^n \mathbb{E}_{k}\left \| \nabla _yH_j(x_{k+1},v_{k})- \nabla _yH_j(x_{k+1},\xi _{k+1}^{j}) \right \| ^2\\
&+\frac{1+\varepsilon ^{-1}}{bn}\sum_{j=1}^n \mathbb{E}_{k}\left \| \nabla _yH_j(x_{k+1},v_{k+1})- \nabla _yH_j(x_{k+1},v_{k})\right \| ^2\\
\le&\frac{(1+\varepsilon)^2 (1-b/n)}{bn}\sum_{j=1}^n \mathbb{E}_{k}\left \| \nabla _yH_j(x_{k},v_{k})- \nabla _yH_j(x_{k+1},\xi _{k}^{j}) \right \| ^2\\
&+\frac{(1+\varepsilon)(1+\varepsilon ^{-1}) (1-b/n)}{bn}\sum_{j=1}^n \mathbb{E}_{k}\left \| \nabla _yH_j(x_{k+1},v_{k})- \nabla _yH_j(x_{k},v_{k}) \right \| ^2\\
&+\frac{1+\varepsilon ^{-1}}{bn}\sum_{j=1}^n \mathbb{E}_{k}\left \| \nabla _yH_j(x_{k+1},v_{k+1})- \nabla _yH_j(x_{k+1},v_{k})\right \| ^2\\
\le&\frac{(1+\varepsilon)^3 (1-b/n)}{bn}\sum_{j=1}^n \left \| \nabla _yH_j(x_{k},v_{k})- \nabla _yH_j(x_{k},\xi _{k}^{j}) \right \| ^2\\
&+\frac{(1+\varepsilon)^2(1+\varepsilon ^{-1}) (1-b/n)}{bn}\sum_{j=1}^n \mathbb{E}_{k}\left \| \nabla _yH_j(x_{k},\xi _{k}^{j})- \nabla _yH_j(x_{k+1},\xi _{k}^{j}) \right \| ^2\\
&+\frac{(1+\varepsilon)(1+\varepsilon ^{-1}) (1-b/n)}{bn}\sum_{j=1}^n \mathbb{E}_{k}\left \| \nabla _yH_j(x_{k+1},v_{k})- \nabla _yH_j(x_{k},v_{k}) \right \| ^2\\
&+\frac{1+\varepsilon ^{-1}}{bn}\sum_{j=1}^n \mathbb{E}_{k}\left \| \nabla _yH_j(x_{k+1},v_{k+1})- \nabla _yH_j(x_{k+1},v_{k})\right \| ^2\\
\le&\frac{(1+\varepsilon)^3 (1-b/n)}{bn}\sum_{j=1}^n \left \| \nabla _yH_j(x_{k},v_{k})- \nabla _yH_j(x_{k},\xi _{k}^{j}) \right \| ^2+\frac{(1+\varepsilon)^2(1+\varepsilon ^{-1}) (1-b/n)M^2}{b}\\
&\mathbb{E}_{k}\left \| x_{k+1}-x_{k} \right \| ^2+\frac{(1+\varepsilon)(1+\varepsilon ^{-1}) (1-b/n)M^2}{b}\mathbb{E}_{k}\left \| x_{k+1}-x_{k} \right \| ^2+\frac{(1+\varepsilon ^{-1})L^2}{b}\mathbb{E}_{k}\left \| v_{k+1}-v_{k} \right \| ^2\\
\le&\frac{(1+\varepsilon)^3 (1-b/n)}{bn}\sum_{j=1}^n \left \| \nabla _yH_j(x_{k},v_{k})- \nabla _yH_j(x_{k},\xi _{k}^{j}) \right \| ^2+\frac{(2+\varepsilon)(1+\varepsilon)(1+\varepsilon ^{-1}) (1-b/n)M^2}{b}\\
&\mathbb{E}_{k}\left \| x_{k+1}-x_{k} \right \| ^2+\frac{(1+\varepsilon ^{-1})L^2}{b}\mathbb{E}_{k}\left (3\left \|v_{k+1}-y_{k+1}\right \| ^2+3\left \|y_{k+1}-y_{k}\right \|^2+3\left \|y_{k}-v_{k}\right \|^2\right)\\
\le&\frac{(1+\varepsilon)^3 (1-b/n)}{bn}\sum_{j=1}^n \left \| \nabla _yH_j(x_{k},v_{k})- \nabla _yH_j(x_{k},\xi _{k}^{j}) \right \| ^2+\frac{(2+\varepsilon)(1+\varepsilon)(1+\varepsilon ^{-1}) (1-b/n)M^2}{b}\\
&\mathbb{E}_{k}\left \| x_{k+1}-x_{k} \right \| ^2+\frac{3L^2(1+\varepsilon ^{-1})(1+2\gamma_{1}^2)}{b}\mathbb{E}_{k}\left \|y_{k+1}-y_{k}\right \| ^2+\frac{6L^2(1+\varepsilon ^{-1})(\gamma_{1}^2+\gamma_{2}^2)}{b}\\
&\left \|y_{k}-y_{k-1}\right \|^2+\frac{6L^2(1+\varepsilon ^{-1})\gamma_{2}^2}{b}\left \|y_{k-1}-y_{k-2}\right \|^2.
\endaligned
\end{equation}
With
\begin{equation*}
\aligned
\Upsilon _{k+1}=&\frac{1}{bn}\sum_{j=1}^n \left (\left \| \nabla _xH_j(u_{k+1},y_{k+1})- \nabla _xH_j(\varphi _{k+1}^{j},y_{k+1}) \right \| ^2 \right.\\
&\left .+4\left \| \nabla _yH_j(x_{k+1},v_{k+1})- \nabla _yH_j(x_{k+1},\xi _{k+1}^{j}) \right \| ^2 \right),
\endaligned
\end{equation*}
adding \eqref{(A.11)} and \eqref{(A.12)}, we can obtain
\begin{equation*}
\aligned
&\mathbb{E}_{k}\Upsilon _{k+1}\\
\le&(1+\varepsilon)^3 (1-b/n)\Upsilon _{k}+\frac{(2+\varepsilon)(1+\varepsilon ^{-1} )M^2 }{b}\mathbb{E}_{k}\left \|y_{k+1}- y_{k} \right \| ^2+\frac{3M^2(1+\varepsilon ^{-1})(1+2\gamma_{1}^2)}{b}\\
&\mathbb{E}_{k}\left \|x_{k+1}-x_{k}\right \| ^2+\frac{6M^2(1+\varepsilon ^{-1})(\gamma_{1}^2+\gamma_{2}^2)}{b}\left \|x_{k}-x_{k-1}\right \|^2+\frac{6M^2(1+\varepsilon ^{-1})\gamma_{2}^2}{b}\\
&\left \|x_{k-1}-x_{k-2}\right \|^2+\frac{4(1+\varepsilon)(1+\varepsilon ^{-1}) (1-b/n)M^2(2+\varepsilon)}{b}\mathbb{E}_{k}\left \| x_{k+1}-x_{k} \right \| ^2\\
&+\frac{12L^2(1+\varepsilon ^{-1})(1+2\gamma_{1}^2)}{b}\mathbb{E}_{k}\left \|y_{k+1}-y_{k}\right \| ^2+\frac{24L^2(1+\varepsilon ^{-1})(\gamma_{1}^2+\gamma_{2}^2)}{b}\left \|y_{k}-y_{k-1}\right \|^2\\
&+\frac{24L^2(1+\varepsilon ^{-1})\gamma_{2}^2}{b}\left \|y_{k-1}-y_{k-2}\right \|^2\\
\le&(1+\varepsilon)^3 (1-b/n)\Upsilon _{k}+\frac{13N^2(1+\varepsilon)(2+\varepsilon)(1+\varepsilon ^{-1} )(1+2\gamma_1^2)}{b}\mathbb{E}_{k}\left \|z_{k+1}- z_{k} \right \| ^2\\
&+\frac{24N^2(1+\varepsilon ^{-1})(\gamma_1^2+\gamma_{2}^2)}{b}\left \|z_{k}-z_{k-1}\right \| ^2+\frac{24N^2\gamma_{2}^2(1+\varepsilon ^{-1})}{b}\left \|z_{k-1}-z_{k-2}\right \|^2\\
\le&(1+\varepsilon)^3 (1-b/n)\Upsilon _{k}+\frac{24N^2(1+\varepsilon)(2+\varepsilon)(1+\varepsilon ^{-1})(1+2\gamma_1^2+\gamma_2^2)}{b}\left (\mathbb{E}_k\left \| z_{k+1}-z_{k} \right \| ^{2}\right.\\
&\left.+\left \| z_{k}-z_{k-1} \right \| ^{2} +\left \| z_{k-1}-z_{k-2} \right \| ^{2}\right ),
\endaligned
\end{equation*}
where $N=\max \left \{ M,L \right \}$.  Choosing $\varepsilon =\frac{b}{6n}$, we have $(1+\varepsilon )^3(1-\frac{b}{n}  ) \le 1-\frac{b}{2n}$, producing the inequality
\begin{equation}
\label{(A.13)}
\aligned
\mathbb{E}_{k}\Upsilon _{k+1}\le&(1-\frac{b}{2n})\Upsilon _{k}+\frac{24N^2(1+\frac{b}{6n})(2+\frac{b}{6n})(1+\frac{6n}{b})(1+2\gamma_1^2+\gamma_2^2)}{b}\left (\mathbb{E}_k\left \| z_{k+1}-z_{k} \right \| ^{2}\right.\\
&\left. +\left \| z_{k}-z_{k-1} \right \| ^{2}+\left \| z_{k-1}-z_{k-2} \right \| ^{2}\right )\\
\le&(1-\frac{b}{2n})\Upsilon _{k}+\frac{408nN^2(1+2\gamma_1^2+\gamma_2^2)}{b^2}\left (\mathbb{E}_k\left \| z_{k+1}-z_{k} \right \| ^{2}+\left \| z_{k}-z_{k-1} \right \| ^{2} +\left \| z_{k-1}-z_{k-2} \right \| ^{2}\right ).
\endaligned
\end{equation}
This completes the proof.
\end{proof}
\end{lemma}

\begin{lemma}{\rm (Convergence of estimator)} 
\label{lem56}
If $\left \{ z_k  \right \} _{k\in \mathbb{N} }$ satisfies $\lim_{k \to \infty } \mathbb{E}\left \| z_{k}-z_{k-1} \right \| ^{2}=0$, then $ \mathbb{E}\Upsilon _k\to 0$ and $\mathbb{E}\Gamma _k\to 0$ as $k\to \infty$.
\begin{proof}
We frist show that $\sum_{j=1}^n \mathbb{E}\left \| \nabla _xH_j(u_{k},y_{k})- \nabla _xH_j(\varphi _{k}^{j},y_{k}) \right \| ^2\to 0$ as $k\to\infty$. Indeed,
\begin{equation}
\label{(A.14)}
\aligned
&\sum_{j=1}^n \mathbb{E}\left \| \nabla _xH_j(u_{k},y_{k})- \nabla _xH_j(\varphi _{k}^{j},y_{k}) \right \| ^2\le L^2\sum_{j=1}^n \mathbb{E}\left \| u_{k}-\varphi _{k}^{j} \right \| ^2\\
\le&nL^2(1+\frac{2n}{b})\mathbb{E}\left \| u_{k}-u_{k-1}\right \| ^2+L^2(1+\frac{b}{2n})\sum_{j=1}^n \mathbb{E}\left \|u_{k-1}-\varphi _{k}^{j}\right \| ^2\\
\le&nL^2(1+\frac{2n}{b})\mathbb{E}\left \| u_{k}-u_{k-1}\right \| ^2+L^2(1+\frac{b}{2n})(1-\frac{b}{n})\sum_{j=1}^n \mathbb{E}\left \|u_{k-1}-\varphi _{k-1}^{j}\right \| ^2\\
\le&nL^2(1+\frac{2n}{b})\mathbb{E}\left \| u_{k}-u_{k-1}\right \| ^2+L^2(1-\frac{b}{2n})\sum_{j=1}^n \mathbb{E}\left \|u_{k-1}-\varphi _{k-1}^{j}\right \| ^2\\
\le&nL^2(1+\frac{2n}{b})\sum_{l=1}^k(1-\frac{b}{2n})^{k-l} \mathbb{E}\left \| u_{l}-u_{l-1}\right \| ^2.
\endaligned
\end{equation}
As $\mathbb{E}\left \| z_{k}-z_{k-1} \right \| ^{2}\to0$, so $\mathbb{E}\left \| u_{k}-u_{k-1} \right \| ^{2}\to0$, it is clear that $\sum_{l=1}^k(1-\frac{b}{2n})^{k-l} \mathbb{E}\left \| u_{l}-u_{l-1}\right \| ^2 \to 0$, and hence $\sum_{j=1}^n \mathbb{E}\left \| \nabla _xH_j(u_{k},y_{k})- \nabla _xH_j(\varphi _{k}^{j},y_{k}) \right \| ^2 \to 0$ as $k\to \infty $. An analogous argument shows that $\sum_{j=1}^n \mathbb{E}\left \| \nabla _yH_j(x_{k},v_{k})- \nabla _yH_j(x_{k},\xi _{k}^{j}) \right \| ^2\to 0$ as $k\to \infty $. So $\mathbb{E}\Upsilon _k\to 0$ as $k\to \infty $. Similarly, we can get $\mathbb{E}\Gamma _k\to 0$ as $k\to \infty $. Indeed, 
\begin{equation}
\label{(A.15)}
\aligned
&\sum_{j=1}^n \mathbb{E}\left \| \nabla _xH_j(u_{k},y_{k})- \nabla _xH_j(\varphi _{k}^{j},y_{k}) \right \|\le L\sum_{j=1}^n \mathbb{E}\left \| u_{k}-\varphi _{k}^{j} \right \|\\
\le&nL\mathbb{E}\left \| u_{k}-u_{k-1}\right \| +L\sum_{j=1}^n \mathbb{E}\left \|u_{k-1}-\varphi _{k}^{j}\right \|\\
\le&nL\mathbb{E}\left \| u_{k}-u_{k-1}\right \| +L(1-\frac{b}{n})\sum_{j=1}^n \mathbb{E}\left \|u_{k-1}-\varphi _{k-1}^{j}\right \|\\
\le&nL\sum_{l=1}^k(1-\frac{b}{n})^{k-l} \mathbb{E}\left \| u_{l}-u_{l-1}\right \|.
\endaligned
\end{equation}
Because $\mathbb{E}\left \| z_{k}-z_{k-1} \right \| ^{2}\to0$, it follows that $\mathbb{E}\left \| z_{k}-z_{k-1} \right \|\to0$ (because Jensen's inequality implies $\mathbb{E}\left \| z_{k}-z_{k-1} \right \|\le\sqrt{\mathbb{E}\left \| z_{k}-z_{k-1} \right \| ^{2}}\to 0$). So $\mathbb{E}\left \| u_{k}-u_{k-1} \right \|\to0$, then it follows that the bound on the right goes to zero as $k\to \infty $, hence $\mathbb{E}\Gamma _k\to 0$.
\end{proof}
\end{lemma}

\setcounter{equation}{0}
\setcounter{lemma}{0}
\setcounter{subsection}{0}
\renewcommand\theequation{B.\arabic{equation}}
\renewcommand\thelemma{B.\arabic{lemma}}
\renewcommand\thesubsection{B}
\subsection{SARAH Variance Bound}\label{72}
\ \par As in the previous section, we use $I_k^x$ and $I_k^y$ to denote the mini-batches used to approximate $\nabla _xH(u_k,y_k) $ and $\nabla _yH(x_{k+1},v_k)$, respectively.
\begin{lemma}{}
\label{lem54}
The SARAH gradient estimator satisfies
\begin{equation*}
\aligned
&\mathbb{E}_{k}\left (\left \| \widetilde{\nabla}_x(u_k,y_k)-\nabla _xH(u_k,y_k) \right \| ^2+\left \| \widetilde{\nabla}_y(x_{k+1},v_k)-\nabla _yH(x_{k+1},v_k) \right \| ^2\right)\\
\le&\left ( 1-\frac{1}{p}  \right )\left ( \left \| \widetilde{\nabla}_x(u_{k-1},y_{k-1})-\nabla_xH(u_{k-1},y_{k-1})\right \|^2+ \left \| \widetilde{\nabla}_y(x_{k},v_{k-1})-\nabla_yH(x_{k},v_{k-1})\right \|^2\right )\\
&+V_{1}\left (\mathbb{E}_k\left \| z_{k+1}-z_{k} \right \| ^{2}+\left \| z_{k}-z_{k-1} \right \| ^{2} +\left \| z_{k-1}-z_{k-2} \right \| ^{2}+\left \| z_{k-2}-z_{k-3} \right \| ^{2}\right ),
\endaligned
\end{equation*}
as well as
\begin{equation*}
\aligned
&\mathbb{E}_k\left (  \left \| \widetilde{\nabla}_x(u_k,y_k)-\nabla _xH(u_k,y_k) \right \| +\left \| \widetilde{\nabla}_y(x_{k+1},v_k)-\nabla _yH(x_{k+1},v_k) \right \| \right )\\
\le& \sqrt{1-\frac{1}{p} } \left (\left \| \widetilde{\nabla}_x(u_k,y_k)-\nabla _xH(u_k,y_k) \right \| +\left \| \widetilde{\nabla}_y(x_{k+1},v_k)-\nabla _yH(x_{k+1},v_k) \right \|\right )\\
&+V_2\left (\mathbb{E}_k\left \| z_{k+1}-z_{k} \right \|+\left \| z_{k}-z_{k-1} \right \| +\left \| z_{k-1}-z_{k-2} \right \|+\left \| z_{k-2}-z_{k-3} \right \|\right ),
\endaligned
\end{equation*}
where $V_{1}=6\left ( 1-\frac{1}{p}  \right )M^2(1+2\gamma_{1}^2+\gamma_{2}^2)$ and $V_{2}=M\sqrt{6(1-\frac{1}{p})(1+2\gamma_{1}^2+\gamma_{2}^2) }$.
\begin{proof}
Let $\mathbb{E}_{k,p}$ denote the expectation conditional on the first $k$ iterations and the event that we do not compute the full gradient at iteration $k$. The conditional expectation of the SARAH gradient estimator in this case is
\begin{equation}
\label{(B.2)}
\aligned
\mathbb{E}_{k,p}\widetilde{\nabla}_x(u_k,y_k)=&\frac{1}{b}\mathbb{E}_{k,p}\left ( \sum_{i\in I_k^x} \nabla _xH_i(u_k,y_k)- \nabla _xH_i(u_{k-1},y_{k-1}) \right ) +\widetilde{\nabla}_x(u_{k-1},y_{k-1}) \\
=&\nabla _xH(u_k,y_k)-\nabla _xH(u_{k-1},y_{k-1})+\widetilde{\nabla}_x(u_{k-1},y_{k-1}),
\endaligned
\end{equation}
and further
\begin{equation}
\label{(B.3)}
\aligned
&\mathbb{E}_{k,p}\left \| \widetilde{\nabla}_x(u_k,y_k) -\nabla_xH(u_k,y_k) \right \|^2 \\
=&\mathbb{E}_{k,p}\left \| \widetilde{\nabla}_x(u_{k-1},y_{k-1})-\nabla_xH(u_{k-1},y_{k-1})+\nabla_xH(u_{k-1},y_{k-1}) -\nabla_xH(u_k,y_k)\right.\\
&\left.+\widetilde{\nabla}_x(u_k,y_k)-\widetilde{\nabla}_x(u_{k-1},y_{k-1}) \right \|^2 \\
=&\left \| \widetilde{\nabla}_x(u_{k-1},y_{k-1})-\nabla_xH(u_{k-1},y_{k-1})\right \|^2+\left \| \nabla_xH(u_{k-1},y_{k-1}) -\nabla_xH(u_k,y_k) \right \|^2 \\
&+\mathbb{E}_{k,p}\left \| \widetilde{\nabla}_x(u_k,y_k)-\widetilde{\nabla}_x(u_{k-1},y_{k-1})\right \|^2 \\
&+2\left \langle \widetilde{\nabla}_x(u_{k-1},y_{k-1})-\nabla_xH(u_{k-1},y_{k-1}), \nabla_xH(u_{k-1},y_{k-1}) -\nabla_xH(u_k,y_k) \right \rangle\\
&-2\left \langle \nabla_xH(u_{k-1},y_{k-1})-\widetilde{\nabla}_x(u_{k-1},y_{k-1}), \mathbb{E}_{k,p}\left ( \widetilde{\nabla}_x(u_{k},y_{k})-\widetilde{\nabla}_x(u_{k-1},y_{k-1}) \right ) \right \rangle \\
&-2\left \langle \nabla_xH(u_k,y_k)-\nabla_xH(u_{k-1},y_{k-1}), \mathbb{E}_{k,p}\left ( \widetilde{\nabla}_x(u_{k},y_{k})-\widetilde{\nabla}_x(u_{k-1},y_{k-1}) \right ) \right \rangle.
\endaligned
\end{equation}
By \eqref{(B.2)}, we see that
$$\mathbb{E}_{k,p}\left ( \widetilde{\nabla}_x(u_{k},y_{k})-\widetilde{\nabla}_x(u_{k-1},y_{k-1}) \right )=\nabla_xH(u_{k},y_{k})-\nabla_xH(u_{k-1},y_{k-1}).$$
Thus, the first two inner products in \eqref{(B.3)} sum to zero and the third one is equal to
\begin{equation*}
\aligned
&-2\left \langle \nabla_xH(u_k,y_k)-\nabla_xH(u_{k-1},y_{k-1}), \mathbb{E}_{k,p}\left ( \widetilde{\nabla}_x(u_{k},y_{k})-\widetilde{\nabla}_x(u_{k-1},y_{k-1}) \right ) \right \rangle\\
=&-2\left \langle \nabla_xH(u_k,y_k)-\nabla_xH(u_{k-1},y_{k-1}), \nabla_xH(u_{k},y_{k})-\nabla_xH(u_{k-1},y_{k-1})\right \rangle\\
=&-2\left \| \nabla_xH(u_k,y_k)-\nabla_xH(u_{k-1},y_{k-1}) \right \| ^2.
\endaligned
\end{equation*}
This yields
\begin{equation*}
\aligned
&\mathbb{E}_{k,p}\left \| \widetilde{\nabla}_x(u_k,y_k) -\nabla_xH(u_k,y_k) \right \|^2 \\
=&\left \| \widetilde{\nabla}_x(u_{k-1},y_{k-1})-\nabla_xH(u_{k-1},y_{k-1})\right \|^2-\left \| \nabla_xH(u_{k-1},y_{k-1}) -\nabla_xH(u_k,y_k) \right \|^2 \\
&+\mathbb{E}_{k,p}\left \| \widetilde{\nabla}_x(u_k,y_k)-\widetilde{\nabla}_x(u_{k-1},y_{k-1})\right \|^2\\
\le&\left \| \widetilde{\nabla}_x(u_{k-1},y_{k-1})-\nabla_xH(u_{k-1},y_{k-1})\right \|^2+\mathbb{E}_{k,p}\left \| \widetilde{\nabla}_x(u_k,y_k)-\widetilde{\nabla}_x(u_{k-1},y_{k-1})\right \|^2.
\endaligned
\end{equation*}
We can bound the second term by computing the expectation.
\begin{equation*}
\aligned
&\mathbb{E}_{k,p}\left \| \widetilde{\nabla}_x(u_k,y_k)-\widetilde{\nabla}_x(u_{k-1},y_{k-1})\right \|^2\\
=&\mathbb{E}_{k,p}\left \| \frac{1}{b}\left ( \sum_{i\in I_k^x} \nabla _xH_i(u_k,y_k)- \nabla _xH_i(u_{k-1},y_{k-1}) \right )\right \|^2\\
\le&\frac{1}{b}\mathbb{E}_{k,p}\left [ \sum_{i\in I_k^x} \left \| \nabla _xH_i(u_k,y_k)- \nabla _xH_i(u_{k-1},y_{k-1}) \right \|^2 \right ]\\
=&\frac{1}{n} \sum_{j=1}^{n} \left \| \nabla _xH_j(u_k,y_k)- \nabla _xH_j(u_{k-1},y_{k-1}) \right \|^2.
\endaligned
\end{equation*}
The inequality is due to the convexity of the function $x\mapsto \left \| x \right \| ^2$. This results in the recursive inequality
\begin{equation*}
\aligned
&\mathbb{E}_{k,p}\left \| \widetilde{\nabla}_x(u_k,y_k) -\nabla_xH(u_k,y_k) \right \|^2 \\
\le&\left \| \widetilde{\nabla}_x(u_{k-1},y_{k-1})-\nabla_xH(u_{k-1},y_{k-1})\right \|^2+\frac{1}{n} \sum_{j=1}^{n} \left \| \nabla _xH_j(u_k,y_k)- \nabla _xH_j(u_{k-1},y_{k-1}) \right \|^2.
\endaligned
\end{equation*}
This bounds the MSE under the condition that the full gradient is not computed. When the full gradient is computed, the MSE is equal to zero, so taking the $M$-Lipschitz continuity of the gradients of the $H_j$ into account, we get
\begin{equation*}
\aligned
&\mathbb{E}_{k}\left \| \widetilde{\nabla}_x(u_k,y_k) -\nabla_xH(u_k,y_k) \right \|^2 \\
\le&\left ( 1-\frac{1}{p}  \right )\left (  \left \| \widetilde{\nabla}_x(u_{k-1},y_{k-1})-\nabla_xH(u_{k-1},y_{k-1})\right \|^2+\frac{1}{n} \sum_{j=1}^{n} \left \| \nabla _xH_j(u_k,y_k)- \nabla _xH_j(u_{k-1},y_{k-1}) \right \|^2 \right )\\
\le&\left ( 1-\frac{1}{p}  \right )\left ( \left \| \widetilde{\nabla}_x(u_{k-1},y_{k-1})-\nabla_xH(u_{k-1},y_{k-1})\right \|^2+M^2\left \| (u_k,y_k)- (u_{k-1},y_{k-1}) \right \|^2 \right ).
\endaligned
\end{equation*}
Using $(a+b+c) ^2\le 3(a^2+b^2+c^2)$, we can estimate
\begin{equation*}
\aligned
&\left \| (u_k,y_k)- (u_{k-1},y_{k-1}) \right \|^2=\left \| u_k-u_{k-1}\right \|^2+\left \| y_k-y_{k-1}\right \|^2\\
\le&3\left \| u_k-x_{k}\right \|^2+3\left \| x_k-x_{k-1}\right \|^2+3\left \| x_{k-1}-u_{k-1}\right \|^2+\left \| y_k-y_{k-1}\right \|^2\\
\le&3(1+2\gamma_{1}^2)\left \| x_k-x_{k-1}\right \|^2+6(\gamma_{1}^2+\gamma_{2}^2)\left \| x_{k-1}-x_{k-2}\right \|^2+6\gamma_{2}^2\left \| x_{k-2}-x_{k-3}\right \|^2+\left \| y_k-y_{k-1}\right \|^2.
\endaligned
\end{equation*}
Substitute into the above inequality, we can obtain
\begin{equation}
\label{(B.4)}
\aligned
&\mathbb{E}_{k}\left \| \widetilde{\nabla}_x(u_k,y_k) -\nabla_xH(u_k,y_k) \right \|^2 \\
\le&\left ( 1-\frac{1}{p}  \right )\left ( \left \| \widetilde{\nabla}_x(u_{k-1},y_{k-1})-\nabla_xH(u_{k-1},y_{k-1})\right \|^2+3M^2(1+2\gamma_{1}^2)\left \| x_k-x_{k-1}\right \|^2\right.\\
&\left .+6M^2(\gamma_{1}^2+\gamma_{2}^2)\left \| x_{k-1}-x_{k-2}\right \|^2+6M^2\gamma_{2}^2\left \| x_{k-2}-x_{k-3}\right \|^2+M^2\left \| y_k-y_{k-1}\right \|^2   \right ).
\endaligned
\end{equation}
By symmetric arguments, it holds
\begin{equation}
\label{(B.5)}
\aligned
&\mathbb{E}_{k}\left \| \widetilde{\nabla}_y(x_{k+1},v_k) -\nabla_yH(x_{k+1},v_k) \right \|^2 \\
\le&\left ( 1-\frac{1}{p}  \right )\left (  \left \| \widetilde{\nabla}_y(x_{k},v_{k-1})-\nabla_yH(x_{k},v_{k-1})\right \|^2+M^2\mathbb{E}_{k}\left \| (x_{k+1},v_k)- (x_{k},v_{k-1}) \right \|^2 \right )\\
\le&\left ( 1-\frac{1}{p}  \right )\left (  \left \| \widetilde{\nabla}_y(x_{k},v_{k-1})-\nabla_yH(x_{k},v_{k-1})\right \|^2+M^2\mathbb{E}_{k}\left \| x_{k+1}-x_{k}\right \|^2+3M^2(1+2\mu_{1k}^2)\right.\\
&\left .\left \| y_k-y_{k-1}\right \|^2+6M^2(\mu_{1,k-1}^2+\mu_{2k}^2)\left \| y_{k-1}-y_{k-2}\right \|^2+6M^2\mu_{2,k-1}^2\left \| y_{k-2}-y_{k-3}\right \|^2   \right )\\
\le&\left ( 1-\frac{1}{p}  \right )\left (  \left \| \widetilde{\nabla}_y(x_{k},v_{k-1})-\nabla_yH(x_{k},v_{k-1})\right \|^2+M^2\mathbb{E}_{k}\left \| x_{k+1}-x_{k}\right \|^2+3M^2(1+2\gamma_{1}^2)\right.\\
&\left .\left \| y_k-y_{k-1}\right \|^2+6M^2(\gamma_{1}^2+\gamma_{2}^2)\left \| y_{k-1}-y_{k-2}\right \|^2+6M^2\gamma_{2}^2\left \| y_{k-2}-y_{k-3}\right \|^2   \right ).
\endaligned
\end{equation}
Combining \eqref{(B.4)} and \eqref{(B.5)}, we can obtain
\begin{equation*}
\aligned
&\mathbb{E}_{k}\left (\left \| \widetilde{\nabla}_x(u_k,y_k)-\nabla _xH(u_k,y_k) \right \| ^2+\left \| \widetilde{\nabla}_y(x_{k+1},v_k)-\nabla _yH(x_{k+1},v_k) \right \| ^2\right)\\
\le&\left ( 1-\frac{1}{p}  \right )\left ( \left \| \widetilde{\nabla}_x(u_{k-1},y_{k-1})-\nabla_xH(u_{k-1},y_{k-1})\right \|^2+ \left \| \widetilde{\nabla}_y(x_{k},v_{k-1})-\nabla_yH(x_{k},v_{k-1})\right \|^2\right.\\
&\left .+M^2\mathbb{E}_{k}\left \| x_{k+1}-x_{k}\right \|^2+M^2\left \| y_{k}-y_{k-1}\right \|^2+3M^2(1+2\gamma_{1}^2)\left \| z_k-z_{k-1}\right \|^2\right.\\
&\left .+6M^2(\gamma_{1}^2+\gamma_{2}^2)\left \| z_{k-1}-z_{k-2}\right \|^2+6M^2\gamma_{2}^2\left \| z_{k-2}-z_{k-3}\right \|^2   \right )\\
\le&\left ( 1-\frac{1}{p}  \right )\Upsilon _{k}+6\left ( 1-\frac{1}{p}  \right )M^2(1+2\gamma_{1}^2+\gamma_{2}^2)\left (\mathbb{E}_k\left \| z_{k+1}-z_{k} \right \| ^{2}+\left \| z_{k}-z_{k-1} \right \| ^{2}\right.\\
&\left . +\left \| z_{k-1}-z_{k-2} \right \| ^{2}+\left \| z_{k-2}-z_{k-3} \right \| ^{2}\right ).
\endaligned
\end{equation*} 
Similar bounds hold for $\Gamma _k$ due to Jensen’s inequality:
\begin{equation*}
\aligned
&\mathbb{E}_k\left (  \left \| \widetilde{\nabla}_x(u_k,y_k)-\nabla _xH(u_k,y_k) \right \| +\left \| \widetilde{\nabla}_y(x_{k+1},v_k)-\nabla _yH(x_{k+1},v_k) \right \| \right )\\
\le& \sqrt{1-\frac{1}{p} } \left (\left \| \widetilde{\nabla}_x(u_k,y_k)-\nabla _xH(u_k,y_k) \right \| +\left \| \widetilde{\nabla}_y(x_{k+1},v_k)-\nabla _yH(x_{k+1},v_k) \right \|\right )\\
&+M\sqrt{6(1-\frac{1}{p})(1+2\gamma_{1}^2+\gamma_{2}^2) }\left (\mathbb{E}_k\left \| z_{k+1}-z_{k} \right \|+\left \| z_{k}-z_{k-1} \right \| +\left \| z_{k-1}-z_{k-2} \right \|+\left \| z_{k-2}-z_{k-3} \right \|\right ).
\endaligned
\end{equation*}
This completes the proof.
\end{proof}
\end{lemma}
Now define
\begin{equation}
\label{(B.1)}
\aligned
\Upsilon _{k+1}=& \left \| \widetilde{\nabla}_x(u_k,y_k)-\nabla _xH(u_k,y_k) \right \| ^2+\left \| \widetilde{\nabla}_y(x_{k+1},v_k)-\nabla _yH(x_{k+1},v_k) \right \| ^2,\\
\Gamma _{k+1}=& \left \| \widetilde{\nabla}_x(u_k,y_k)-\nabla _xH(u_k,y_k) \right \| +\left \| \widetilde{\nabla}_y(x_{k+1},v_k)-\nabla _yH(x_{k+1},v_k) \right \|.
\endaligned
\end{equation}
By  Lemma \ref{lem54}, we have
\begin{equation*}
\aligned
&\mathbb{E}_k\left [  \left \| \widetilde{\nabla}_x(u_k,y_k)-\nabla _xH(u_k,y_k) \right \| ^2+\left \| \widetilde{\nabla}_y(x_{k+1},v_k)-\nabla _yH(x_{k+1},v_k) \right \| ^2\right ]\\
\le &\Upsilon _k+V_1\left (\mathbb{E}_k\left \| z_{k+1}-z_{k} \right \| ^{2}+\left \| z_{k}-z_{k-1} \right \| ^{2} +\left \| z_{k-1}-z_{k-2} \right \| ^{2}+\left \| z_{k-2}-z_{k-3} \right \| ^{2}\right ),
\endaligned
\end{equation*}
and
\begin{equation*}
\aligned
&\mathbb{E}_k\left [  \left \| \widetilde{\nabla}_x(u_k,y_k)-\nabla _xH(u_k,y_k) \right \| +\left \| \widetilde{\nabla}_y(x_{k+1},v_k)-\nabla _yH(x_{k+1},v_k) \right \| \right ]\\
\le& \Gamma_k+V_2\left (\mathbb{E}_k\left \| z_{k+1}-z_{k} \right \|+\left \| z_{k}-z_{k-1} \right \| +\left \| z_{k-1}-z_{k-2} \right \|+\left \| z_{k-2}-z_{k-3} \right \|\right ).
\endaligned
\end{equation*}
This is exactly the MSE bound, where $V_{1}=6\left ( 1-\frac{1}{p}  \right )M^2(1+2\gamma_{1}^2+\gamma_{2}^2)$ and \\
$V_{2}=M\sqrt{6(1-\frac{1}{p})(1+2\gamma_{1}^2+\gamma_{2}^2) }$. 
\begin{lemma}{\rm (Geometric decay)} 
\label{lem55}
Let $\Upsilon _{k}$ be defined as in \eqref{(B.1)}, then we can establish the geometric decay property:
\begin{equation}
\label{(B.6)}
\aligned
\mathbb{E}_{k}\Upsilon _{k+1}\le\left ( 1-\rho  \right )\Upsilon _{k}+V_{\Upsilon}\left (\mathbb{E}_k\left \| z_{k+1}-z_{k} \right \| ^{2}+\left \| z_{k}-z_{k-1} \right \| ^{2} +\left \| z_{k-1}-z_{k-2} \right \| ^{2}+\left \| z_{k-2}-z_{k-3} \right \| ^{2}\right ),
\endaligned
\end{equation}
where $\rho= \frac{1}{p}$, $V_{\Upsilon}=6\left ( 1-\frac{1}{p}  \right )M^2(1+2\gamma_{1}^2+\gamma_{2}^2)$.
\begin{proof}
This is a direct result of Lemma \ref{lem54}.
\end{proof}
\end{lemma}

\begin{lemma}{\rm (Convergence of estimator)} 
\label{lem56}
If $\left \{ z_k  \right \} _{k\in \mathbb{N} }$ satisfies $\lim_{k \to \infty } \mathbb{E}\left \| z_{k}-z_{k-1} \right \| ^{2}=0$, then $ \mathbb{E}\Upsilon _k\to 0$ and $\mathbb{E}\Gamma _k\to 0$ as $k \to \infty$.
\begin{proof}
By \eqref{(B.6)}, we have
\begin{equation*}
\aligned
&\mathbb{E}\Upsilon _{k}\\
\le&\left ( 1-\rho  \right )\mathbb{E}\Upsilon _{k-1}+V_{\Upsilon}\mathbb{E}\left (\left \| z_{k}-z_{k-1} \right \| ^{2} +\left \| z_{k-1}-z_{k-2} \right \| ^{2}+\left \| z_{k-2}-z_{k-3} \right \| ^{2}+\left \| z_{k-3}-z_{k-4} \right \| ^{2}\right )\\
\le&V_{\Upsilon}\sum_{l=1}^{k}\left ( 1-\rho  \right )^{k-l}\mathbb{E}\left (\left \| z_{l}-z_{l-1} \right \| ^{2} +\left \| z_{l-1}-z_{l-2} \right \| ^{2}+\left \| z_{l-2}-z_{l-3} \right \| ^{2}+\left \| z_{l-3}-z_{l-4} \right \| ^{2}\right ), 
\endaligned
\end{equation*}
which implies $ \mathbb{E}\Upsilon _k\to 0$ as $k \to \infty$. By Jensen’s inequality, we have $\mathbb{E}\Gamma _k\to 0$ as $k \to \infty$.
\end{proof}
\end{lemma}

\end{document}